\documentclass{amsart}

\usepackage{hyperref}
\usepackage[english]{babel}
\usepackage[utf8]{inputenc}
\usepackage[babel]{csquotes}
\usepackage{amsmath}
\usepackage{amssymb}
\usepackage{amsthm} 
\usepackage{mathtools}
\usepackage{mathrsfs, stmaryrd, tikz-cd}
\usepackage{enumitem}
\usepackage[margin=3cm]{geometry}
\usepackage{textcomp}
\usepackage{color, graphicx}
\usepackage{pifont}

\usepackage{fullpage}
\numberwithin{equation}{subsection}

\title[Symplectic resolutions for multiplicative quiver varieties ]{Symplectic resolutions for multiplicative quiver varieties 
and character varieties for punctured surfaces}

\author{Travis Schedler}
\address[Travis Schedler]{Department of Mathematics, Imperial College, London, 180 Queen’s Gate, London SW7 2AZ, UK}
\email{t.schedler@imperial.ac.uk}
\author{Andrea Tirelli}
\address[Andrea Tirelli]{Department of Mathematics, Imperial College, London, 180 Queen’s Gate, London SW7 2AZ, UK and SISSA, via Bonomea 265, 34136, Trieste, IT}
\email{a.tirelli15@imperial.ac.uk}
\date{}

\newcommand{\mc}[1]{\mathcal{#1}}
\newcommand{\mb}[1]{\mathbb{#1}}

\newcommand{\mr}[1]{\mathrm{#1}}
\newcommand{\tit}[1]{\textit{#1}}

\newcommand{\vdim}{\underline{\dim}\hspace{0.5mm}}

\newcommand{\OO}{\mc{O}}

\newcommand{\mm}{\mc{M}}

\newcommand{\Mod}[1]{\ (\mathrm{mod}\ #1)}
\newcommand{\Z}{\mathbb{Z}}
\newcommand{\N}{\mathbb{N}}
\newcommand{\Q}{\mathbb{Q}}
\newcommand{\R}{\mathbb{R}}
\newcommand{\C}{\mathbb{C}}

\newcommand{\mulquiv}{\mc{M}_{q, 0}(Q, \alpha)}
\newcommand{\iso}{{\;\stackrel{_\sim}{\to}\;}}
\newcommand{\Spec}{\operatorname{Spec}}
\newcommand{\GL}{\operatorname{GL}}
\newcommand{\PGL}{\operatorname{PGL}}
\newcommand{\SL}{\operatorname{SL}}
\newcommand{\Sp}{\operatorname{Sp}}
\newcommand{\Hom}{\operatorname{Hom}}

\newtheorem{thm}{Theorem}[section]
\newtheorem*{mainthm*}{Main result}
\newtheorem{lem}[thm]{Lemma}
\newtheorem{coro}[thm]{Corollary}
\newtheorem{prop}[thm]{Proposition}
\newtheorem{conj}[thm]{Conjecture}
\newtheorem{ques}[thm]{Question}
\theoremstyle{definition}
\newtheorem{defn}[thm]{Definition}
\newtheorem{eg}[thm]{Example}
\theoremstyle{remark}
\newtheorem*{claim*}{Claim}
\newtheorem{rem}[thm]{Remark} 

\begin{document}
\begin{abstract} 
  We study the algebraic symplectic geometry of multiplicative quiver
  varieties, which are moduli spaces of representations of certain
  quiver algebras, introduced by Crawley-Boevey and Shaw
  \cite{cb-shaw}, called \tit{multiplicative preprojective
    algebras}. They are multiplicative analogues of Nakajima quiver varieties.
  They include  character varieties of (open) Riemann surfaces fixing conjugacy class closures of the monodromies around punctures, when the quiver is
  ``crab-shaped''.
  We prove that, under suitable hypotheses
  on the dimension vector of the representations, or the conjugacy classes of monodromies in the character variety case, the normalisations of such moduli spaces
  are symplectic singularities and the existence of a symplectic
  resolution depends on a combinatorial condition on the quiver and the
dimension vector. These results are analogous to those obtained by Bellamy and the first author in the ordinary quiver variety case, and for character varieties of closed Riemann surfaces.
 At the end of the paper, we outline some conjectural generalisations to moduli
  spaces of objects in 2-Calabi--Yau categories.
\end{abstract}
	
\maketitle
\centerline{\emph{To Sasha Beilinson and Victor Ginzburg on the occasion of their 60th birthdays, with admiration}} 

\section{Introduction}\label{intro}
\subsection{Motivation}
This paper is devoted to the study of the symplectic algebraic
geometry of coarse moduli spaces of (semistable) representations of
multiplicative preprojective algebras. These can be thought of as
multiplicative analogues of Nakajima quiver varieties \cite{nakajima},
which includes character varieties of (open) Riemann
surfaces. In particular, our attention is focused on tackling two main
problems: the first is to understand whether these multiplicative
quiver varieties are symplectic singularities, as defined by Beauville
in \cite{beauville}; the second is to classify all the possible cases
in which they admit symplectic resolutions.

Multiplicative
preprojective algebras were first defined by Crawley-Boevey and Shaw
in \cite{cb-shaw}, with the aim of better understanding Katz’s middle
convolution operation for rigid local systems, \cite{katz}. Another
important application contained in the seminal paper \cite{cb-shaw} is
the solution of (one direction of) the multiplicative Deligne-Simpson
problem in terms of the root data of a certain star-shaped quiver.
Moduli spaces of representations, in the sense of King \cite{king}, of
these algebras give rise to the so-called multiplicative quiver
varieties.
Ordinary (Nakajima) quiver varieties have appeared in numerous places
in representation theory, algebraic geometry, and mathematical
physics; their homology theories are closely related to the
representation theory of Kac--Moody Lie algebras \cite{nakajima}, and
their quantum cohomology is closely related to quantum R-matrices and
the Casimir connection \cite{maulik-okounkov}. A number of authors
have studied multiplicative quiver varieties since their definition:
among others, Jordan \cite{jordan} considered quantisations of such
varieties from a representation theoretic point of view by
constructing flat $q$-deformations of the algebra of differential
operators on certain affine spaces; a more geometric approach was used
by Yamakawa in \cite{yamakawa}, where a symplectic structure on these
moduli spaces was defined and studied. Some of these results will be
recalled in the next sections of the present paper. More recently,
(derived) multiplicative preprojective algebras appeared in the study
of wrapped Fukaya categories of certain Weinstein 4-manifolds
constructed by plumbing cotangent bundles of Riemann surfaces: see
\cite{etgu-lekili}. In the recent work of Chalykh and Fairon
\cite{fairon}, multiplicative quiver varieties are used to construct a
new integrable system generalising the Ruijsenaars--Schneider system,
which plays a central role in supersymmetric gauge theory and
cyclotomic DAHAs. Moreover, in work of McBreen--Webster \cite{mcbreen-webster}
and McBreen--Gammage--Webster \cite{mcbreen-gammage-webster}, 
 related to
 \cite[\S7]{kapranov}, mirror symmetry is studied for multiplicative hypertoric varieties, which include multiplicative quiver varieties when the dimensions are one.  Finally, mixed Hodge polynomials of
 character varieties and quiver varieties
 were studied in the ground-breaking work of Hausel, Lettelier and Rodriguez-Villegas using arithmetic methods \cite{h-l-rv-1, h-l-rv-2, h-l-rv-3}; they
 suggested that similar methods should apply to general multiplicative quiver varieties.
Given their appearance in so many different contexts, it
seems natural to perform a careful analysis of multiplicative quiver
varieties from the point of view of symplectic algebraic geometry.

The subject of symplectic resolutions and the more general symplectic
singularities (the latter dating to Beauville \cite{beauville}) has
recently gained importance in many areas of mathematics and
physics. Their quantisations subsume many of the important examples of
algebras appearing in representation theory (Cherednik and symplectic
reflection algebras, $D$-modules on flag varieties and representations
of Lie algebras, quantised hypertoric and quiver algebras,
etc.). There is a growing theory of symplectic duality, or
three-dimensional physical mirror symmetry (\cite{BFN-mdCb2,
  Nak-mdCb1, CHZ-moHsCb, blpw14} and many others), between pairs of
these varieties. Pioneering work of Braverman, Maulik and Okounkov
\cite{BMO-qcsr} (continued in the aforementioned
\cite{maulik-okounkov} and in many other places) show that their
quantum cohomology is also deeply tied to connections arising in
representation theory, related to derived autoequivalences of duals in
the sense of homological mirror symmetry (two-dimensional field
theories). Since, as mentioned before, quiver varieties play such an
important role here, it is expected that multiplicative quiver
varieties will as well. Moreover, the varieties in question are
instances of moduli spaces parametrising geometric objects. The study
of such spaces and their singularities is, in general, important in
algebraic geometry.

For all of these reasons, it is natural to ask when multiplicative
quiver varieties have symplectic singularities and admit symplectic
resolutions. We largely answer these questions, leaving a couple cases
(the so-called ``(2,2)''-cases related to O'Grady's examples
\cite{OGr-K3}, and the so-called isotropic cases, which are
multiplicative analogues of symmetric powers of du Val singularities), that
appear to require local structure theory. Our methods generalise those
of \cite{bellamy-schedler}, which we largely follow. They build on
Crawley-Boevey and Shaw’s pioneering work on multiplicative quiver
varieties (and extensions by Yamakawa \cite{yamakawa}), and apply (as
in \cite{bellamy-schedler}) Drezet’s factoriality criteria
\cite{Drezet} and Flenner’s theorem \cite{flenner} on extendability of
differential forms beyond codimension four.

\subsection{Summary of results on character varieties}\label{ss:char-v}
Since they are the easiest to state and perhaps of the broadest
interest, we first explain the results on
character varieties that follow from our considerations on multiplicative quiver varieties.  Fix a connected compact Riemann surface $X$ of genus $g\geq 0$, let $S=\{p_1, \dots, p_k\}\subset X$ be a subset of $k \geq 0$ points, and fix a tuple $\mathcal{C}=(\mathcal{C}_1, \dots, \mathcal{C}_k)$ of conjugacy classes
$\mathcal{C}_i\subset GL_n(\C), i=1, \dots, k$. Let $X^\circ := X \setminus \{p_1,\ldots,p_k\}$ be the corresponding punctured surface, and
let $\gamma_i$ be the homotopy class in $\pi_1(X^\circ)$ of some choice of loop around the puncture $p_i$ (having the same free homotopy class as a small counterclockwise loop around $p_i$). We define the character variety of $X^\circ$ with monodromies in $\overline{\mathcal{C}_i}$ as follows:
\begin{equation}
\mathcal{X}(g, k, \overline{\mathcal{C}}) := \{\chi: \pi_1(X^\circ) \to GL_n \mid \chi(\gamma_i) \in \overline{\mathcal{C}_i}\} /\!/ GL_n.
\end{equation}
As recalled in Section \ref{section-char} below, this has the structure of an affine algebraic variety. Note that $X$ (or $X^\circ$) does not appear in the notation on the left-hand side, since the result does not depend on the choice of $X$ up to isomorphism (only the identification of $\pi_1(X^\circ)$ is relevant).

Observe that, in order for this character variety to be nonempty, we
must have $\prod_{i=1}^k \det(\mathcal{C}_i) = 1$, where we let
$\det(\mathcal{C}_i)$ be defined as the determinant of any element of
$\mathcal{C}_i$. Let us assume this from now on.
Given $m \geq 1$ we let $m \cdot \mathcal{C}
=
(\mathcal{C}_1^{\oplus m}, \ldots, \mathcal{C}_k^{\oplus m})$.
We call $\mathcal{C}$ \emph{$q$-divisible} if
$\mathcal{C} = m \cdot \mathcal{C}'$ for $m \geq 2$ and
$\prod_{i=1}^k \det(\mathcal{C}'_i) = 1$. Call it $q$-indivisible if it is not $q$-divisible.
Below, $q$-indivisibility will be the most important criterion for the
existence of symplectic resolutions for
$\mathcal{X}(g, k, \overline{\mathcal{C}})$.

For each $\mathcal{C}_i$, let the minimal polynomial of any $A \in \mathcal{C}_i$ be $(x-\xi_{i,1}) \cdots (x-\xi_{i,w_i})$, ordered so that the sequence $\alpha_{i,j} := \text{rank}(A-\xi_{i,1}) \cdots (A-\xi_{i,j})$ has the property that $\alpha_{i,j}-\alpha_{i,j+1}$ is non-increasing in $j$ (for $0 \leq j \leq w_i-1$, setting $\alpha_{i,0} = n$). This is possible since the non-increasing property obviously holds when all the $\xi_{i,j}$ are equal.  The following quantities will have importance for us: 
\begin{equation}
  \ell := \sum_i \alpha_{i,1}, \quad p(\alpha) := 1 + n^2(g-1) + n\ell + \sum_{i=1}^k \sum_{j=1}^{w_i-1} \alpha_{i,j} \alpha_{i,j+1}
  - \sum_{i=1}^k \sum_{j=1}^{w_i} \alpha_{i,j}^2.  
  \end{equation}
  The quantity $2p(\alpha)$ is the ``expected dimension'' of the character variety, which is its actual dimension in many cases, as explained below.


Our main results on character varieties can be summarised as follows.
We divide separately into the genus $0$ and the positive genus cases.

Recall here that a \emph{symplectic singularity} is a normal variety
$X$ equipped with a symplectic structure $\omega_{\text{reg}}$ on the
smooth locus $X_{\text{reg}}$ such that, for any (or equivalently
every) resolution of singularities $\rho: \widetilde X \to X$,
$\rho^* \omega_{\text{reg}}$ extends to a regular two-form
$\widetilde \omega \in \Omega^2(\widetilde X)$. The map $\rho$ is
furthermore a symplectic resolution if $\widetilde \omega$ is
non-degenerate.
\begin{thm}\label{t:char-g0}
  Let $g=0$ and fix $n$ and conjugacy classes $\mathcal{C}_1,\ldots,\mathcal{C}_k \subseteq GL_n(\C)$ as above.
  \begin{itemize}
  \item If $\ell < 2n$, then one of the following exclusive possibilities occur, and can be computed
    by an explicit algorithm:
    \begin{itemize}
    \item $\mathcal{X}(0, k, \overline{\mathcal{C}})$ is empty;
    \item $\mathcal{X}(0, k, \overline{\mathcal{C}})$ is a point;
      \item There is a canonical datum $(n',k',\mathcal{C}'_1, \ldots, \mathcal{C}'_{k'},\iota)$ of $n' < n$, $k' \leq k$, and 
  conjugacy
  classes $\mathcal{C}'_1,\ldots,\mathcal{C}'_{k'} \subseteq GL_{n'}(\C)$ such that $\ell'$ (defined as above) satisfies $\ell' \geq 2n'$, and an isomorphism $\iota: \mathcal{X}(0, k', \overline{\mathcal{C}}') \to \mathcal{X}(0, k, \overline{\mathcal{C}})$.
\end{itemize}
  Suppose, therefore, that $\ell \geq 2n$.
\item If $\mathcal{C}$ is $q$-indivisible,
  then $\mathcal{X}(0,k,\overline{\mathcal{C}})$ admits a projective symplectic resolution (via geometric invariant theory). Therefore, its normalisation is a symplectic singularity. Moreover, $\dim \mathcal{X}(0,k,\overline{\mathcal{C}}) = 2 p(\alpha)$.
\item Suppose that $\mathcal{C}$ is $q$-divisible. Then, unless one of the conditions listed after Corollary \ref{c:crab-sr-intro} is satisfied (for $k \leq 5$), 
  the normalisation of $\mathcal{X}(0,k,\overline{\mathcal{C}})$ is a symplectic singularity which does not admit a symplectic resolution (in fact, it contains a singular terminal factorial open subset).
\end{itemize}  
\end{thm}
As mentioned in the theorem, the technique used to show non-existence
of symplectic resolutions is by identifying an open singular
factorial terminal subset.  It is well-known that singular factorial
terminal varieties cannot admit crepant resolutions, and hence not symplectic resolutions. Indeed, by 
Van der Waerden purity, any resolution of a singular factorial variety has
exceptional locus which is a
divisor. By definition, any crepant resolution of a terminal variety has
exceptional
locus of codimension at least two.  Put together, there is no crepant resolution of a singular factorial terminal variety.
\begin{rem} Note that, when $k \leq 2$ in genus zero, the character variety is always a point (or empty).
  \end{rem}
\begin{thm}\label{t:char-gg0}
  Suppose that $g \geq 1$.  Then the following holds:
  \begin{itemize}
    \item If $\mathcal{C}$ is $q$-indivisible, then  $\mathcal{X}(g,k,\overline{\mathcal{C}})$ admits a projective symplectic resolution (via geometric invariant theory). Therefore, its normalisation is a symplectic singularity. Moreover, it has dimension $2p(\alpha)$.
    \item If $\mathcal{C}$ is $q$-divisible, then unless
      one of the following conditions is satisfied, the normalisation of $\mathcal{X}(g,k,\overline{\mathcal{C}})$ is a symplectic singularity which does not admit a symplectic resolution (in fact, it contains a singular terminal factorial open subset):
      \begin{itemize}
      \item[(a)] $g=2, k=0$, and $n=2$;
      \item[(b)] $g=1, k=0$;
      \item[(c)] $g=1, k=1, w_1=2$, and $\alpha_{1,1} = p$, with $p$ prime.
      \end{itemize}
    Moreover, in all cases except case (b), $\dim \mathcal{X}(g,k,\overline{C})=2p(\alpha)$.
    \end{itemize}
  \end{thm}

  The proofs of these theorems is given in Section \ref{ss:proof-char}; they are consequences of our main results on multiplicative quiver varieties (particularly Corollary \ref{c:crab-sr}).
\begin{rem}
  Actually, the results above (slightly modified) should also apply
  to twisted character varieties, where we
  replace $\pi_1(X^\circ)$ by a finite central extension,
  corresponding to setting the relation
  $\prod_{i=1}^g[A_i, B_i]\prod_{j=1}^kM_j$ to be a root of unity
  times the identity matrix. To prove such a statement
  would require a straightforward generalisation of \cite{cb-monod}
  and of Section \ref{section-char} below. With this in hand, these results
  would 
   follow from Corollary \ref{c:crab-sr} just as before. For some more details,
  see Section \ref{ss:gen-dim-intro} of the introduction, where we
  describe roughly how to translate this corollary into the setting of
  twisted character varieties.
  \end{rem}

  \subsection{Multiplicative quiver varieties with special dimension vectors}\label{ss:mqv-intro}
  Recall that a quiver $Q$ is a directed graph. We let $Q_0$ denote
  the set of vertices and $Q_1$ the set of arrows (=edges).  Given $Q$
  together with a tuple of non-zero complex numbers
  $q\in (\C^\times)^{Q_0}$, one can define the multiplicative
  preprojective algebra $\Lambda^q(Q)$, over the semisimple ring
  $\C^{Q_0}$ (see Section \ref{ss:mpa} below). To a representation, we associate a
  dimension vector in $\N^{Q_0}$.  Given furthermore a stability
  parameter $\theta \in \Z^{Q_0}$, one can define a variety, denoted
  $\mm_{q,\theta}(Q,\alpha)$, which is a coarse moduli space of
  $\theta$-semistable representations of $\Lambda^q(Q)$ of dimension
  vector $\alpha$.  It is natural to ask what the dimension vectors of
  $\theta$-stable representations are.  Towards this end, one
  considers a combinatorially-defined subset
  $\Sigma_{q, \theta} \subseteq \N^{Q_0}$ of the set of all possible
  dimension vectors (defined in Section \ref{ss:sigma} below).  It has
  the property that, for $\alpha \in \Sigma_{q,\theta}$,
  the $\theta$-stable locus is dense in $\mm_{q,\theta}(Q,\alpha)$
  (and it is always open). However,
  it is unknown in general if $\mm_{q,\theta}(Q,\alpha)$ is
  non-empty.  It is expected, but not known, that these conversely describe all
  dimension vectors of stable representations, i.e.:  
\begin{equation}\label{sigma-converse}
\text{If there is a $\theta$-stable representation of $\Lambda^q(Q)$ of dimension $\alpha \in \N^{Q_0}$, then $\alpha \in \Sigma_{q,\theta}$.}\tag{*}
\end{equation}
In the case $\theta=0$, Crawley-Boevey kindly pointed out a work in progress
with Hubery towards a proof of (*).
We prove a weakened version of (*) below
(Corollary \ref{c:wk-sigma-converse}), replacing $\Sigma_{q,\theta}$
by a larger set. Note that, if (*) holds and furthermore
$\mm_{q,\theta}(Q,\alpha) \neq \emptyset$ for all
$\alpha \in \Sigma_{q,\theta}$, then put together we would obtain a
characterisation of the set $\Sigma_{q,\theta}$: in this case,
$\alpha \in \Sigma_{q,\theta}$ if and only if there exists a
$\theta$-stable representation of dimension $\alpha$. However, this is, again, unknown.


To define $\mm_{q,\theta}(Q,\alpha)$, we require
$\alpha \cdot \theta = 0$, and for it to be non-empty, we require that
$q^\alpha:=\prod_{i \in Q_0} q_i^{\alpha_i} = 1$.  Let
$N_{q,\theta} := \{\alpha \in \N^{Q_0} \mid q^\alpha = 1, \alpha \cdot
\theta = 0\}$.  We call a vector $\alpha \in N_{q,\theta}$
\emph{$q$-indivisible} if $\frac{1}{m} \alpha \notin N_{q,\theta}$ for
any $m \geq 2$.  Equivalently, writing $\alpha = m \beta$ for
$m=\text{gcd}(\alpha_i)$, we have that $q^\beta$ is a primitive $m$-th
root of unity.
Note that, if $\alpha \in N_{q,\theta}$ is indivisible, it is clearly
$q$-indivisible, although the converse does not hold in general.
(Unlike in the case of character varieties, here the $q$ in ``$q$-(in)divisible'' refers to an actual parameter; see Remark \ref{r:qdiv-char} for an explanation how the two notions nonetheless coincide.)

We denote by $p$ the following
function:
\[
p: \N^{Q_0} \rightarrow \Z,\ \ \ p(\alpha)=1-\frac{1}{2}(\alpha, \alpha) \geq 0,
\]
where $(-,-)$ denotes the Cartan--Tits form associated to the quiver
$Q$ (see Section \ref{subsec-pre} for more details).  Geometrically, $2p(\alpha)$ gives the ``expected dimension'' of
$\mm_{q,\theta}(Q,\alpha)$ (which is the actual dimension if
$\alpha \in \Sigma_{q,\theta}$ and
$\mm_{q,\theta}(Q,\alpha)\neq \emptyset$: see Remark \ref{r:dim-stab}
below).  If $p(\alpha)=1$, i.e., $(\alpha,\alpha)=0$, then $\alpha$ is
called \emph{isotropic}. Otherwise it is called \emph{anisotropic}.

One of the main results of this paper, proved in Section
\ref{section-sing}, is the following:
\begin{thm}\label{main-result}Let $\alpha \in \Sigma_{q, \theta}$ and
  assume that $\alpha \neq 2\beta$ for $\beta \in N_{q,\theta}$ and
  $p(\beta)=2$. Then, assuming it is non-empty,
  $\mm_{q, \theta}(Q, \alpha)$ satisfies the following:
\begin{itemize}
\item its normalisation is a symplectic singularity;
\item if $\alpha$ is $q$-indivisible, then for suitable generic
  $\theta'$, it admits a symplectic resolution
  of the form $\mm_{q,\theta'}(Q,\alpha) \to \mm_{q,\theta}(Q,\alpha)$;
\item if $\alpha=m\beta$ for $\beta \in \Sigma_{q,\theta}$ and
  $m \geq 2$, and $\mm_{q,\theta}(Q,\beta) \neq \emptyset$, then
  $\mm_{q,\theta}(Q,\alpha)$ does not admit a symplectic
  resolution. Moreover, for suitable generic $\theta'$,
  $\mm_{q,\theta'}(Q,\alpha)$ is a singular factorial
  terminalisation. In fact, $\mm_{q,\theta}(Q,\alpha)$ itself contains
  a singular, factorial, terminal open subset.
\end{itemize}
\end{thm}

Implicit in Theorem \ref{main-result} is the fact (see Lemma
\ref{l:varytheta} and Corollary \ref{c:birational} below) that, for all $\alpha \in \Sigma_{q,\theta}$,
$\mm_{q,\theta'}(Q,\alpha) \to \mm_{q,\theta}(Q,\alpha)$ is a
projective birational Poisson morphism for suitable
$\theta'$.  This implies, by definition, that it is a symplectic resolution if the source is smooth symplectic.
In the last part of the theorem, by singular factorial
terminalisation, we mean a projective birational Poisson morphism with
source a singular factorial terminal
variety.

In the case of generic $\theta$, the theorem can be simplified as follows, avoiding the need to check if a vector is in
$\Sigma_{q,\theta}$.
First, note that $\Sigma_{q,\theta}$, by definition, is a subset of the set of \emph{roots} for the quiver (which in turn equals the set of roots of the associated Kac-Moody Lie algebra in suitable cases). The real roots are those vectors obtained from elementary vectors $e_i, i \in Q_0$ by simple reflections $\alpha \mapsto \alpha - (\alpha, e_i) e_i$; the imaginary roots are those obtained by such reflections from non-negative (or non-positive) vectors with connected support and non-positive Cartan pairing with all $e_i$.
\begin{coro} \label{c:main-result}
  Fix an imaginary root $\alpha$ for $Q$.
  Let $q$ be such that $q^\alpha=1$. Let $\theta$ be generic
  (inside the hyperplane $\{\theta \cdot \alpha=0\}$). Then:
  \begin{itemize}
  \item[(i)] We have $\alpha \in \Sigma_{q,\theta}$ if and only if
    $\alpha$ is  $q$-indivisible or anisotropic. If $\alpha$ is $q$-indivisible, $\mm_{q,\theta}(Q,\alpha)$ is smooth symplectic.
  \item[(ii)] Assume $\alpha$ is $q$-divisible and anisotropic.
    Moreover, assume that $\alpha \neq 2 \beta$ for $p(\beta)=2$ and $q^\beta = 1$. Then $\mm_{q,\theta}(Q,\alpha)$ is a (normal) symplectic singularity.
  \item[(iii)] Under the assumptions of (ii), we have the following:
  \begin{itemize}
  \item If there exists a $\theta$-stable representation of $\Lambda^q(Q)$ of dimension $\frac{1}{m} \alpha$, for some $m \geq 2$, then $\mm_{q,\theta}(Q,\alpha)$ is singular, factorial, and terminal, and hence does not admit a symplectic resolution.
  \item  If, on the other hand,
  there are no $\theta$-stable representations of $\Lambda^q(Q)$ of dimension $r\alpha$ for all rational $r<1$,  then $\mm_{q,\theta}(Q,\alpha)$ is smooth.
\end{itemize}
\end{itemize}
\end{coro}
Note that, in the general case with $\alpha \in \Sigma_{q,\theta}$, the above corollary always describes the source of projective birational Poisson morphisms obtained by suitably varying $\theta$.
\begin{rem}
  Note that every $r\alpha, r \in \mathbb{Q}_{< 1}$ appearing in the
  theorem is also in $\Sigma_{q,\theta}$, by part (i).  Thus, if there
  exists a $\theta$-stable representation of every dimension in
  $\Sigma_{q,\theta}$, then in part (iii) we are necessarily in the
  first case. This condition holds in the additive case (with
  $\lambda \in \mathbb{R}^{Q_0}$), by \cite{bellamy-schedler}, but we don't
  have any other evidence that this holds here.  Also, note that the two cases are not exhaustive, so it could happen that there are some stable representations of dimension $r \alpha$ but not when $r = \frac{1}{m}$. In this (unexpected) situation, it would require more detailed analysis to determine whether a symplectic resolution exists.
  \end{rem}

\begin{rem}\label{r:(2,2)}
In the case left out of the theorem, where $\alpha=2\beta$ for some $\beta \in N_{q,\theta}$ satisfying $p(\beta)=2$ (we call this the ``$(2,2$)-case''),
we conjecture, as in the special case of character
varieties of rank two local systems on genus two surfaces handled in \cite{bellamy-schedler}, that $\mm_{q,\theta}(Q,\alpha)$ has a symplectic resolution obtained by blowing up the singular locus of $\mm_{q,\theta'}(Q,\alpha)$, for suitably generic $\theta'$.  However, in order to prove this, it is necessary to understand the \'etale local structure of $\mm_{q,\theta}(Q,\alpha)$, while at the moment all of our techniques are global in nature.  In Section \ref{sec-fut}, we discuss an approach to understand the local structure of the multiplicative quiver varieties, based on the conjectural 2-Calabi--Yau property of the multiplicative preprojective algebra for non-Dynkin quivers.
\end{rem}
\begin{rem} Note that, as part of Corollary \ref{c:main-result},
  when $\theta$ is
  generic (and $\alpha \in \Sigma_{q,\theta}$), we prove normality of the variety
  $\mm_{q, \theta}(Q, \alpha)$, see Proposition
  \ref{prop-normal}. Moreover, we conjecture normality for all
  $\theta$ (as well as for the $(2,2)$-case). Such a result requires a
  local understanding of the varieties $\mm_{q,\theta}(Q,\alpha)$,
  which would again follow from the conjectural 2-Calabi--Yau property
  for $\Lambda^q(Q)$ when $Q$ is not Dynkin: see Section
  \ref{sec-fut}.
\end{rem}
\subsection{Character varieties as (open subsets of) multiplicative quiver varieties}
In Section \ref{section-char}, extending results of \cite{cb-shaw} and
\cite{yamakawa}, we explain how character varieties identify as
natural open subsets of the multiplicative quiver varieties for
crab-shaped quivers (Theorem \ref{t:iso-mult-char}), also known as ``comet-shaped'' in \cite{h-l-rv-1}.  Namely,
the character variety identifies as an open subset of a
multiplicative quiver variety for the crab-shaped quiver described in Section
\ref{ss:char-v}, with appropriate parameter $q \in (\C^\times)^{Q_0}$. The open subset is defined by requiring the loops in the original (undoubled) quiver to act invertibly. In particular, in the genus zero case, the character variety equals the multiplicative character variety.
\begin{rem} \label{r:qdiv-char} By the above correspondence, a
  collection of conjugacy classes
  $\mathcal{C} \subseteq GL_n(\C^\times)$ is $q$-divisible in the
  sense of Section \ref{ss:char-v} (where $q$ is not yet a parameter) if and
  only if, for the associated quiver $Q$, dimension vector
  $\alpha \in \N^{Q_0}$, and parameter $q \in (\C^\times)^{Q_0}$, the
  vector $\alpha$ is $q$-divisible in the sense of Section \ref{ss:mqv-intro}.  We hope that this abuse of notation aids understanding.
  \end{rem}

It is then an interesting question which character varieties exhibit
the different properties discussed above, in particular, which ones
are the ``(2,2)''-cases where an ``O'Grady'' type resolution is
expected (see Remark \ref{r:(2,2)})?
Their classification
is achieved in Theorems \ref{comb1} and \ref{comb2}: all of these are
in the genus zero case (i.e., they are star-shaped quivers), with
three to five punctures and particular monodromy conditions, as
classified in Theorem \ref{comb1}, except for two cases in Theorem
\ref{comb2}. The latter cases correspond to once-punctured tori and to
closed genus two surfaces (with even rank and rank two local systems,
respectively, the former having particular monodromy about the
puncture).

\subsection{General dimension vectors}\label{ss:gen-dim-intro}
Although it is difficult to study directly quiver varieties of
dimensions $\alpha \notin \Sigma_{q,\theta}$, in the additive setting
this issue is alleviated by Crawley-Boevey's canonical decomposition,
expressing an arbitrary variety as a product of varieties for
dimension vectors in $\Sigma_{q, \theta}$ \cite[Theorem 1.1]{cb-deco}
(extended to $\theta \neq 0$ in \cite[Proposition
2.1]{bellamy-schedler}).  In Theorem \ref{t:decompo} below, we provide
a version of this decomposition in the multiplicative setting
using reflection functors, following the proof of \cite{cb-deco}, which
is weaker in the sense that the dimension vectors of the factors need
not be in $\Sigma_{q,\theta}$, and hence the factors could further decompose (although it is not known in general if they do). One of
the reasons why we must give the weaker statement is the
unavailability of (*); see Section \ref{ss:ref-dec} for more details.
Along with this, we 
prove a more general sufficient criterion for varying $\theta$ to
produce a symplectic resolution (Theorem \ref{t:res-flat}), that does
not require dimension vectors to be in
 $\Sigma_{q,\theta}$.
Using these results, in Theorem \ref{t:main-result-nonsigma}, we are able to extend Theorem \ref{main-result} to general dimension vectors.  The content of
Theorems \ref{t:decompo} and \ref{t:main-result-nonsigma} can be summarised in the following.  Here, $\widetilde \Sigma_{q,\theta}$ is a larger set
than $\Sigma_{q,\theta}$, consisting of roots for which a certain multiplicative moment map is flat; $\Sigma_{q,\theta}^{\text{iso}} \subseteq \Sigma_{q,\theta}$ is the subset of isotropic roots.  See Sections \ref{mult-quiv-var} and \ref{s:decomp} for details on these definitions. For any subset $X \subseteq \N^{Q_0}$, let $\N_{\geq 2} \cdot X := \{m \alpha \mid m \geq 2, \alpha \in X\}$.
\begin{thm}\label{t:main-result-nonsigma-intro}
Assume that $\mm_{q,\theta}(Q,\alpha)$ is non-empty. 

\noindent 
(i) There is a decomposition
$\alpha = \beta^{(1)} + \cdots + \beta^{(k)}$ with
  $\beta^{(i)} \in \widetilde{\Sigma}_{q,\theta} \cup \N_{\geq 2}\cdot \Sigma_{q,\theta}^{\text{iso}}$, such that
 the direct sum map produces an isomorphism (of reduced varieties):
\[
\prod_{i=1}^k \mathcal{M}_{q,\theta}(Q,\beta^{(i)}) \iso
\mathcal{M}_{q,\theta}(Q,\alpha).
\]

\noindent
(ii) Assume that this decomposition 
has neither
elements $\beta^{(i)} \in \N_{\geq 2}\cdot \Sigma_{q,\theta}^{\text{iso}}(Q,\alpha)$
nor $\beta^{(i)}=2\alpha$ for $\alpha \in N_{q,\theta}$ and $p(\alpha)=2$.  Then:
\begin{itemize}
\item The normalisation of $\mathcal{M}_{q,\theta}(Q,\alpha)$ is a symplectic singularity;
\item   Each factor $\mm_{q,\theta}(Q,\beta^{(i)})$ with $\beta^{(i)} \notin \Sigma_{q,\theta}$ admits a symplectic resolution;
\item If for any factor $\beta^{(i)}$ there exists a $\theta$-stable
  representation of dimension
  $\gamma^{(i)} = \frac{1}{m} \beta^{(i)}$ with $m \geq 2$, then $\mm_{q,\theta}(Q,\alpha)$ does not admit a symplectic resolution. In fact, it has an open, singular, terminal, factorial subset.
\end{itemize}
\end{thm}
Putting everything together, in Corollary \ref{c:crab-sr}, we are able
to give a classification
of crab-shaped settings whose multiplicative quiver varieties admit
symplectic resolutions. By Theorem \ref{t:iso-mult-char}, we also
deduce the corresponding statement for character varieties (Theorems
\ref{t:char-g0} and \ref{t:char-gg0}), which are open subsets of these varieties, for $\theta=0$ and
for certain values of the parameter $q$.
To state the result, first recall that the Jordan quiver is the quiver with one vertex
and one arrow (a loop). The \emph{fundamental region} $\mc{F}(Q)$
consists of those nonzero vectors $\alpha \in \N^{Q_0}$ with connected
support and with $(\alpha, e_i) \le 0$ for all $i$. As we explain
below, by applying certain reflection functors, we can reduce to this
case.  We give a simplified version of the statement of Corollary
\ref{c:crab-sr} below; see the full statement for precise details.
\begin{coro}\label{c:crab-sr-intro}
  Let $Q$ be a crab-shaped quiver and $\alpha \in N_{q,\theta}$ a vector in the fundamental region with $\alpha_i > 0$ for all $i \in Q_0$.
Further assume that $(Q,\alpha)$ is \textbf{not} one of the following cases: 
\begin{itemize}
\item[(a)] $\beta:=\frac{1}{2}\alpha$ is integral, $q^\beta = 1$,
  and  $(Q,\beta)$ is one of the quivers in Theorem \ref{comb1} and Theorem \ref{comb2};
\item[(b)] $Q$ is affine Dynkin of type
  $\tilde{A}_0$ (i.e., the Jordan quiver with one vertex and one arrow), $\tilde{D}_4$ or $\tilde{E}_6, \tilde{E}_7, \tilde{E}_8$)
  and $\alpha$ is a $q$-divisible multiple of the indivisible imaginary root $\delta$ of $Q$.
\end{itemize}
Then:
\begin{itemize}
\item The normalisation of $\mm_{q,\theta}(Q,\alpha)$ is a symplectic singularity;
\item If $\alpha$ is $q$-indivisible, 
$\mm_{q,\theta}(Q,\alpha)$ admits a symplectic resolution;
\item If $\alpha$ is $q$-divisible, and $\alpha$ is not: (c) a prime multiple
  of one of the quivers listed in Theorem \ref{t:flat-sigma}.(b2) below (a framed affine Dynkin quiver with dimension vector $(1,m\delta)$ with $m\delta$ $q$-divisible) with $\theta \cdot \delta = 0$,
then $\mm_{q,\theta}(Q,\alpha)$ does not admit a symplectic resolution (it contains an open singular factorial terminal subset).
\end{itemize}
\end{coro}
Thus, after reducing to the fundamental region, a symplectic resolution exists if and only if the dimension vector is $q$-indivisible, unless we are in one of the following three open cases:
\begin{itemize}
\item[(a)] twice one of the dimension vectors appearing in Theorems \ref{comb1} and \ref{comb2} below, which correspond to one of certain (twisted) character
  varieties of a sphere with $3$ to $5$ punctures (with rank at most $24$),
    of a once-punctured torus (of rank $4$) or a closed genus two surface (of rank $2$);
\item[(b)] a $q$-divisible imaginary root on an affine Dynkin quiver (of type $\tilde A_0, \tilde D_4, \tilde E_6, \tilde E_7$, or $\tilde E_8$), which corresponds to either a (twisted) character variety of a closed torus (type $A$),
  or a (twisted) character variety of a sphere with $3$ or $4$ punctures (type $E$ or $D$, respectively)
  with particular rank and monodromy conditions;
  \item[(c)] a prime multiple of the vector $(1,\ell\delta)$ on a
    framed affine Dynkin quiver (again of type
    $\tilde A_0, \tilde D_4, \tilde E_6, \tilde E_7$, or $\tilde E_8$)
    with $\theta \cdot \delta=0$, which corresponds again to a certain
    character variety of a once-punctured torus or a sphere with $4$
    or $5$ punctures.
  \end{itemize}
  Here when we say ``correspond'', we mean precisely that, for
  $\theta=0$, the multiplicative quiver varieties equal the given
  (twisted) character varieties, whereas for the genus $\geq 1$ case,
  the latter is the open subset of the former where the
  transformations corresponding to loops in the undoubled quiver are
  invertible.  Setting $\theta \neq 0$ gives a partial resolution
  (which may be an actual one, as in case of $\theta$ generic and
  $\alpha$ $q$-indivisible).
  \begin{rem} For the ordinary (untwisted) character varieties of closed
    genus one or two surfaces appearing in the lists for cases (b)
    and (a) above, a symplectic resolution exists; see, e.g., \cite[\S 8]{bellamy-schedler}. The proof makes use of
    Poincar\'e--Verdier duality for closed surfaces; perhaps a suitable generalisation of this for orbifolds would allow us to extend those results to the orbifold case. If so, we could remove 
    $\tilde A_0$
    from case (b) and the quiver with one vertex and two loops from
    (a).
\end{rem}

\subsection{Outline of the paper}
The outline of the paper is as follows: in Section \ref{mult-quiv-var}
we recall some basic facts about quivers and root systems and
establish the notation that shall be used throughout the paper. We
then recall the definition of multiplicative preprojective algebras
and outline some of their algebraic properties. These are needed in
the construction, via Geometric Invariant Theory (GIT), of their
moduli spaces of semistable representations, following \cite{king}. In
Definition \ref{defn-sigma}, we introduce the fundamental
combinatorially-defined subset $\Sigma_{q, \theta}$ of roots appearing
in our main results (which is expected to contain, if not equal, the
dimension vectors of $\theta$-stable representations of the
multiplicative preprojective algebra). We extend some properties
of multiplicative quiver varieties with dimension vector in
$\Sigma_{q, \theta}$ that were originally formulated and proved in
\cite{cb-shaw} in the case of a trivial stability condition to the general case.

In Section \ref{section-char} we prove that, for quivers of special type, namely those which are crab-shaped (see Figure 1), there is an isomorphism between (an open subset of) the corresponding multiplicative quiver variety and a character variety arising from considering representations of the fundamental group of a punctured Riemann surface where the monodromies of loops around the punctures are assumed to lie in the closure of certain conjugacy classes. In order to build such a correspondence we exploit \cite[Lemma 8.2 and Theorem 1.1]{cb-shaw}. Note that an instance of the correspondence between multiplicative quiver varieties and local systems on punctured surfaces already appeared in \cite{yamakawa}, where a proof is given for the case of the punctured projective line. Our result applies to all genera.
 Thanks to this correspondence, and to the results proved in Section \ref{section-sing}, we are able to extend the work of Bellamy and the first author from closed Riemann surfaces to open ones.
 Another interesting aspect of this correspondence is that it could be conjecturally combined with the Non-abelian Hodge Theorem to extend the main results of \cite{tirelli}, proved by the second author, in the context of moduli spaces of parabolic Higgs bundles. More details on this topic are provided in Section \ref{sec-fut}, where possible future research directions of the present work are discussed. 

 Section \ref{section-sing} contains the proof of Theorem \ref{main-result}. A careful study of the singularities of multiplicative quiver varieties is carried out. First, we show that the smooth locus is precisely the $\theta$-stable locus.
The remaining part of the section is devoted to the study of the nature of the singular locus. To this end, we use techniques from the work \cite{bellamy-schedler} of Bellamy and the first author to prove that, under suitable hypotheses, the singularities are symplectic. We also prove that, under certain conditions, the moduli space $\mm_{q, \theta}(Q, \alpha)$ contains an open subset which is singular, factorial, and terminal. 
As a consequence, $\mm_{q,\theta}(Q,\alpha)$ does not admit a symplectic resolution.  Moreover, for generic $\theta$, we see that the open subset is the entire variety.
The only case left out by Theorem \ref{main-result}, when $\alpha=2\beta$
for $\beta \in N_{q,\theta}$ and $p(\beta)=2$,
is more subtle than the others. The corresponding result in the context of ordinary quiver varieties, treated in \cite{bellamy-schedler}, is based on the study of the local structure of such varieties. In our case, such a tool is still not available, but will hopefully be the object of future research. 

In Section \ref{sec-comb}, namely in Theorems \ref{comb1} and
\ref{comb2}, we combinatorially classify all the the pairs
$(Q, \alpha)$, formed by a crab-shaped quiver and a corresponding
dimension vector in the fundamental region
such that
$(p(\mr{gcd}(\alpha)^{-1}\alpha),\ \mr{gcd}(\alpha))= (2,2)$. This is
relevant as these are the cases expected, for generic $\theta$, to
admit ``O'Grady''-type resolutions (i.e., by blowing up the
singular locus). This is also important since, by Theorem \ref{t:main-result-nonsigma-intro},
it allows us to recognise whether a dimension
vector in the fundamental region is expected to admit a symplectic
resolution or not.

In Section \ref{s:decomp} we face the problem of existence of
symplectic resolutions of multiplicative quiver varieties for general
dimension vectors. In order to do so, we follow the approach of
Bellamy and the first author. In particular, we prove that a
multiplicative quiver variety has a canonical decomposition into
natural factors, see Theorem \ref{t:decompo}. This can be
viewed as a multiplicative analogue of Crawley-Boevey's decomposition
\cite{cb-deco}, and we follow his proof, obtaining some more factors
due to the unavailability of (*) and some local structure results.
Our result makes it possible to solve the problem by understanding it
only at the level of such indecomposable factors, which are
multiplicative quiver varieties with particular dimension vectors.  We
then extend the GIT construction of symplectic resolutions by varying
$\theta$ to dimension vectors not in $\Sigma$ (Theorem
\ref{t:res-flat}); this includes multiplicative analogues of framed
quiver varieties such as Hilbert schemes of $\C^2$ and of
hyperk\"ahler almost locally Euclidean spaces.  In Theorem
\ref{t:main-result-nonsigma} we make use of our canonical
decomposition and, modulo some cases for which the question remains
still open, we classify all multiplicative quiver varieties with
arbitrary dimension vector that admit a symplectic resolution. As an
application of this result, by restricting to crab-shaped quivers, we
give an explicit classification of the character varieties of
punctured surfaces admitting symplectic resolutions (Corollary
\ref{c:crab-sr}), combining Theorem \ref{t:main-result-nonsigma} and
the results of Section \ref{section-char}.

Last, Section \ref{sec-fut} contains some open questions which naturally arise from the study carried out in the present paper. 
One open question which would be interesting to tackle regards the
possibility to extend the main results of \cite{tirelli} starting from
the correspondence outlined in Section \ref{section-char}: in Section
\ref{sec-fut} we provide some details on this topic by describing
which moduli space one would need to consider, via the non-abelian
Hodge Theorem in the non-compact case \cite{simpson-harm}, and we
conjecture a generalisation of the Isosingularity Theorem for such
moduli spaces. To conclude, we outline a general setting and of pose a
number of questions which should generalise the work of the present
and many other papers, e.g., \cite{arb-sacca, bellamy-schedler,
  kaledin-lehn, tirelli}: it seems that many of the techniques
exploited in the mentioned works are particular instances of theorems
which conjecturally hold in the context of moduli spaces of semistable
objects in 2-Calabi--Yau categories, under suitable hypotheses. This
assertion is motivated also by the work of Bocklandt, Galluzzi and
Vaccarino, \cite{bocklandt-et-al}, who studied moduli spaces of
representations of 2-Calabi--Yau algebras and proved that such
varieties locally look like representations of (ordinary)
preprojective algebras. This seems to be a singular, local, underived
version of the phenomenon that representation varieties of Calabi--Yau
algebras are (shifted) symplectic (as announced by Brav and
Dyckerhoff; see a similar result in \cite{Yeu-wCYsmr}).


\subsection*{Acknowledgements} This work is part of the second author's PhD thesis. We thank Gwyn Bellamy, William Crawley-Boevey, Ben Davison, Emilio Franco, Marina Logares, for many useful conversations and enlightening comments on the topic of this paper. The second author would like to thank Jacopo Stoppa and the Department of Mathematics at SISSA, Trieste, for providing an excellent and stimulating working environment. The work of the second author was supported by the Engineering and Physical Sciences Research Council [EP/L015234/1]. The EPSRC Centre for Doctoral Training in Geometry and Number Theory (The London School of Geometry and Number Theory), University College London and Imperial College London. The first author would like to thank the Hausdorff Institute for Mathematics and the Max Planck Institute for Mathematics in Bonn, where some of this research was carried out. We would also like to thank the anonymous referee for his/her useful suggestions.

\section{Multiplicative quiver varieties}\label{mult-quiv-var}
In this section we give the definition of multiplicative quiver varieties following \cite{cb-shaw} and recall some basic properties of such moduli spaces which will be useful in the arguments of the proof of our main theorems. In addition to these known results, we prove a new one, concerning the normality of the aforementioned varieties.
	
Throughout the paper, we work over the field $\C$ of complex numbers. 

\subsection{Preliminaries on quivers and root systems}\label{subsec-pre}We recall the basic definitions and fix the notations from the theory of quiver representations. Let $Q$ be  a finite quiver ($=$ directed graph with finitely many vertices and edges). We let let $Q_0$ and $Q_1$ denote the set of vertices and the set of arrows ($=$ edges) of $Q$, respectively. Moreover, for an arrow $a\in Q_1$, let $h(a)$ and $t(a)$ denote the head and the tail of $a$, respectively. For a dimension vector $\alpha\in \mb{N}^{Q_0}$, we will denote by $\mr{Rep}(Q, \alpha)$ the space of representations of $Q$ of dimension $\alpha$, which is naturally acted upon by the group $\GL(\alpha):=\prod_{i\in Q_0}\GL(\alpha_i)$. 
	
The coordinate vector at vertex $i$ is denoted $e_i$. The set $\N^{Q_0}$ of dimension vectors is partially ordered by $\alpha \ge \beta$ if $\alpha_i \ge \beta_i$ for all $i$ and we say that $\alpha > \beta$ if $\alpha \ge \beta$ with $\alpha \neq \beta$. The \emph{support} of a vector $\alpha$ is the set of $i \in Q_0$ with $\alpha_i \neq 0$; $\alpha$ is called \emph{sincere} if its support is all of $Q_0$.
The Euler (or Ringel) form on $\Z^{Q_0}$ is defined by
$$
\langle \alpha, \beta \rangle = \sum_{i \in Q_0} \alpha_i \beta_i - \sum_{a \in Q_1} \alpha_{t(a)} \beta_{h(a)}.
$$
Let $(\alpha,\beta) = \langle \alpha, \beta \rangle + \langle \beta, \alpha \rangle$ denote the corresponding Cartan (or Tits) form and set $p(\alpha) = 1 -\langle \alpha, \alpha \rangle$. The fundamental region $\mc{F}(Q)$ is the set of nonzero $\alpha \in \N^{Q_0}$ with connected support and with $(\alpha, e_i) \le 0$ for all $i$. For $q \in (\C^\times)^{Q_0}$ and $\alpha \in \N^{Q_0}$, let $q^\alpha := \prod_{i \in Q_0} q_i^{\alpha_i}$.
	
If $i$ is a loopfree vertex, so $p(e_i) = 0$, there is a reflection $s_i : \Z^{Q_0} \rightarrow \Z^{Q_0}$ defined by $s_i(\alpha) = \alpha -(\alpha,e_i)e_i$. The real roots (respectively imaginary roots) are the elements of $\Z^{Q_0}$ which can be obtained from the coordinate vector at a loopfree vertex (respectively $\pm$ an element of the fundamental region) by applying some sequence of reflections at loopfree vertices.  Let $R^+$ denote the set of positive roots. Recall that a root $\beta$ is \textit{isotropic} imaginary if $p(\beta) = 1$ and \textit{anisotropic} imaginary if $p(\beta) > 1$. We say that a dimension vector $\alpha$ is \textit{indivisible}  if the greatest common divisor of the $\alpha_i$ is one.  
\subsection{Multiplicative preprojective algebras}\label{ss:mpa}
We now define the quiver algebras whose moduli of representations are the varieties of interest in the present paper. To this purpose, let $Q$ be a finite quiver, fixed once and for all in this section. First, recall that, for a vector $\lambda\in \C^{Q_0}$, the \tit{deformed preprojective algebra} $\Pi^\lambda(Q)$ is the quotient of the path algebra $\C \overline{Q}$ of the doubled quiver $\overline{Q}$ by the relation
\[
\sum_{x\in Q_1}[x, x^*]=\sum_{i\in Q_0}\lambda_i e_i,
\]
where $x^*$ denotes the dual loop to $x$ in $\overline{Q}_1$; it is well-known that Nakajima quiver varieties can be interpreted as moduli spaces of ($\theta$-semistable) representations of such algebras. As one might expect, the defining relation for multiplicative preprojective algebras is a multiplicative analogue of the above equation: choose $q\in (\C^\times)^{Q_0}$ and define $A(Q)$ to be the universal localisation of the path algebra $\C \overline{Q}$ such that $1+x x^*$ and $1+x^*x$ are invertible, for $x\in \overline{Q}_1$. Then, following \cite[Definition 1.2]{cb-shaw}, the \tit{multiplicative preprojective algebra} $\Lambda^q(Q)$ is defined as the quotient of $A(Q)$ by the relation 
\[
\prod^{<}_{x\in \overline{Q}_1}(1+xx^*)^{\varepsilon(x)}=\sum_{i\in Q_0}q_i e_i, 
\]
where $\varepsilon(x)$ equals $1$ if $x\in Q_1$ and $-1$ otherwise and the product is ordered by an arbitrary choice of ordering ``$<$'' on $\overline{Q}_1$. It is known, by \cite[Theorem 1.4]{cb-shaw} that, up to isomorphism, $\Lambda^q(Q)$ does not depend on the orientation of the quiver or the chosen ordering on $\overline{Q}_1$.  When the quiver $Q$ is clear from the context, we will use the shortened notation $\Lambda^q$ in place of $\Lambda^q(Q)$.

Analogously to the additive case mentioned above, representations of $\Lambda^q(Q)$ are representations of the underlying quiver $\overline{Q}$, $\{(V_i)_{i\in \overline{Q}_0},(\phi_a)_{a\in \overline{Q}_1}\}$, satisfying the additional relations: 
\[
\mr{Id}_{V_{h(a)}}+\phi_a\phi_a^*\ \mr{is\ an\ invertible\ endomorphism\ of}\ V_{h(a)}\ \mr{for\ all}\ a\in\overline{Q}_1\]
\[
\prod_{a\in \overline{Q}_1, h(a)=i}(\mr{Id}_{V_{h(a)}}+\phi_a\phi_a^*)^{\varepsilon(a)}=q_i\mr{Id}_{V_i}\ \mr{for\ all}\ i\in Q_0,
\]
where, for an edge $a\in \overline{Q}_1$, $\phi_a^*$ denotes the linear map $\phi_{a^*}$, $a^*$ being the dual edge of $a$. 

For a positive vector $\alpha\in \N^{Q_0}$, we denote by $\mr{Rep}(\Lambda^q, \alpha)$ the set of representations of $\Lambda^q$ with $V_i=\C^{\alpha_i}$ for all $i$. This can be given an obvious affine scheme structure via the subset of matrices satisfying the obvious polynomial equations. We will work below with the reduced subvariety of this affine scheme.
\begin{rem}
By taking determinants of the defining relation for the multiplicative preprojective algebra, one can easily see that if $\Lambda^q$ has a representation of dimension vector $\alpha$, then $q^{\alpha}=1$, which, thus, is a necessary condition to be satisfied in order to have a non-empty moduli space. 
\end{rem}
The following results, which will be used in the next sections, are proved in \cite{cb-shaw}. It is worth pointing out that, even though we work over $\C$, these statements hold true over an arbitrary field $\mb{K}$.
\begin{prop}\label{dim-ext}
If $X$ and $Y$ are finite-dimensional representations of $\Lambda^q$, then
\[
\dim \mr{Ext}^1_{\Lambda^q}(X, Y)=\dim \mr{Hom}_{\Lambda^q}(X, Y)+\mr{Hom}_{\Lambda^q}(Y, X)-(\underline{\dim}X, \underline{\dim}Y)
\]
\end{prop}

The following result concerns the geometry of the space $\mr{Rep}(\Lambda^q, \alpha)$ of representations of the algebra $\Lambda^q$, when a dimension vector $\alpha\in \N^{Q_0}$ is fixed. Define $g_\alpha$ as $g_\alpha:=-1+\sum_{i\in Q_0} \alpha_i^2$.
\begin{prop}\label{prop-equi}
$\mr{Rep}(\Lambda^q, \alpha)$ is an affine variety, and every irreducible component has dimension at least $g_\alpha+2p(\alpha)$. The subset $T\subset\mr{Rep}(\Lambda^q, \alpha)$ of representations $X$ with trivial endomorphism algebra, $\mr{End}(X)=\C$, is open and, if non-empty, smooth of dimension $g_\alpha+2p(\alpha)$.
\end{prop}

\subsection{Reflection functors for $\Lambda^q(Q)$}As in the additive case, one can define reflection functors for the multiplicative preprojective algebra $\Lambda^q(Q)$: let $v$ a loopfree vertex in $Q$ and define 
\[
u_v: (\C^\times)^{Q_0}\rightarrow (\C^\times)^{Q_0}, \ \ \ u_v(q)_w=q_v^{-(e_v, e_w)}q_w.
\]
It is easy to see that the map $u_v$ satisfies the following  identity:
\[
(u_v(q))^\alpha=q^{s_v(\alpha)},
\]
where $s_v$ is the reflection map defined in Section \ref{subsec-pre}. The main result concerning such maps is analogous to the properties of  reflections functors for $\Pi^\lambda(Q)$.
\begin{prop} \cite[Theorem 1.7]{cb-shaw} If $v$ is a loopfree vertex and $q_v\neq 1$, then there is an equivalence of categories $F_q$ from the category of representations of $\Lambda^q$ to the category of representations of $\Lambda^{u_v(q)}$, acting on dimension vectors through the reflection $s_v$. The inverse equivalence is given by $F_{u_v(q)}$.
\end{prop} 
We will need also reflections on $\theta$.
 Define
\[
r_v: \Z^{Q_0}\rightarrow \Z^{Q_0}, \ \ \ r_v(\theta)_w=\theta_w - (e_v, e_w) \theta_v.
\]
\begin{defn} \label{d:admiss-refl}
The map $(q,\theta,\alpha) \mapsto (u_v(q),r_v(\theta), s_v(\alpha))$, 
is called a reflection. If $\theta_v \neq 0$ or $q_v \neq 1$, it is called
an \emph{admissible reflection}.
\end{defn}
We will explain below isomorphisms of multiplicative quiver varieties,
due to Yamakawa, which are closely related to the above equivalence.

\subsection{Moduli of representations of $\Lambda^q(Q)$} We shall now outline the construction of the varieties of interest for the present work. As mentioned above, the general definition involves a stability condition $\theta\in\Z^{Q_0}$, which we fix for the rest of this section.\\ 
The seminal work of King \cite{king}
allows one to define the notion of $\theta$-semistability for modules over
$\Lambda^q$:
\begin{defn}
Let $M$ be a finite-dimensional representation of $\Lambda^q$ such that $\underline{\dim}M\cdot \theta=0$. The module $M$ is said to be \tit{$\theta$-semistable} if, for any sub-module $N\subset M$ 
\[
\theta \cdot \vdim N\leq 0.
\]
The module $M$ is said to be \tit{$\theta$-stable} if the strict inequality holds. Finally, $M$ is said to be \tit{$\theta$-polystable} if it is a direct sum of $\theta$-stable representations. Given a set (or scheme) $X$ of representations, let $X^{\theta\mr{-}s}$ and $X^{\theta\mr{-}ss}$ denote the $\theta$-stable and $\theta$-semistable loci, respectively. We will use the notation $\mr{Rep}^{\theta\mr{-}s}(Q,\alpha) := \mr{Rep}(Q,\alpha)^{\theta\mr{-}s}$ and similarly for $\theta\mr{-}ss$.
\end{defn}
\begin{rem} By \cite[Proposition 3.1]{king}, one has that the above definition of stability coincides with the usual one coming from GIT: indeed,  consider the character 
\[
\chi_\theta: \GL(\alpha)\rightarrow\C^\times, \ \ \ \ \ (g_i)_{i\in Q_0}\mapsto \prod_{i\in Q_0}(\det g_i)^{-\theta_i}.
\]
It defines a linearisation on the trivial line bundle $\mr{Rep}(Q, \alpha)\times \C$ of the action of $\GL(\alpha)$ on $\mr{Rep}(Q, \alpha)$; thus, one can define the notion of $\chi_{\theta}$-(semi)stability \`a la Mumford, \cite{mumford}. The aforementioned result of King proves that $M$ is $\theta$-(semi)stable if and only if it is $\chi_{\theta}$-(semi)stable.
\end{rem}
Using the notion above one can construct the moduli space of (semistable) representations of $\Lambda^q$ of dimension $\alpha$ as follows (see \cite[\S 2]{yamakawa}, for the details):
define
\[
  \mr{Rep}^{\circ}(\overline{Q}, \alpha)=\{\phi
  \in \mr{Rep}(\overline{Q}, \alpha)\ |\ \det(1+\phi_a\phi_a^*)\neq 0,\ a\in \overline{Q}_1\}.
\]
Here and in the following, for $\phi \in \mr{Rep}(\overline{Q}, \alpha)$, we let
$\phi_a, a \in \overline{Q}_1$ denote the component linear maps. 
One can then consider the map
\[
\Phi: \mr{Rep}^{\circ}(\overline{Q}, \alpha)\longrightarrow \GL(\alpha),
\]
defined by the formula
\[
  \Phi(\phi
)=\prod_{a\in\overline{Q}_1}^{<}(1+\phi_a\phi_a^*)^{\varepsilon(j)}.
\]
Let us identify $\C^\times$ also with the scalar matrices in $\GL(\alpha_i)$, and hence $(\C^\times)^{Q_0}$ also with a subset of $\GL(\alpha)$.
Fixing $q\in (\C^\times)^{Q_0}$, one has that $\mr{Rep}(\Lambda^q(Q), \alpha)$ is the set-theoretic preimage $\Phi^{-1}(q)$. Thus, one can give the following
\begin{defn}
The \tit{multiplicative quiver variety} $\mathcal{M}_{q, \theta}(Q, \alpha)$ is the GIT quotient
\[
\mathcal{M}_{q, \theta}(Q, \alpha):=(\mr{Rep}^{\theta\mr{-}ss}(\overline{Q}, \alpha)\cap\Phi^{-1}(q))\ /\!\!/\ \GL(\alpha).
\]
\end{defn}
\begin{rem}
The reason for the terminology in the previous definition is apparent: the equations defining the multiplicative preprojective relation are modifications of the ones used to define the usual \tit{{deformed} preprojective algebras}, whose moduli of (semistable) representations are Nakajima quiver varieties.
\end{rem}
It is worth recalling a fundamental result of King, which gives a moduli-theoretic interpretation---in the sense of (representable) moduli functors---to $\mathcal{M}_{q, \theta}(Q, \alpha)$. 
\begin{thm}\cite[Propositions 3.1 and 3.2]{king} 
\label{t:king}
Assume $\theta\in \Z^{Q_0}$. Then, $\mathcal{M}_{q, \theta}(Q, \alpha)$ is a coarse moduli space for families of $\theta$-semistable representations up to $S$-equivalence.
\end{thm}
Here two $\theta$-semistable representations are $S$-equivalent if and only if they have the same composition factors into $\theta$-stable representations (i.e., they have filtrations whose subquotients are isomorphic $\theta$-stable representations). This means that every point in $\mm_{q,\theta}(Q,\alpha)$ has a unique representative which is $\theta$-polystable, up to isomorphism.

Precisely as in \cite[Lemma 2.4]{bellamy-schedler}, we have the following instance of the well-known principle of GIT:
\begin{defn} \label{d:tss}
We say that $\theta' \geq \theta$ if every $\theta'$-semistable representation of $\Lambda^q$ is also $\theta$-semistable.
\end{defn}
Note that $\theta' \geq \theta$ is implied if the purely combinatorial condition holds, that $\theta \cdot \beta > 0$ implies $\theta' \cdot \beta > 0$ for all $\beta < \alpha$. 
\begin{lem}\cite[Lemma 2.4]{bellamy-schedler} \label{l:varytheta}
Let $\alpha \in N_{q,\theta}$ be such that $\mm_{q,\theta}(Q,\alpha) \neq \emptyset$. Take $\theta' \geq \theta$. Then we have a projective Poisson morphism
$\mm_{q,\theta'}(Q,\alpha) \to \mm_{q,\theta}(Q,\alpha)$ induced by
the inclusion $\Phi^{-1}(q)^{\theta'-ss} \subseteq \Phi^{-1}(q)^{\theta\mr{-}ss}$.
\end{lem}

We caution that this morphism need not be surjective (and indeed the source could be empty when the target is not). However, in many cases, as we will see,
it produces a symplectic resolution.

\subsection{Reflection isomorphisms}
There is a multiplicative analogue of the Lusztig-Maffei-Nakajima reflection isomorphisms of quiver varieties (see in particular \cite[Theorem 26]{Maffei}), due to Yamakawa, which makes use of the reflection functors $F_q$. Let us extend the definition of  $\mathcal{M}_{q,\theta}(Q,\alpha)$ to $\alpha \in \Z^{Q_0}$
by setting it to be empty in the case that  $\alpha_i < 0$ for some $i$.
\begin{thm} \cite[Theorem 5.1]{yamakawa} \label{t:yama}
An admissible reflection $(q,\theta,\alpha) \mapsto (u_v(q),r_v(\theta), s_v(\alpha))$ induces an isomorphism of multiplicative quiver varieties,
$\mathcal{M}_{q,\theta}(Q,\alpha) \cong \mathcal{M}_{u_v(q),r_v(\theta)}(Q,s_v(\alpha))$.
\end{thm}

\subsection{Poisson structure on $\mathcal{M}_{q, \theta}(Q, \alpha)$} In order to construct a Poisson structure on $\mathcal{M}_{q, \theta}(Q, \alpha)$, we shall use the theory of quasi-Hamiltonian reductions, first developed in \cite{alekseev} for the case of real manifolds, and then treated by Boalch, \cite{boalch}, and Van den Bergh \cite{van-den-bergh-1, van-den-bergh-2} in the holomorphic and algebraic settings. To this end, note that the map $\Phi$ defined above is a group valued moment map for the quasi-Hamiltonian action of $\GL(\alpha)$ on $\mr{Rep}^{\circ}(\overline{Q}, \alpha)$. Thus, the variety $\mathcal{M}_{q, \theta}(Q, \alpha)$ can be considered as the quasi-Hamiltonian reduction of $\mr{Rep}^{\circ}(\overline{Q}, \alpha)$ modulo the action of $\GL(\alpha)$. From the properties of such a reduction, we obtain that $\mathcal{M}_{q, \theta}(Q, \alpha)$ is a Poisson variety. Moreover, defining 
\[
\mathcal{M}^s_{q, \theta}(Q, \alpha):= (\mr{Rep}^{\theta\mr{-}s}(\overline{Q}, \alpha)\cap\Phi^{-1}(q))/\GL(\alpha),
\]
where $\mr{Rep}^{\theta\mr{-}s}(\overline{Q}, \alpha)\subset\mr{Rep}^{\theta\mr{-}ss}(\overline{Q}, \alpha)$ denotes the $\theta$-stable locus, one has the following result, which will be crucial in proving that $\mathcal{M}_{q, \theta}(Q, \alpha)$ is a symplectic singularity. Note that, in the above definition the quotient is the usual orbit space, if we replace $\GL(\alpha)$ by $\PGL(\alpha) = \GL(\alpha)/\C^\times$, as a point in the stable locus has trivial stabiliser group under $\PGL(\alpha)$.
\begin{prop}\cite[Theorem 3.4]{yamakawa}\label{p:s-smooth}
$\mathcal{M}^s_{q, \theta}(Q, \alpha)$, if non-empty, is an equidimensional algebraic symplectic manifold and its dimension is $2p(\alpha)$.
\end{prop}
\subsection{Stratification by representation type} An important result proved in \cite[\S 7]{cb-shaw} concerns a natural stratification of the affine variety $\mathcal{M}_{q, 0}(Q, \alpha)$
which parametrises semisimple representations of the algebra $\Lambda^q$. This stratification and its generalisation, proved below, to the case of $\theta$-semistable representations are important in order to understand the singular locus of $\mathcal{M}_{q, \theta}(Q, \alpha)$.\\
Consider $M\in\mr{Rep}^{\theta\mr{-}ss}(\Lambda^q, \alpha)$. Replace it by the
unique $\theta$-polystable representation which is $S$-equivalent to it
 (see the discussion after Theorem \ref{t:king}).
 $M$ is then said to be of \tit{representation type} $\tau=(k_1, \beta^{(1)}; \dots; k_r, \beta^{(r)} )$ if it can be decomposed into the direct sum $M\cong M_1^{k_1}\oplus\dots\oplus M_r^{k_r}$, where $M_i$ is a $\theta$-stable representation of $\Lambda^q$ of dimension vector $\beta^{(i)}$, $i=1, \dots, r$, and $M_i\ncong M_j $ for $i\neq j$. 

\begin{prop}\label{strati} If $\tau$ is a representation type for
  $\Lambda^q$, then the set $C^{\tau}_{q, \theta}(Q, \alpha)$ of
  $\theta$-semistable representations of type $\tau$ is a locally
  closed subset of $\mathcal{M}_{q, \theta}(Q, \alpha)$, which, if
  non-empty, has dimension
  $\sum_{i=1}^r2p(\beta^{(i)})$. $\mathcal{M}_{q, \theta}(Q, \alpha)$
  is the disjoint union of the strata
  $C^{\tau}_{q, \theta}(Q, \alpha)$, where $\tau$ runs over the set of
  representation types that can occur for $\Lambda^q$.
\end{prop}
\begin{proof} First, note that the case when $\theta=0$ is treated in
  \cite{cb-shaw} and proved in Lemma 7.1 therein. For the case when
  $\theta\neq 0$ we use the same arguments. Indeed, the fact that
  $\mm_{q, \theta}(Q, \alpha)$ is a disjoint union of subsets of a
  fixed representation type is immediate from the fact that the
  decomposition of a $\theta$-polystable module into $\theta$-stable modules is
  unique. This, in turn, holds because, for $\theta$-stable modules $M$ and $N$, we have $\dim \Hom(M,N) \leq 1$, with equality if and only if $M$ and $N$ are isomorphic.
  Moreover, to prove that each
  $C_{q, \theta}^{\tau}(Q, \alpha)$ is locally closed and of the
  dimension prescribed by the lemma, one can adapt the proof
  \cite[Theorem 1.3]{cb-geom}: indeed, those arguments can be repeated
  in this case as well, replacing $\mr{Rep}(\overline{Q}, \alpha)$
  with $\mr{Rep}^{\theta\mr{-}ss}(\overline{Q}, \alpha)$,
  $\mu_{\alpha}^{-1}(\lambda)$ with $\Phi^{-1}(q)$, the word
  `(semi)simple' with `$\theta$-(semi)stable' in the proof, and noting
  that everything goes through in the same way because
  $\mr{Rep}^{\theta\mr{-}ss}(\overline{Q}, \alpha)$ is open in
  $\mr{Rep}(\overline{Q}, \alpha)$. The only difference is that, in
  this case, we do not claim irreducibility, since
  $\mr{Rep}^{\theta\mr{-}ss}(\overline{Q},\beta)$ is not known to be
  irreducible.
\end{proof}
We will need also the following property
of $\mr{Rep}^{\theta\mr{-}ss}(\Lambda^q(Q), \alpha)$:
\begin{lem}\label{lem-irred}Every irreducible component of $\mr{Rep}^{\theta\mr{-}ss}(\Lambda^q(Q), \alpha)$ has dimension at least $g_\alpha+2p(\alpha)$ and the set of $\theta$-stable representations form an open subset of $\mr{Rep}(\Lambda^q(Q), \alpha)$ which, if non-empty, is smooth of dimension $g_\alpha+2p(\alpha)$.
\end{lem}
\begin{proof}
For the first part, Lemma 6.2 in \cite{cb-norm} proves the statement in the case when $\theta=0$, of which the above result is a consequence since $\mr{Rep}^{\theta\mr{-}ss}(\Lambda^q(Q), \alpha)$ is an open subset of $\mr{Rep}(\Lambda^q(Q), \alpha)$: indeed, every irreducible component of the former variety is contained in only one irreducible component of the latter and, hence, the dimension estimate holds. For the second part, one just needs to note that, if $X$ is a $\theta$-stable representation, then $\mr{End}(X)=\C$  and, hence, by \cite[Theorem 1.10]{cb-shaw} defines a smooth point of $\mr{Rep}(\Lambda^q(Q), \alpha)$, which implies that it is a smooth point of $\mr{Rep}^{\theta\mr{-}ss}(\Lambda^q(Q), \alpha)$.
\end{proof}
For the proof of the following proposition, apply the strategy carried out in \cite[\S 6, 7]{cb-norm} and \cite[\S 7]{cb-shaw}: the only change is that, in the definition of representation of \tit{top-type}, one has to replace the word `simple' with the word `$\theta$-stable' and use Proposition \ref{strati} instead of \cite[Lemma 7.1]{cb-shaw} and Lemma \ref{lem-irred} instead of \cite[Theorem 1.1]{cb-shaw}. 
\begin{prop}\label{p:inverseimage}
	The inverse image in $\mr{Rep}^{\theta\mr{-}ss}(\Lambda^q(Q), \alpha)$ of the stratum of representations of type $\tau=(k_1, \beta^{(1)};, \dots; k_r, \beta^{(r)})$ has dimension at most $g_\alpha+p(\alpha)+\sum_{l=1}^rp(\beta^{(l)})$.
      \end{prop}

\subsection{The set $\Sigma_{q,\theta}$}\label{ss:sigma}
      As mentioned in the introduction, the dimension vectors of stable representations are closely related to the following combinatorially-defined set,
 which is the multiplicative analogue of the set $\Sigma_\lambda$ introduced by Crawley-Boevey in \cite{cb-deco} and extensively used in \cite{bellamy-schedler}:
\begin{defn}\label{defn-sigma}
Fix $q\in(\C^\times)^{Q_0}$ and $\theta\in \Z^{Q_0}$ and set 
$N_{q,\theta} := \{\alpha \in \N^{Q_0} \mid
q^\alpha = 1, \alpha \cdot \theta = 0\}$. Define
$R^+_{q, \theta}:=R^+ \cap N_{q,\theta}$. Then,
\[
\Sigma_{q,\theta} := \left\{ \alpha \in R_{q,\theta}^+ \ \left| \ p(\alpha) > \sum_{i
	= 1}^r p \left( \beta^{(i)} \right) \textrm{ for any decomposition } \right. \right.\]
\[
\hspace{65mm} 
\left. \alpha =
\beta^{(1)} + \dots + \beta^{(r)} \textrm{ with } r \ge 2, \ \beta^{(i)}
\in R_{q,\theta}^+ \right\}.
\]
When $\theta=0$, we shall use the shortened notation $\Sigma_q$ in place of $\Sigma_{q, 0}$.
\end{defn}
The following is an extension of \cite[Theorem 1.11]{cb-shaw} to the case $\theta \neq 0$.
\begin{prop}\label{rep-equidim2}
	Let $\alpha\in\Sigma_{q, \theta}$. Then, if non-empty,
$\mr{Rep}^{\theta\mr{-}ss}(\Lambda^q(Q), \alpha)$ is a complete intersection in $\mr{Rep}^{\theta\mr{-}ss}(Q,\alpha)$, equidimensional of dimension $g_\alpha+2p(\alpha)$. The locus of $\theta$-stable representations $\mr{Rep}^{\theta\mr{-}s}(\Lambda^q(Q), \alpha)$ is dense inside $\mr{Rep}^{\theta\mr{-}ss}(\Lambda^q(Q), \alpha)$.
\end{prop}
\begin{proof}
This is a direct consequence of Lemma \ref{lem-irred} and Proposition \ref{p:inverseimage}, by the definition of $\Sigma_{q,\theta}$.
\end{proof}
\begin{rem} \label{r:dim-stab}
  Note that a consequence of the above proposition is that, if $\pi:\mr{Rep}^{\theta\mr{-}ss}(\Lambda^q(Q), \alpha)\rightarrow \mm_{q, \theta}(Q, \alpha)$ is the projection map, the image $\mm^s_{q, \theta}(Q, \alpha)$ of the stable locus
  is dense in the moduli space $\mm_{q, \theta}(Q,\alpha)$. As a corollary of this and
  Proposition \ref{p:s-smooth} (or Proposition \ref{strati}), one has that every component of $\mm_{q, \theta}(Q, \alpha)$ has dimension $2p(\alpha)$.
\end{rem}
A useful corollary of the proposition is the following criterion for birationality of the maps $\mm_{q, \theta'}(Q, \beta) \to \mm_{q, \theta}(Q, \beta)$. Together with Lemma \ref{l:varytheta}, this explains that these maps will be resolutions of singularities when the source is smooth.
\begin{coro}\label{c:birational}
Let $\alpha \in \Sigma_{q,\theta}$ be such that $\mm_{q,\theta}(Q,\alpha) \neq \emptyset$. Take $\theta' \geq \theta$ such that every $\theta$-stable representation is $\theta'$-stable. Then, the morphism $\mm_{q, \theta'}(Q, \beta) \to \mm_{q, \theta}(Q, \beta)$
is birational.
\end{coro}
\begin{rem}\label{r:rational stability}
 Note that $\theta' \geq \theta$ is guaranteed if, whenever $\beta < \alpha$, then $\theta \cdot \beta > 0$ implies $\theta' \cdot \beta > 0$. Similarly, the assumption
 that every $\theta$-stable representation is $\theta'$-stable is implied if, for $\beta < \alpha$, then $\theta \cdot \beta < 0$ implies $\theta' \cdot \beta < 0$. To find $\theta'$ satisfying these conditions, first note that they will be satisfied for rational stability conditions $\theta' \in \mathbb{Q}^{Q_0}$ sufficiently close to $\theta$.  But they hold for a  rational vector if and only if they hold for an integral multiple.
 \end{rem}
\begin{proof}
  By Definition \ref{d:tss}, $\mr{Rep}^{\theta'\mr{-}ss}(\Lambda^q, \alpha)$ is a subset of $\mr{Rep}^{\theta\mr{-}ss}(\Lambda^q,\alpha)$, and it is open.
  By assumption, the locus $\mr{Rep}^{\theta\mr{-}s}(\Lambda^q,\alpha)$ is open
  in $\mr{Rep}^{\theta'\mr{-}ss}(\Lambda^q,\alpha)$. It is also dense, since it is
  dense in $\mr{Rep}^{\theta\mr{-}ss}$. Therefore the locus $\mm^{\theta\mr{-}s}_{q, \theta'}(Q, \beta)$ is open and dense in $\mm_{q,\theta'}(Q,\beta)$. As the stable $\GL(\alpha)$-orbits are closed, $\mm^{\theta\mr{-}s}_{q,\theta'}(Q,\beta)$ maps isomorphically to $\mm^s_{q,\theta}(Q,\beta)$. As the latter is dense in $\mm_{q,\theta}(Q,\beta)$, we conclude the desired birationality.
\end{proof}

Using the above results, one can derive an important geometric
property of the moduli space $\mathcal{M}_{q, \theta}(Q,
\alpha)$.
For reasons which are clear in the proof of the proposition, we assume
that a certain codimension estimate holds. As usual, let $\pi:\mr{Rep}^{\theta\mr{-}ss}(\Lambda^q, \alpha)\rightarrow\mm_{q, \theta}(Q, \alpha)$ denote the quotient map.
\begin{lem}\label{l:dimest}
Assume $\alpha\in\Sigma_{q, \theta}$ and let $\tau$ be a stratum. The following inequality holds true:
	\[
	\mr{codim}_{\mr{Rep}^{\theta\mr{-}ss}(\Lambda^q, \alpha)}(\pi^{-1}(C^\tau_{q, \theta}(Q, \alpha)))\geq\frac{1}{2}\mr{codim}_{\mm_{q, \theta}(Q, \alpha)}(C^\tau_{q, \theta}(Q, \alpha)).
	\]
\end{lem}
\begin{proof}
By Proposition \ref{rep-equidim2}, one has that 
\begin{equation*}
\mr{codim}(\pi^{-1}(C^\tau_{q, \theta}(Q, \alpha))) = 
g_\alpha+2p(\alpha)-\dim\pi^{-1}(C^\tau_{q, \theta}(Q, \alpha)).
\end{equation*}
Moreover, from Proposition \ref{p:inverseimage} it follows that 
\[
g_\alpha+2p(\alpha)-\dim\pi^{-1}(C^\tau_{q, \theta}(Q, \alpha))\geq g_\alpha+2p(\alpha)-g_\alpha-p(\alpha)-\sum_{l=1}^rp(\beta^{(l)})=p(\alpha)-\sum_{l=1}^rp(\beta^{(l)}).
\]
On the other hand, by Proposition \ref{strati}, one has that 
\[
p(\alpha)-\sum_{l=1}^rp(\beta^{(l)})=\frac{1}{2}\left(\dim \mm_{q, \theta}(Q, \alpha)-\dim
C^{\tau}_{q, \theta}(Q, \alpha)\right),
\]
which, combined with the above inequality, leads to the desired statement.
\end{proof}
By taking the minimum of these codimensions, we immediately conclude:
\begin{coro}\label{coro-codim} Let $Z$ denote the complement inside $\mm_{q, \theta}(Q, \alpha)$ of the set of $\theta$-stable representations $\mm_{q, \theta}^s(Q, \alpha)$, i.e., $Z$ is the union of all the non-open strata of $\mm_{q, \theta}(Q, \alpha)$. Then, the following inequality holds: 
	\[
\operatorname{codim} \pi^{-1}(Z) \geq \frac{1}{2} \min_{\tau\neq (1, \alpha)}
 \operatorname{codim} C^{\tau}_{q, \theta}(Q, \alpha).
\]

\end{coro}
\begin{prop}\label{quiv-norm}
  Consider $\alpha\in \Sigma_{q, \theta}$ and assume that all strata
  in the non-empty multiplicative quiver variety
  $\mathcal{M}_{q, \theta}(Q, \alpha)$ have codimension at least 4,
  i.e., assume that
\[
\min_{\tau\neq (1, \alpha)}\left(\dim \mm_{q, \theta}(Q, \alpha)-\dim
  C^{\tau}_{q, \theta}(Q, \alpha))\right)\geq 4.
\] 
Then, the variety $\mathcal{M}_{q, \theta}(Q, \alpha)$ is normal.
\end{prop}
\begin{proof}
  The arguments to prove the above statement are analogous to the ones
  used in \cite[Proposition 8.3]{bellamy-schedler}. In particular, we
  shall use a criterion proved by Crawley-Boevey, \cite[Corollary
  7.2]{cb-norm}. We first deal with the case when $\theta=0$ and then
  explain how to adapt the arguments for general $\theta$. When
  $\theta=0$, $\mc{M}_{q, 0}(Q, \alpha)$ is the categorical quotient
  $\mr{Rep}(\Lambda^q, \alpha)/\!\!/\GL(\alpha)$ of an affine variety
  modulo a reductive group. Thus, we only need to show that
  $\mr{Rep}(\Lambda^q, \alpha)$ satisfies Serre's condition ($S_2$)
  and that certain codimension estimates hold true. The first
  condition is ensured by the fact that, by
  \cite[Theorem 1.11]{cb-shaw} (the case $\theta=0$ of Proposition
  \ref{rep-equidim2}), $\mr{Rep}(\Lambda^q, \alpha)$ is a complete
  intersection and, hence, Cohen-Macaulay, which indeed implies
  condition ($S_2$). Now, denote by $S$ the open subset
  $S\subset \mc{M}_{q, 0}(Q, \alpha)$ of simple
  representations, which is non-empty by our assumption. $S$ is
  contained in the smooth locus and hence is normal. Moreover, let $Z$
  denote its complement in $\mc{M}_{q, 0}(Q, \alpha)$ and denote with
  $\pi: \mr{Rep}(\Lambda^q, \alpha)\rightarrow \mulquiv$ the quotient
  map; then, by Corollary \ref{coro-codim}, one has
	\[
	\dim \mr{Rep}(\Lambda^q, \alpha)-\dim\pi^{-1}(Z)\geq
        \frac{1}{2}\min_{\tau\neq (1, \alpha)}(\dim \mulquiv-\dim
        C^{\tau}_{q, 0}(Q, \alpha)),
	\]
	and the right hand side is greater or equal than two by
        assumption. Thus, all the hypotheses of \cite[Corollary
        7.2]{cb-norm} are satisfied and we can conclude that
        $\mulquiv$ is normal. For the case when $\theta\neq 0$,
        keeping in mind that normality is a local property, we fix a
        point $x\in \mc{M}_{q, \theta}(Q, \alpha)$ and aim at proving
        normality at $x$. This is achieved by choosing an open
        neighbourhood $V$ of $x$ such that the restriction to
        $\pi^{-1}(V)$ of the projection morphism
        $\pi^{-1}(V)\rightarrow V$ is an affine quotient (note that
        this can be done thanks to the properties of the GIT
        construction). One can now repeat the same arguments as for
        the $\theta=0$ case, noting that, by Proposition \ref{strati},
        the estimates above hold true also in this more general
        setting: being Cohen-Macaulay is a local statement and thus
        the previous part of the proof ensures that $\pi^{-1}(V)$,
        which is open in $\mr{Rep}(\Lambda^q, \alpha)$, satisfies this
         property. Moreover, defining $S_{\theta}$ to be the subset
        of $V$ of $\theta$-stable representations, then one may
        proceed as in the first part of the proof to obtain the
        desired conclusion.
\end{proof}
\begin{rem}
  In the next sections, we will examine some cases in which the technical
  assumption in the previous result is satisfied, thus giving explicit
  examples of when $\mm_{q, \theta}(Q, \alpha)$ is normal.
\end{rem}
Finally, for the sequel, we will have to consider the following analogue of divisibility:
\begin{defn}
  A dimension vector $\alpha \in N_{q,\theta}$ is said to be
  $q$-indivisible if $\frac{1}{m}\alpha \notin N_{q, \theta}$ for all
  $m\geq 2$. Equivalently, for $\alpha=m\beta$ and $\beta$
  indivisible, then $q^\beta$ is a primitive $m$-th root of unity.
\end{defn}

\section{Punctured character varieties as multiplicative quiver varieties}\label{section-char}
In this section, we explain how it is possible to realise certain character varieties as particular examples of multiplicative quiver varieties by considering quivers of special type, the so-called crab-shaped quivers. Such character varieties parametrise representations of the fundamental group of a compact Riemann surface with a finite number of punctures, where the monodromies at closed loops around such punctures are fixed to lie in (the closure of) certain conjugacy classes.  We use the language of quiver Riemann surfaces introduced by Crawley-Boevey in \cite{cb-monod}. Moreover, in what follows, we shall adopt the term \tit{punctured character variety} to refer to the character variety of a Riemann surface with punctures.

Fix a connected compact Riemann surface $X$ of genus $g\geq 0$, let $S=\{p_1, \dots, p_k\}\subset X$ be the set of punctures and fix a tuple $\mathcal{C}=(\mathcal{C}_1, \dots, \mathcal{C}_k)$ of conjugacy classes $\mathcal{C}_i\subset GL_n(\C), i=1, \dots, k$. Recall that the fundamental group $\pi_1(X\setminus S)$ of the punctured surface $X\setminus S$ admits the following presentation:
\[
\pi_1(X\setminus S)=\langle a_1, \dots, a_g, b_1, \dots, b_g, c_1, \dots, c_k\ |\ [a_1, b_1]\cdot\cdots\cdot[a_g, b_g]c_1\cdot\cdots\cdot c_k=1\rangle,
\]
where $[a, b]=aba^{-1}b^{-1}$ denotes the commutator. Note that the generators $c_1, \dots, c_k$ represent homotopy classes of closed loops around the punctures, in the same free homotopy classes as small counter-clockwise loops around the punctures. Thus, a representation of $\pi_1(X\setminus S)$ whose monodromies about the punctures are in the conjugacy classes $\mc{C}_i$ is given by a tuple of matrices $(A_1, \dots, A_g, B_1, \dots, B_g, C_1, \dots, C_k)\in GL_n(\C)^{2g}\times \mathcal{C}_1\times\dots\times\mathcal{C}_k$, satisfying the relation
\[
\prod_{i=1}^g[A_i, B_i]\prod_{j=1}^kC_j=I.
\] 
Given the above, from the fact that isomorphic representations correspond to conjugate matrices, one has that the character variety $\mathcal{X}(g, k, \overline{\mathcal{C}})$ associated to the pair $(X, S)$ and monodromies lying in the conjugacy classes fixed above is isomorphic to the affine quotient 
\[
\begin{split}
\mathcal{X}(g, k, \overline{\mathcal{C}}):=\{
	(A_1, \dots, A_g, B_1, \dots, B_g, C_1, \dots, C_k)\in GL_n(\C)^{2g}\times \overline{\mathcal{C}}_1\times\dots\times\overline{\mathcal{C}}_k |
\\ \prod_{i=1}^g[A_i, B_i]\prod_{j=1}^kC_j=I\}/\!\!/GL_n(\C).
	\end{split}
\]
\begin{rem}
Note that the  closures $\overline{\mc{C}}_i$ are affine varieties and hence the quotient is indeed that of an affine variety by an algebraic group.
\end{rem}
We shall now explain how to realise the variety $\mathcal{X}(g, k, \overline{\mathcal{C}})$ as an open subset of a multiplicative quiver variety, using an equivalence of categories proved in \cite{cb-monod}. As mentioned above, such a correspondence holds when one considers the so-called \tit{crab-shaped} quivers (called ``comet-shaped'' in \cite{h-l-rv-1}), i.e., quivers such that there exists a vertex $v$ satisfying the following condition: the set of arrows 
is formed by loops at $v$ and a finite number of \tit{legs} ending at $v$. See Figure \ref{figure1}. A \emph{star-shaped} quiver is a crab-shaped quiver with no loops.

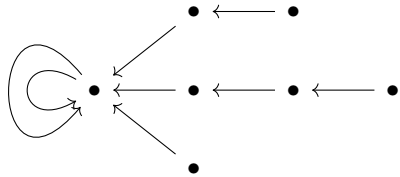
\begin{figure}[!ht]\label{figure1}
\centering
\begin{tikzcd}
&\arrow[ld] \bullet &\arrow{l}  \bullet \\
\arrow[out=150, in=210,loop,distance=1cm]{}{} 
\arrow[out=130, in=230,loop,distance=2cm]{}[swap]{}  \bullet &  \arrow{l}{}  \bullet & \arrow{l}{}  \bullet& \arrow{l}{}\bullet \\
&\arrow[lu] \bullet \\ 
\end{tikzcd}
	\caption{A crab-shaped quiver with 2 loops and 3 legs, of length 2, 3 and 1 respectively.}
\end{figure}
For the remainder of this section, the following notation will be used: $g$, for the number of loops around the central vertex; $k$ for the number of legs and $l_i$, for $i=1, \dots, k$, for the length of the $i$-th leg. As we shall see, $g$ contains the information regarding the genus of the surface, while the integers $k$ and $l_i$ encode information about the (prescribed) conjugacy classes of the monodromies of the loops around the punctures.
\begin{defn}\cite[\S2]{cb-monod} A \tit{Riemann surface quiver} $\Gamma$ is a quiver whose set of vertices has the structure of a Riemann surface $X$ with finitely many connected components. $\Gamma$ is said to be \tit{compact} if $X$ is compact. A point $p\in X$ is called \tit{marked} if it is a head or a tail of an arrow of $\Gamma$. 
\end{defn}   
\begin{defn}Given a Riemann surface quiver $\Gamma$, the \tit{component quiver} $[\Gamma]$ of $\Gamma$, is the quiver whose set of vertices is the set of connected components of $\Gamma$ and arrows given by $[a]:[p]\rightarrow[q]$ for any arrow $a:p\rightarrow q$, where $p$ and $q$ are points of $X$ and $[p]$ denotes the connected component of $X$ containing $p$.
\end{defn}
\begin{rem}
Although, by definition, there are in general infinitely many vertices, we will consider (Riemann surface) quivers with finitely many arrows.
\end{rem}
Following closely \cite[\S5, \S 8]{cb-monod}, starting from a Riemann surface quiver $\Gamma$, it is possible to define two categories of representations, $\mr{Rep}_{\sigma}(\pi(\Gamma))$ and $\mr{Rep}\ \Lambda^q([\Gamma])$, whose equivalence is the key point to proving the correspondence between multiplicative quiver varieties and punctured character varieties. \\
Fix a quiver Riemann surface $\Gamma$ and let $\{X_i\}_{i\in I}$ the set of connected components of the underlying Riemann surface $X$. For each $i\in I$ let $D_i$ be the set of marked points of $\Gamma$ contained in $X_i$. Moreover, let $D=\cup_iD_i$: fix $\sigma \in (\C^\times)^{D}$, $b_i\in X_i\setminus D_i$ and, for each $p\in D_i$ fix a loop $l_p\in \pi_1(X_i\setminus D_i, b_i)$ around $p$. 

$\mr{Rep}_\sigma\pi(\Gamma)$ is defined to be the category whose objects are given by collections $(V_i, \rho_i, \rho_a, \rho^*_a)$ consisting of representations $\rho_i:\pi_1(X_i\setminus D_i, b_i)\rightarrow \GL(V_i)$, for $i\in I$ and linear maps $\rho_a:V_i\rightarrow V_j$ and $\rho_a^*:V_j\to V_i$
for each arrow $a:p\to q$ in $\Gamma$, where $X_i = [p]$ and $X_j = [q]$,
satisfying
\[
\sigma_p^{-1} \rho_i(\ell_p)^{-1} =  1_{V_i} + \rho_a^* \rho_a
\quad\text{and}\quad
\sigma_q \rho_j(\ell_q) = 1_{V_j} + \rho_a \rho_a^*
\]
and whose morphisms are the natural ones.\\
Consider the component quiver $[\Gamma]$ and define $Q$ to be the
quiver obtained from $[\Gamma]$ by adjoining $g_i$ loops at each
vertex $i$, where $g_i$ is the genus of $X_i$. Moreover, define
$q\in (\C^\times)^I$ by $q_i=\prod_{p \in D_i} \sigma_p$.
We define $\mr{Rep}\ \Lambda^q([\Gamma])'$ to be the category of representations of the multiplicative preprojective algebra $\Lambda^q(Q)$ in which the linear maps representing the added loops in $Q$ (but not their reverse loops in $\overline{Q}$) are invertible. 
\begin{lem}\cite[Proposition 2]{cb-monod} \label{l:monod}
There is an equivalence of categories 
\[\mr{Rep}_{\sigma}\pi(\Gamma) \simeq \mr{Rep}\ \Lambda^q([\Gamma])'.
\]
This induces a $\GL(\alpha)$-equivariant isomorphism of affine algebraic varieties,
\[
\mr{Rep}_\sigma(\pi(\Gamma),\alpha) \overset{\sim}{\longrightarrow} \mr{Rep}(\Lambda^q([\Gamma])', \alpha),
\]
defined as the collections of representations with $V_i = \C^{\alpha_i}$ for all $i$.
\end{lem}
\begin{proof} 
  The first statement is precisely \cite[Proposition 2]{cb-monod}. For
  the second, 
  both
  $\mr{Rep}_\sigma(\pi(\Gamma),\alpha)$ and
  $\mr{Rep}(\Lambda^q([\Gamma])', \alpha)$ are acted upon by the group
  $\GL(\alpha)$ and the above equivalence of categories implies that
  there is a $\GL(\alpha)$-equivariant bijection as desired.  Moreover,
  $\mr{Rep}_\sigma(\pi(\Gamma),\alpha)$ and
  $\mr{Rep}(\Lambda^q([\Gamma]', \alpha))$ are easily seen to be
  affine algebraic varieties, defined as tuples of matrices satisfying
  certain polynomial relations, with certain polynomials inverted. To
  see that the above map is a $\GL(\alpha)$-equivariant algebra
  isomorphism, observe that the proof of \cite[Proposition
  2]{cb-monod} uses explicit invertible polynomial formulae.
\end{proof}

In order to explain how the above equivalence of categories implies the correspondence between character varieties and preprojective algebras, we shall explain how it is possible to encode the datum of a number of conjugacy classes into a star-shaped quiver. We follow \cite[\S 8]{cb-shaw} and \cite[\S 2]{cb-indec}: fix $k$ conjugacy classes $\mc{C}_1,\dots, \mc{C}_k$ in $GL_n(\C)$, for $k\geq 1$. We can encode the datum of such conjugacy classes in a combinatorial object as follows: take $A_i\in \mathcal{C}_i$ and let $w_i\geq 1$ be the degree of its minimal polynomial, for $i=1, \dots, k$; choose elements $\xi_{ij}\in \C^\times$, $1\leq i\leq k, 1\leq j\leq w_i$, such that 
\[
(A_i-\xi_{i1}I)\cdot\dots\cdot(A_i-\xi_{iw_i}I)=0.
\]
The closure of the conjugacy class $\mathcal{C}_i$ is then determined by the ranks of the partial products 
\[
\alpha_{ij}=\mr{rank}(A_i-\xi_{i1}I)\cdot\dots\cdot (A_i-\xi_{ij}I),
\]
for $A_i\in \mathcal{C}_i$ and $1\leq j\leq w_i-1$. In addition, if we set $\alpha_0=n$, we get a dimension vector $\alpha$ for the following quiver $Q_w$
\setlength{\unitlength}{1.5pt}
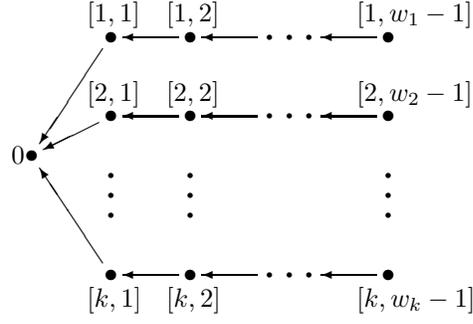
\[
\begin{picture}(110,80)
\put(10,40){\circle*{2.5}} \put(30,10){\circle*{2.5}}
\put(30,50){\circle*{2.5}} \put(30,70){\circle*{2.5}}
\put(50,10){\circle*{2.5}} \put(50,50){\circle*{2.5}}
\put(50,70){\circle*{2.5}} \put(100,10){\circle*{2.5}}
\put(100,50){\circle*{2.5}} \put(100,70){\circle*{2.5}}
\put(28,67){\vector(-2,-3){16}}
\put(27,48.5){\vector(-2,-1){14}}
\put(28,13){\vector(-2,3){16}}
\put(47,10){\vector(-1,0){14}}
\put(47,50){\vector(-1,0){14}}
\put(47,70){\vector(-1,0){14}}
\put(67,10){\vector(-1,0){14}}
\put(67,50){\vector(-1,0){14}}
\put(67,70){\vector(-1,0){14}}
\put(97,10){\vector(-1,0){14}}
\put(97,50){\vector(-1,0){14}}
\put(97,70){\vector(-1,0){14}}
\put(70,10){\circle*{1}}
\put(75,10){\circle*{1}} \put(80,10){\circle*{1}}
\put(70,50){\circle*{1}} \put(75,50){\circle*{1}}
\put(80,50){\circle*{1}} \put(70,70){\circle*{1}}
\put(75,70){\circle*{1}} \put(80,70){\circle*{1}}
\put(30,25){\circle*{1}} \put(30,30){\circle*{1}}
\put(30,35){\circle*{1}} \put(50,25){\circle*{1}}
\put(50,30){\circle*{1}} \put(50,35){\circle*{1}}
\put(100,25){\circle*{1}} \put(100,30){\circle*{1}}
\put(100,35){\circle*{1}} \put(5,38){0} \put(24,2){$[k,1]$}
\put(24,54){$[2,1]$} \put(24,74){$[1,1]$} \put(44,2){$[k,2]$}
\put(44,54){$[2,2]$} \put(44,74){$[1,2]$} \put(92,2){$[k,w_k-1]$}
\put(92,54){$[2,w_2-1]$} \put(92,74){$[1,w_1-1]$}
\end{picture}
\]
Now, for every $i \in \{1, \dots, k\}$, let $t_i$ be a non-negative integer, and define a Riemann surface quiver $\Gamma$ as follows: its underlying Riemann surface $X$ is given by the disjoint union
\[
X=X_0\sqcup\bigsqcup_{i\in\{1, \dots, k\},\\ j\in\{1, \dots, t_i\}}\mathbb{P}^1_{i, j},
\]
where $X_0$ is an arbitrary closed Riemann surface of genus $g$ (the choice
does not matter),
and $\mathbb{P}^1_{i, j}$ is simply a copy of $\mathbb{P}^1$ for the index $(i, j)$. For each pair of indices $(i, j)$ fix a point $p_{i, j}\in \mathbb{P}^1_{i, j}$, and, for $i=1, \dots, k$, fix distinct points $p_i\in X_0$. Define $D$ as before to be 
\[
D=\{p_{i, j}\}\cup\{p_l\}.
\]
The arrows of $\Gamma$ are listed as follows: 
\begin{itemize}
\item $a_{i, 0}:p_{i, 1}\rightarrow p_i$, for $i=1, \dots, k$;
\item $a_{i, j}:p_{i, j+1}\rightarrow p_{i, j}$, for $i=1, \dots, k$, $j=1, \dots, t_{i}-1$.
\end{itemize}
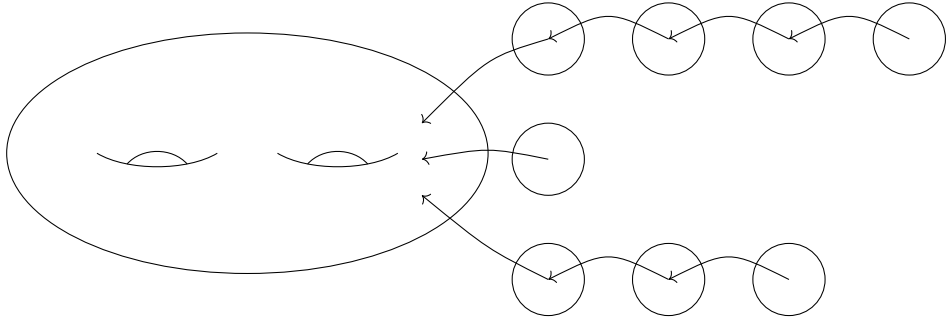
\begin{figure}[h]\label{fig:3-stab}
	\centering
	\begin{tikzpicture}[scale=.8]
\draw (2,-3) circle (.6cm);
\draw (4,-3) circle (.6cm);
\draw (6,-3) circle (.6cm);
\draw (8,-3) circle (.6cm);

\draw[<-] (-.1,-4.4) .. controls (1,-3.3)..(2,-3);
\draw[<-] (2,-3) .. controls (3,-2.5)..(4, -3);
\draw[<-] (4,-3) .. controls (5,-2.5)..(6,-3);
\draw[<-] (6,-3) .. controls (7,-2.5)..(8,-3);

\draw (-3,-4.9) ellipse (4cm and 2cm);

\draw (2,-5) circle (.6cm);

\draw (2,-7) circle (.6cm);
\draw (4,-7) circle (.6cm);
\draw (6,-7) circle (.6cm);

\draw[<-] (-.1,-5.6) .. controls (1,-6.5)..(2,-7);
\draw[<-] (2,-7) .. controls (3,-6.5)..(4, -7);
\draw[<-] (4,-7) .. controls (5,-6.5)..(6,-7);

\draw (-5.5,-4.9) .. controls (-5, -5.2) and (-4, -5.2) .. (-3.5,-4.9);
\draw (-5, -5.08) .. controls (-4.75, -4.8) and (-4.25, -4.8) ..(-4, -5.08);

\draw (-2.5,-4.9) .. controls (-2, -5.2) and (-1, -5.2) .. (-.5,-4.9);
\draw (-2, -5.08) .. controls (-1.75, -4.8) and (-1.25, -4.8) ..(-1, -5.08);

\draw[<-] (-.1,-5) .. controls (1,-4.8)..(2,-5);

\end{tikzpicture}
\caption{An example of a Riemann surface quiver associated with a tuple of conjugacy classes.}
\end{figure} 
Unfolding the definition for the objects of the category $\mr{Rep}_{\sigma}\pi(\Gamma)$, one has that for such a Riemann surface quiver $\Gamma$ these are representations 
\[
\rho: \pi_1(X_0\setminus \{p_1, \dots, p_k\})\rightarrow \GL(V),
\]
and linear maps
$$\rho_{i, 0}: V_{i, 1}\rightarrow V,\ \  \rho_{i, j}:V_{i, j+1}\rightarrow V_{i, j}
$$
and 
$$\rho_{i, 0}^*: V\rightarrow V_{i, 1},\ \  \rho_{i, j}^*:V_{i, j}\rightarrow V_{i, j+1},
$$
for $i=1\,\dots, k$ and $j=1, \dots, t_i-1$, such that, if $l_i\in\pi_1(X_0\setminus\{p_1, \dots, p_k\})$ is the loop around $p_i$, $i=1, \dots, k$, the linear automorphism $\rho(l_i)$ satisfies the condition 
\[
\sigma_i\rho(l_i)=\mr{1}_V+\rho_{i, 0}\rho_{i, 0}^*
\]
and the linear maps $\rho_{i, j}$ and $\rho_{i, j}^*$ satisfy the equations: 
\begin{equation*}
\begin{array}{c}
\sigma_{i, j+1}^{-1}1_{V_{i, j+1}}=1_{V_{i, j+1}}+\rho_{i, j}^*\rho_{i, j},\\
\sigma_{i, l}1_{V_{i, l}}=1_{V_{i, l}}+\rho_{i, l}\rho^*_{i, l},
\end{array}
\end{equation*}
for $i=1, \dots, k$, $j=1, \dots, t_i-1$ and $l=1, \dots, t_i$, which, setting $j=l-1$ and summing the equations involving operators on the same space $V_{i, l}$, can be rewritten as 
\begin{equation*}
\begin{array}{c}
\rho(l_i)=\sigma_{i}^{-1}1_V+\sigma_i^{-1}\rho_{i, 0}\rho^*_{i, 0},\\
\rho^*_{i, j-1}\rho_{i, j-1}-\rho_{i, j}\rho^*_{i, j}=(\sigma_{i, j}^{-1}-\sigma_{i, j})1_{V_{i,j}},\\
\sigma_{i,t_i}1_{V_{i, t_i}}=1_{V_{i, t_i}}+\rho_{i, t_i}\rho^*_{i, t_i},
\end{array}
\end{equation*}
for $i=1, \dots, k$ and $j=1, \dots, t_i-1$. Now, we specialise to the case $t_i=w_i-1$ and assume that $\dim V_{i, j}=\alpha_{i, j}$ and $\dim V=n$, where $w_i$ and $\alpha_{i, j}$ are defined as before. Through some simple algebraic computations, it is possible to see that, given $\xi_{i, j}$ as before, it is possible to find corresponding $\sigma_{i, j}$, defined as 
\begin{equation*}
\begin{array}{c}
\sigma_0=\frac{1}{\prod_{i=1}^k \xi_{i, 1}},\ \ 
\sigma_{i, j}=\frac{\xi_{i, j}}{\xi_{i, j+1}},
\end{array}
\end{equation*}
 such that the above sets of equations can be rewritten in terms of linear operators $\phi_{i, j}$ and $\psi_{i, j}$, for $i=1, \dots, k$ and $j=i, \dots, w_{i}-1$,
\[
V \overset{\phi_{i1}}{\underset{\psi_{i1}}{\rightleftarrows}} V_{i1} \overset{\phi_{i2}}{\underset{\psi_{i2}}{\rightleftarrows}} V_{i2}
\overset{\phi_{i3}}{\underset{\psi_{i3}}{\rightleftarrows}} \dots
\overset{\phi_{i, w_{i-1}}}{\underset{\psi_{i, w_{i-1}}}{\rightleftarrows}} V_{i,w_i-1}
\]
satisfying
\begin{equation*}\label{char-preproj}
\begin{array}{c}
\rho(l_i) - \psi_{i1} \phi_{i1}= \xi_{i1} \, 1_V
\\
\phi_{ij} \psi_{ij} - \psi_{i,j+1}\phi_{i,j+1}= (\xi_{i,j+1}-\xi_{ij})\, 1_{V_{ij}}\qquad (1\le j < w_i-1)
\\
\phi_{i,w_i-1} \psi_{i,w_i-1}= (\xi_{i,w_i} - \xi_{i,w_i-1})\, 1_{V_{i,w_i-1}},
\end{array}
\end{equation*}
which, by \cite[Theorem 2.1]{cb-indec}, implies that $\rho(l_i)$ lies in the closure of the conjugacy class $\mathcal{C}_i$, $i=1, \dots, k$. In fact, this theorem says that this is a necessary and sufficient condition; thus, given a representation 
\[
\rho: \pi_1(X_0\setminus \{p_1, \dots, p_k\})\rightarrow \GL(V),
\]
where $\dim V=n$ and $\rho(l_i)\in \overline{\mathcal{C}_i}$, for prescribed conjugacy classes $\mathcal{C}_1, \dots, \mathcal{C}_k$ in $GL_n(\C)$, we can find linear maps $\rho_{i, j}$ and $\rho_{i, j}^*$ and vector spaces $V_{i, j}$ of dimension $\alpha_{i, j}$ as above, such that the tuple $(V, \rho, V_{i, j}. \rho_{i, j}, \rho^*_{i, j})$ is an object of the category $\mr{Rep}_{\sigma}\pi(\Gamma)$. Then, combining this with Lemma \ref{l:monod}, one has the following result. 
\begin{thm}\label{t:iso-mult-char}
There is an isomorphism between the character variety $\mathcal{X}(g, k, \overline{\mathcal{C}})$ and the affine quotient $\widetilde{\mm}_{q, 0}([\Gamma], \alpha):=\mr{Rep}\ \Lambda^q([\Gamma], \alpha)'\ \!/\!/\GL(\alpha)$.
\end{thm}
\begin{rem}\label{r:imc-open}
  From its definition, one can see that $\mr{Rep}\ \Lambda^q([\Gamma], \alpha)'$ is an open affine $\GL(\alpha)$-invariant subset of $\mr{Rep}(\Lambda^q[\Gamma], \alpha)$ which is obtained by inverting certain $\GL(\alpha)$-invariant functions (the determinants of the linear transformations corresponding to loops of the undoubled quiver at the node). Since the quotient in the affine case, with $G$ reductive, is obtained by passing to $G$-invariant functions, i.e., $\Spec B /\!/ G = \Spec B^G$, we deduce that
   the affine quotient $ \widetilde{\mm}_{q, 0}([\Gamma], \alpha)$ can be identified with an open subset of the multiplicative quiver variety $\mm_{q, 0}([\Gamma], \alpha)$. This is important because, as outlined in the following section, in order to show the non-existence of symplectic resolutions we prove that certain such varieties contain an open subset which is factorial and terminal.   
\end{rem}
\begin{rem}
  We note that, in the star-shaped case, this result follows from
  \cite[Section 8]{cb-shaw}. Moreover, in the general case, Yamakawa
  proves a similar result to the one obtained in this section in the
  language of local systems on punctured surfaces, see \cite[Theorem
  4.14]{yamakawa} for more details.
\end{rem}

\section{Singularities of multiplicative quiver varieties} \label{section-sing} 
Throughout this section, which is
devoted to the study of the singularities
$\mm_{q, \theta}(Q, \alpha)$ and to the proof of Theorem \ref{main-result}, we use the notation introduced in Section \ref{mult-quiv-var}. 

In order to carry out this analysis, in Section \ref{ss:sing-alpha},
we describe the singular locus of the varieties in question. As one
might expect, for $\alpha \in \Sigma_{q,\theta}$, this is given by the
locus of strictly semistable representations. This follows because
these varieties are Poisson, the stable locus is symplectic and
smooth, and its complement has codimension at least two. Since a
generically non-degenerate Poisson structure on a smooth variety can
only degenerate along a divisor (the vanishing locus of the Pfaffian
of the Poisson bivector), we conclude that the entire smooth locus is
non-degenerate. Since the strictly semistable locus is degenerate, it
must therefore be singular.
Moreover, in the case where $\alpha$ is $q$-indivisible, a
symplectic resolution can be obtained by varying $\theta$, by Lemma
\ref{l:varytheta} and Corollary \ref{c:birational} (and Remark \ref{r:rational stability}). These arguments, spelled out below, prove the
second statement of Theorem \ref{main-result}.

In Section \ref{ss:na}, we complete the proof of Theorem
\ref{main-result} by considering strata of representation type
$\nu \beta$, where $\alpha = n\beta$ and $\nu$ is a partition of $n$.
We compute their codimension. As a consequence, taking $\beta$ to be
$q$-indivisible, for suitable $\theta' \geq \theta$,
$\mm_{q, \theta'}(Q, \alpha)$ has singularities in codimension
$\geq 4$. Hence by Flenner's theorem \cite{flenner}, its normalisation
is a symplectic singularity, which proves the first statement of
Theorem \ref{main-result}.  Finally, we show, using Drezet's criterion
of factoriality, that the singularities along most strata $\nu \beta$
are factorial and terminal. This proves the final statement of Theorem
\ref{main-result}. Note that Section \ref{ss:na} closely follows
\cite{bellamy-schedler}, where the analogous strata are considered for
ordinary quiver varieties.

\subsection{Singular locus of $\mm_{q, \theta}(Q, \alpha)$ for $\alpha \in \Sigma_{q,\theta}$}\label{ss:sing-alpha}
Before proving the main statement, we need a well-known result which is valid for any variety endowed with a Poisson structure. 
\begin{lem}\label{l:degen-div}
  Let $X$ be a smooth variety and $\pi \in \wedge^2TX$ a generically non-degenerate Poisson
  bivector. Let $D$ the degenerate locus of $\pi$. Then, if nonempty, $D$ is a
  divisor.
\end{lem}
\begin{proof}
  By generic non-degeneracy, $\dim X$ has to be even, therefore $\dim X= 2d$. Define
  the top polyvector field $\gamma=\wedge^d\pi$. Then, $D$ coincides with
  the zero locus of $\gamma$. On the other hand, $\gamma$ is a section
  of a line bundle and, therefore, its zero locus is a divisor (if nonempty).
\end{proof}
This implies the following criterion for the singular locus of a Poisson variety:
\begin{coro}\label{c:sing-crit}
Let $X$ be a Poisson variety which is smooth and symplectic in the complement of a closed Poisson subvariety $Z \subseteq X$ which has codimension at least two everywhere. Then $Z$ equals the set-theoretic singular locus of $X$.
\end{coro}
\begin{proof}
Suppose for a contradiction that $X$ is smooth at a point $z \in Z$. Since $Z$ is a closed Poisson subvariety, the Poisson structure of $X$ is degenerate at $z$.  It follows from Lemma \ref{l:degen-div} that the degeneracy locus of $X$ has codimension $1$ at $z$.  However, this locus is contained in $Z$, which has codimension at least two at $z$. This is a contradiction.
\end{proof}

\begin{prop}\label{sing-quiv}Let $\alpha \in \Sigma_{q,\theta}$. The smooth locus of $\mm_{q, \theta}(Q, \alpha)$ is $\mm^s_{q, \theta}(Q, \alpha)$.
\end{prop}
\begin{proof} By Proposition \ref{p:s-smooth} $\mm^s_{q, \theta}(Q, \alpha)$ is smooth and symplectic.
Let $Z$ be the complement. It is
  the union of all the
  non-open strata of $\mm_{q, \theta}(Q, \alpha)$.  There are finitely many
and these all have purely even dimension; hence $Z$ has codimension at least two everywhere (as $\mm^s_{q,\theta}$ is dense and it has purely even dimension, $2p(\alpha)$).  Furthermore, we claim that $Z$ is a Poisson subvariety, i.e., all Hamiltonian vector fields are tangent to it.  
  Indeed, Hamiltonian vector fields descend from $\GL(\alpha)$-invariant Hamiltonian vector fields on representation varieties. These integrate to formal automorphisms which commute with the $G$-action, which hence preserve the stratification by conjugacy classes of stabiliser.  Therefore, the hypotheses of Corollary \ref{c:sing-crit} are satisfied, and the statement follows.
\end{proof}
\begin{rem} It is reasonable to ask if a stronger statement is true,
  which makes sense for general $\alpha$: are the connected components
  of the representation type strata the symplectic leaves?
  Equivalently, do the Hamiltonian vector fields span the tangent
  spaces to the representation type strata?  If so, then (a)
  $\mm_{q,\theta}(Q,\alpha)$ has finitely many symplectic leaves, and
  (b) the representation type strata are all smooth.  The converse
  statement also holds: if a stratum is smooth and it is a union of
  finitely many symplectic leaves, its Poisson structure must be
  non-degenerate outside a locus of codimension at least two.  So Lemma
  \ref{l:degen-div} implies that it is actually non-degenerate.

  Let us comment briefly on conditions (a) and (b). First, if
  $\mm_{q,\theta}(Q,\alpha)$ is a symplectic singularity, it has
  finitely many symplectic leaves, by \cite[Theorem 2.5]{Kalss}. Next,
  for a representation type
  $\tau=(k_1, \beta^{(1)}; \dots; k_r, \beta^{(r)} )$, the direct sum map produces a surjection
  $(\mm_{q,\theta}^s(Q,\beta^{(1)}) \times \cdots \times
  \mm_{q,\theta}^s(Q,\beta^{(r)}))^{\text{dist}} \to
  C^{\tau}_{q,\theta}(Q,\alpha)$ with smooth source, where the
  $\text{dist}$ refers to the open subset where the elements of the
  $i$-th and $j$-th factors are unequal for all distinct $i$ and $j$. It seems
  reasonable to expect this to be a covering, in which case the
  stratum is smooth.
\end{rem}
\subsection{Generalities on symplectic singularities}
In order to prove the first statement of Theorem \ref{main-result}, we need a criterion for the normalisation of a variety to have symplectic singularities. This   is an extension
of \cite[Lemma 6.12]{bellamy-schedler}, using \cite[Theorem 1.5]{Kalnpa}:
\begin{prop}\label{prop-symp-sing}
Let $X$ be a Poisson variety and assume that $\pi: Y\rightarrow X$ is a proper birational Poisson morphism from a variety $Y$ with symplectic singularities. Then the normalisation $X'$ of $X$ has symplectic singularities. Moreover, the induced map $\pi: Y \to X'$ is Poisson.
\end{prop}
\begin{proof}
  In \cite[Lemma 6.12]{bellamy-schedler}, the result is proved under the assumption that $X$ is in fact normal.  To conclude the lemma from this result, we
  may apply \cite[Corollary 1.4, Theorem 1.5]{Kalnpa}. By these results (and their proofs), given a Poisson variety $X$, the normalisation $X'$ has a unique Poisson structure such that the normalisation map $\nu: X' \to X$ is a Poisson morphism. The map $\pi$ factors through $\nu$, and the induced map $\pi': Y \to X'$ must be Poisson, since the Poisson bracket on $\OO_{X'}$ is the unique extension of the Poisson bracket on $\OO_X$ to a biderivation $\OO_{X'} \times \OO_{X'} \to \OO_Y$.  Then the fact that $X'$ has symplectic singularities follows from \cite[Lemma 6.12]{bellamy-schedler}.
\end{proof}
\begin{rem}\label{r:pss}
  For convenience, we will apply this result even in the case where $\pi$ is a symplectic resolution. However, in this case, the statement follows from definitions, without really requiring the results of \cite{bellamy-schedler,Kalnpa}, as follows.
  The map $\pi: Y \to X$ factors through $\pi': Y \to X'$, which induces on $X'$ a unique Poisson structure such that $\pi'$ is Poisson; as $Y$ is non-degenerate and its symplectic form is pulled back from $X'$, $X'$ must also be non-degenerate on the smooth locus. By definition, $X'$ is then a symplectic singularity.  Since $\pi$ is dominant,
  the Poisson structure on $X$ is uniquely determined from the one on $Y$,
  and must be the one obtained from $X'$ via the inclusion $\OO_X \to \OO_{X'}$. This proves the last statement.
\end{rem}  
  \begin{rem}
  Actually, in the above proposition, the biconditional holds:
  $X$ has symplectic singularities if and only if $Y$ does. Moreover, one can generalise to the case where $X$ is a non-reduced Poisson scheme: in this case, the map $\pi$ factors through the reduced subvariety $X^{\text{red}}$, which is canonically Poisson by
  \cite[Corollary 1.4]{Kalnpa}.
  \end{rem}

\subsection{The $q$-indivisible case}
We now prove the second statement of Theorem \ref{main-result}. Suppose that
$\beta \in \Sigma_{q,\theta}$ is $q$-indivisible.

First suppose that $\beta$ is real. In this case, 
by \cite[Theorem 2.1]{cb-shaw}, $\Lambda^q(Q)$ admits a simple rigid representation $X$ and any other representation $Y$ of the same dimension must be isomorphic to $X$, which means that the variety $\mm_{q, \theta}(Q, \beta)$ is a point. So there is nothing to prove.

Next suppose that $\beta$ is imaginary. In this case one may proceed as follows: by choosing a generic stability parameter $\theta'\geq \theta$, there is a projective symplectic resolution
\[
\pi: \mm_{q, \theta'}(Q, \beta)\longrightarrow\mm_{q, \theta}(Q, \beta).
\]
Indeed, by \cite[Proposition 3.5]{yamakawa}, for $\theta'$ generic, the stable locus $\mm^s_{q, \theta'}(Q, \beta)$, which is smooth, coincides with the semistable locus. Hence we can find $\theta'\geq \theta$ such that $\mm_{q, \theta'}$ is smooth and symplectic. Moreover, the fact that the morphism $\pi$  exists and is projective and Poisson follows, in the $\theta=0$ case, from the very definitions of affine and GIT quotient and, for general $\theta$, from Lemma \ref{l:varytheta}.
Finally, birationality of $\pi$ is ensured by 
Corollary \ref{c:birational} (and Remark \ref{r:rational stability}).
Thus,  we can conclude that $\mm_{q, \theta}(Q, \beta)$ admits a symplectic resolution, given by the morphism $\pi$.  By Proposition \ref{prop-symp-sing} (or Remark \ref{r:pss}), this implies that the normalisation of $\mm_{q,\theta}(Q,\beta)$ has symplectic singularities.


\subsection{The $q$-divisible case}
\label{ss:na}
In this subsection, we prove the first and third statements of Theorem \ref{main-result}.  We may assume that $\alpha$ is $q$-divisible: this is automatic in the third part, whereas in the first part, the result follows from the second part (proved in the preceding subsection) in the $q$-indivisible case. This means that $\alpha$ is anisotropic, by the following result:
\begin{lem}\label{l:aniso}
  Let $\alpha \in N_{q,\theta}$ be $q$-divisible.
  Then $\alpha \in \Sigma_{q,\theta}$ only if $\alpha$ is anisotropic.  Conversely, if $\alpha=m\beta$ and $\beta \in \Sigma_{q,\theta}$ is anisotropic, then
$\alpha \in \Sigma_{q,\theta}$.
\end{lem}
\begin{proof} This is a generalisation of \cite[Proposition
  1.2]{cb-deco} (in view of Remarks \ref{sigma-qt} and
  \ref{mult-add}), with the same proof. For details, see Corollary
  \ref{c:indiv} below (whose proof is independent of any of the
  results of this section).
\end{proof}

 Recall that a \tit{weighted partition} of $n$ is a sequence $\nu=(l_1, \nu_1; \dots; l_k, \nu_k)$ such that $\nu_1\geq \dots\geq \nu_k$ and $\sum_{i=1}^kl_i\nu_i=n$.  If $\nu$ is a partition of $n$, we shall denote by $\nu\beta$ the representation type $(l_1, \nu_1 \beta; \ldots, l_k,  \nu_k \beta)$. 
\begin{lem}\label{weigh}
\begin{enumerate}
\item The set $\Sigma_{q, \theta}$ contains $\{m\beta\ |\ m\geq 1\}$;
\item 
$\dim C_{q, \theta}^{\nu\beta}(Q, n\beta)=  2\left(k+(p(\beta)-1)\sum_{i=1}^k\nu_i^2\right)$;
\item for $(p(\beta), n)\neq (2,2)$, $\dim \mm_{q, \theta}(Q, n\beta)-\dim C_{q, \theta}^{\nu\beta}(Q, n\beta)\geq 4$ for all $\nu\neq (1, n)$.
\item for $(p(\beta), n)\neq (2,2)$ and $\nu\neq(1, n)$, one has $\dim \mm_{q, \theta}(Q, n\beta)-\dim C^{\nu\beta}_{q, \theta}(Q, n\beta)\geq 8$ unless one of the following holds: (i) $(p(\beta), n)=(2,3)$ and $\nu=(1, 2;1,1)$; (ii) $(p(\beta), n)=(3,2)$ and $\nu=(1,1;1,1)$.
\end{enumerate}
\end{lem}
\begin{proof}
The arguments are completely analogous to those of \cite[Lemma 6.1]{bellamy-schedler}, except here that we use the dimension estimates given by Proposition \ref{strati}. The first statement is a consequence of Lemma \ref{l:aniso}.
\end{proof}
Note that the above result has the following interesting consequence. 
\begin{prop}\label{prop-normal}
Assume that all $\theta$-stable representations of dimension $\gamma<n\beta$ have $\gamma=m\beta$ for some $m$.  Moreover, assume that $(p(\beta),n)\neq (2,2)$. Then, $\mm_{q, \theta}(Q, n\beta)$ is normal.
\end{prop}
For example, the first condition holds if $\beta$ is $q$-indivisible and $\theta$ is generic.
\begin{proof}
This is an immediate consequence of Proposition \ref{quiv-norm} and point (3) of Lemma \ref{weigh}, given that, by assumption on $\theta$, all strata except for the open one have codimension greater than 4.
\end{proof}
Now, let $\alpha \in \Sigma_{q,\theta}$ be $q$-divisible. Write
$\alpha = n \beta$ for $\beta$ $q$-indivisible and $n \geq 2$. For
generic $\theta' \geq \theta$, the only strata of
$\mm_{q,\theta'}(Q,\alpha)$ are those of the form $\nu \beta$, which appear in Lemma \ref{weigh}.
If $(p(\beta),n) \neq (2,2)$, then, taking into
account Remark \ref{r:dim-stab}, all non-open strata have
codimension at least four by Lemma \ref{weigh}.(3).  Therefore
$\mm_{q,\theta'}(Q, \alpha)$ is a symplectic singularity by Flenner’s
Theorem \cite{flenner}.  Now, the map
$\mm_{q,\theta'}(Q, \alpha) \to \mm_{q,\theta}(Q,\alpha)$ is
birational, projective, and Poisson by Corollary \ref{c:birational} (and
Remark \ref{r:rational stability}),
and Lemma \ref{l:varytheta}.  Therefore,
the normalisation of $\mm_{q,\theta}(Q,\alpha)$
is
itself a symplectic singularity by Proposition
\ref{prop-symp-sing}. This proves
the first statement of Theorem \ref{main-result}.

It remains to prove the final statement of Theorem \ref{main-result}.
For this purpose, assume that $\alpha=n\beta$ for $n \geq 2$
and that $\beta \in N_{q, \theta}$ (not necessarily $q$-indivisible or in $\Sigma_{q,\theta}$), such that there exists a $\theta$-stable representation of dimension $\beta$.
Let $U$ be the union of all the strata indexed by $\nu \beta$ for
$\nu$ a weighted partitions of $n$, 
\[
U:=\bigcup_{\nu}C_{q, \theta}^{\nu\beta}(Q, \alpha).
\]
As well as for the previous lemma, to prove the following result one can repeat \tit{verbatim} the arguments in \cite[Lemma 6.2]{bellamy-schedler}.
\begin{lem}
  The subset $U$ is open in $\mm_{q, \theta}(Q, \alpha)$.
  If $\theta$ is generic and $\beta$ is $q$-indivisible, this subset is the entire variety.
\end{lem}
In order to prove that $U$ is factorial, we shall follow the approach of \cite{bellamy-schedler}, which was itself inspired by results of Drezet \cite{Drezet} on factoriality of points in moduli spaces of semistable sheaves on rational surfaces. Assuming the notation above, with $\pi:\mr{Rep}^{\theta\mr{-}ss}(\Lambda^q, \alpha)\rightarrow \mm_{q, \theta}(Q, \alpha)$ denoting the quotient map, define $V:=\pi^{-1}(U)$. We aim at proving that $V$ is a local complete intersection and that it is factorial and normal. We shall then descend the factoriality property to the subvariety $U$. 
\begin{prop}\label{compl-int} $V$ is a local complete intersection, factorial and normal.
\end{prop}
The proof of this proposition follows closely the arguments used in \cite[Proposition 6.5]{bellamy-schedler}. 



\begin{proof}[Proof of Proposition \ref{compl-int}]
Since $V$ is open inside $\mr{Rep}^{\theta\mr{-}ss}(\Lambda^q,\alpha)$,
Proposition \ref{rep-equidim2} implies that it is a local complete intersection.
To prove normality and factoriality, recall that a local complete intersection satisfies Serre's $S_2$ property, so Serre's criterion implies that it is normal if it is smooth outside 
a locus of codimension at least $2$.  Moreover, 
by a result of Grothendieck (\cite[Theorem 3.12]{kaledin-lehn-sorger}),
a local complete intersection which is smooth outside a locus of codimension at least $4$ is factorial. Put together, to show that $V$ is normal and factorial,
it suffices to show that it is
 smooth outside of a locus of codimension at least 4. For this, one can repeat verbatim the arguments used in \cite[Proposition 6.5]{bellamy-schedler}, replacing Corollary 6.4 and Lemma 6.1 with Lemma \ref{l:dimest} and Lemma \ref{weigh}, respectively.
\end{proof}
In order to descend factoriality from $V$ to $U$ we use Drezet's method. In particular, \cite[Theorem 6.7]{bellamy-schedler} holds true in this context as well with no change in the proof of the result, as we have already made sure that all of the tools used there are still applicable here, Proposition \ref{compl-int} being the most important one. Thus, the corresponding statement of \cite[Corollary 6.9]{bellamy-schedler} is the following:
\begin{thm}\label{t:factorial}
  $U$ is a factorial variety.
\end{thm}
We omit the proof, as it is exactly the same as in \cite[Corollary 6.9]{bellamy-schedler}.

Using the previous theorem and the estimates on the codimension of the singular locus, one can conclude that $\mm_{q, \theta}(Q, \alpha)$ does not admit a symplectic resolution.  We state this formally below, where we also recall our running hypotheses for the reader's convenience.
\begin{thm}\label{t:fact-term}
Let $\alpha=n\beta\in\Sigma_{q, \theta}$ be anisotropic imaginary, for $n \geq 2$,
such that there exists a $\theta$-stable representation of $\Lambda^q$ of dimension $\beta$, and 
$(p(\beta), n)\neq (2,2)$. Then $\mm_{q, \theta}(Q, \alpha)$ has an open subset which is factorial, terminal, and singular. Hence it
does not admit a symplectic resolution.  Moreover, if $\theta$ is generic and $\alpha$ is $q$-indivisible, then this open subset is the entire variety.
\end{thm}
\begin{proof}
The subset $U$ is singular, since it contains the non-open stratum $(n,\beta)$.  It is factorial by Theorem \ref{t:factorial}. Under the assumptions, the singular strata in $U$ all have codimension at least four, hence also the singular locus.  Thus, $U$ is terminal by \cite{naminote}, since it has symplectic singularities and the singular locus has codimension at least four.
\end{proof}
This completes the proof of the third and final statement of Theorem \ref{main-result}.

\subsection{Proof of Corollary \ref{c:main-result}}
Write $\alpha=m\beta$ for $\beta$ $q$-indivisible.
Note that, for $\theta$ generic, the only possible decompositions of $\alpha$ are into multiples of $\beta$. If $\alpha$ is $q$-indivisible,
it therefore follows trivially that $\alpha \in \Sigma_{q,\theta}$. Since
the only stratum in $\mm_{q,\theta}(Q,\alpha)$ is the open one of stable representations, it also follows
from Proposition \ref{p:s-smooth} 
that $\mm_{q,\theta}$ is smooth symplectic.
Suppose that $\alpha$ is $q$-divisible.
It then follows from 
Lemma \ref{l:aniso}
that $\alpha$ is in $\Sigma_{q,\theta}$ if and only if it is anisotropic.
 This completes the proof of part (i).

Part (ii) follows from Proposition \ref{prop-normal} and
Theorem \ref{main-result}.  The first statement of part (iii)
follows from Theorem \ref{main-result}. Finally, the last statement follows from Proposition \ref{p:s-smooth} because, in this case, there is only one stratum in $\mm_{q,\theta}(Q,\alpha)$, consisting of $\theta$-stable representations.



\subsection{The anisotropic imaginary $(p(\alpha),n)=	(2,2)$ case} The only case left out in this analysis is that of $2\alpha \in \Sigma_{q,\theta}$ for $\alpha \in N_{q,\theta}$ satisfying $p(\alpha)=2$.
The analogous question of existence of a symplectic resolution in the setting of Nakajima quiver varieties is settled in \cite[Theorem 1.6]{bellamy-schedler}, where it is shown that, for generic $\theta$, blowing up the ideal sheaf defining the singular locus gives a symplectic resolution of singularities. This is achieved by showing that, \'etale locally, the variety is isomorphic to the product of $\C^4$ with the closure of the six-dimensional nilpotent orbit closure in $\Sp(\C^4)$:
see \cite[Theorem 5.1]{bellamy-schedler} and the references therein.	Given this, one might conjecture that an analogous result holds for multiplicative quiver varieties, and such a result should be proved by studying the \'etale local structure of the variety. In fact, by Artin's approximation theorem \cite{artin}, it would be sufficient to give a description of the formal neighbourhood of a point. This will be discussed in a future work.
For more details, see Section \ref{sec-fut}. 

\section{Combinatorics of multiplicative quiver varieties}\label{sec-comb}
In this section we study some combinatorial problems which are related to the geometry of multiplicative quiver varieties. Indeed, an interesting problem is to classify all the possible ``$(2,2)$-cases'': these are the main $q$-divisible
cases for which we conjecture that there exists a symplectic resolution. In the next subsection we carry out these computations in the case of crab-shaped quivers (which we defined in Section \ref{section-char}).  We shall see how most of the $(2,2)$-cases occur in the case of star-shaped quivers, i.e., where there are no loops, so that the corresponding surface has genus zero.
It is also important to point out that the classification below yields
an explicit classification of the crab-shaped quivers for which symplectic resolutions exist, or are conjectured to exist: see Corollary \ref{c:crab-sr} in the next section for details on this.

\subsection{(2,2) cases for crab-shaped quivers} The analysis is based on some standard numerical arguments and the constraints on the dimension vector $\alpha$ for it to satisfy the conditions of Section \ref{section-char}, i.e., it has to represent the multiplicities of the eigenvalues in the prescribed conjugacy class. 

\begin{thm} \label{comb1}There are exactly 13 pairs $(Q, \alpha)$, where $Q$ is a star-shaped quiver as in Section \ref{section-char} and $\alpha \in \mc{F}(Q)$ is in the fundamental region, such that $p(\alpha)=2$. Such pairs are depicted as follows, where a vertex is substituted by the corresponding entry of the dimension vector: 
\begin{equation}\label{quiver13}
\begin{tikzcd}
1& 1 & 1 \\
1&\arrow[l] 2\arrow[lu]\arrow[u]\arrow[ru]\arrow[r]&1\\
\end{tikzcd}
\end{equation}
\begin{equation}\label{quiver12}
\begin{tikzcd}
1& &2\arrow[r]&1\\
1&\arrow[l]3\arrow[lu]\arrow[ru]\arrow[r]&2\arrow[r]&1\\
\end{tikzcd}
\end{equation}
\begin{equation}\label{quiver11}
\begin{tikzcd}
2& &2\\
1&\arrow[l]4\arrow[lu]\arrow[ru]\arrow[r]&3\arrow[r]&2\arrow[r]&1\\
\end{tikzcd}
\end{equation}
\begin{equation}\label{quiver10}
\begin{tikzcd}
& & &3\arrow[r]&2\arrow[r]&1\\
1&\arrow[l]2&\arrow[l]4\arrow[ru]\arrow[r]&3\arrow[r]&3\arrow[r]&1\\
\end{tikzcd} 
\end{equation}
\begin{equation}\label{quiver9}
\begin{tikzcd}
2& &2\\
2&\arrow[l]4\arrow[lu]\arrow[ru]\arrow[r]&2\arrow[r]&1\\
\end{tikzcd}
\end{equation}
\begin{equation}\label{quiver8}
\begin{tikzcd}
& &4\arrow[r]&3\arrow[r]&2\arrow[r]&1\\
2&\arrow[l]5\arrow[ru]\arrow[r]&4\arrow[r]&3\arrow[r]&2\arrow[r]&1\\
\end{tikzcd}
\end{equation}
\begin{equation}\label{quiver7}
\begin{tikzcd}
&&1&&&\\
&&3\arrow[u]&&&\\
1&\arrow[l]3&\arrow[l]5\arrow[u]\arrow[r]&4\arrow[r]&3\arrow[r]&2\arrow[r]&1\\ 
\end{tikzcd}
\end{equation}
\begin{equation}\label{quiver5}\begin{tikzcd}
&&&&&3&&&\\
1&\arrow[l]2&\arrow[l]3&\arrow[l]4&\arrow[l]5&\arrow[l]6\arrow[u]\arrow[r]&4\arrow[r]&2\arrow[r]&1\\ 
\end{tikzcd}\end{equation}
\begin{equation}\label{quiver6}\begin{tikzcd}
&&2&&&\\
&&4\arrow[u]&&&\\
2&\arrow[l]4&\arrow[l]6\arrow[u]\arrow[r]&4\arrow[r]&2\arrow[r]&1\\ 
\end{tikzcd}\end{equation}
\begin{equation}\label{quiver4}\begin{tikzcd}
& &4&&&\\
2&\arrow[l]5&\arrow[l]8\arrow[u]\arrow[r]&7\arrow[r]&6\arrow[r]&5\arrow[r]&4\arrow[r]&3\arrow[r]&2\arrow[r]&1\\
\end{tikzcd}\end{equation}
\begin{equation}\label{quiver3}\begin{tikzcd}
&&&&4&&&\\
1&\arrow[l]2&\arrow[l]4&\arrow[l]6&\arrow[l]8\arrow[u]\arrow[r]&6\arrow[r]&4\arrow[r]&2\\ 
\end{tikzcd}\end{equation}
\begin{equation}\label{quiver2}\begin{tikzcd}
&&& 5&\\
1&\arrow[l]4&\arrow[l]7&\arrow[l]10\arrow[u]\arrow[r]&8\arrow[r]&6\arrow[r]&4\arrow[r]&2\\
\end{tikzcd}\end{equation}
\begin{equation}\label{quiver1}
\begin{tikzcd}
&&6&&&&&&\\
4&\arrow[l]8&\arrow[l]12\arrow[u]\arrow[r]&10\arrow[r]&8\arrow[r]&6\arrow[r]&4\arrow[r]&2\arrow[r]&1\\ 
\end{tikzcd}
\end{equation}\end{thm}
\begin{rem}\label{r:2delta-1-star}
It is important to highlight that quivers (\ref{quiver9}), (\ref{quiver6}), (\ref{quiver3}), (\ref{quiver1}) are the framed affine Dynkin quivers $\tilde{D}_4, \tilde{E}_6, \tilde{E}_7, \tilde{E}_8$, respectively, with dimension vector given by $(2\delta, 1)$, where $\delta$ is the minimal isotropic imaginary root of the corresponding quiver. See Remark \ref{r:comb} for the significance of this.
\end{rem}
\begin{proof}[Proof of Theorem \ref{comb1}]
  Note that $p(\alpha)=2$ if and only if $\langle \alpha, \alpha \rangle = -1$.
 Let us calculate the value of $\langle\alpha, \alpha\rangle$ explicitly, for $\alpha$ a general dimension vector. The general star-shaped quiver has $g$ loops and $k$ legs, each of which has $l_i$ arrows, $i=1, \dots, k$. We have
\[
\langle\alpha, \alpha \rangle=(1-g)n^2+\sum_{i, j}\alpha^2_{i, j} -n\sum_{i=1}^k\alpha_{i,1}-\sum_{i=1}^{k}\sum_{j=1}^{l_i-1}\alpha_{i, j}\alpha_{i, j+1}.
\]
Assume now that $\alpha\in\mc{F}(Q)$ and that $\langle\alpha, \alpha\rangle=-1$; then, given that $\langle\alpha, \alpha \rangle=\sum_{i\in Q_0}\alpha_i\langle \alpha, e_i\rangle=\sum_{i\in Q_0}\alpha_i\langle e_i, \alpha\rangle$, this implies that there can only be two possibilities: 
\begin{itemize}
	\item[a)] there exists a unique vertex $i\in Q_0$ such that either $\alpha_i=1$ and $(\alpha, e_i)=-2$, or $\alpha_i=2$ and $(\alpha, e_i)=-1$, with $(\alpha, e_j)=0$ for $j\neq i$; this implies that, denoted by $\mr{Adj}(i)$ the set of vertices which are adjacent to $i$, one has $\sum_{j\in\mr{Adj}(i)}\alpha_j=5$ for $\alpha_i=2$ and $\sum_{j\in\mr{Adj}(i)}\alpha_j=4$ for $\alpha_i=1$;
	\item[b)] there are two distinct vertices $i$ and $i'$ such that $(\alpha, e_{i})=(\alpha, e_{i'})=-1$ and $\alpha_i=\alpha_{i'}=1$, with $(\alpha, e_j)=0$ for $j\neq i, i'$. In this case one has $\sum_{j\in\mr{Adj}(k)}\alpha_j=3$ for $k=i, i'$.
\end{itemize}
In this case, if $i$ or $i'$ is the central vertex, then the only possibility is given by the quiver (1) in the statement of the theorem. Otherwise, if $v$ is the central vertex, then $(\alpha, e_v)=0$, which implies that $\sum_{j\in \mr{Adj}(v)}\alpha_j=2n$, where $\alpha_v=n$: indeed, 
\[
0=(\alpha, e_v)=\langle\alpha, e_v\rangle+\langle e_v,\alpha \rangle= n-\sum_{k\rightarrow v}\alpha_k+n-\sum_{v\rightarrow l}\alpha_l=2n-\sum_{j\in \mr{Adj}(v)}\alpha_j.
\]
Now, fix a branch along which none of the special vertices $i$ and $i'$ appear, let $l$ be its length and let $\beta_0=n, \beta_1, \dots, \beta_l$ be the components of the vector $\alpha$ along the branch.

Then, using that $(\alpha, e_j)=0$ for $j\neq i,i'$, we get the recursive formula
\[
2\beta_j=\beta_{j-1}+\beta_{j+1},
\]
for $j=1, \dots, l-1$, and also $\beta_{l-1}=2\beta_l$, which implies that 
\[
\beta_j=(l+1-j)\beta_l.
\]
Therefore, the branch has the form
\[
n\longrightarrow n-c\longrightarrow n-2c\longrightarrow \dots \longrightarrow c,\\ 
\]
where $c$ is a positive integer such that $c|n$. Moreover, in order for  condition a) to be satisfied there has to be one branch ending with one of the following 
\[
5\longrightarrow 2,\ \ \ 4\longrightarrow 2\longrightarrow 1, \ \ \ 4\longrightarrow 1,
\]
and, thus, having the form

\begin{multline*}
\tag{\S}
n-3\longrightarrow\dots, \longrightarrow 5\longrightarrow 2, \\ n-2\longrightarrow\dots, \longrightarrow 4\longrightarrow 2\longrightarrow 1, \\ n-3\longrightarrow\dots, \longrightarrow 4\longrightarrow 1
\end{multline*}

respectively; for condition b), there have to be two branches ending as
\[
3\longrightarrow 1, 
\] 
having the form 
\begin{equation*}
\tag{\S\S}
n-2\longrightarrow\dots\longrightarrow 3\longrightarrow 1.
\end{equation*}

Therefore, we are left to consider a star-shaped quiver where all but one or two branches are as follows:
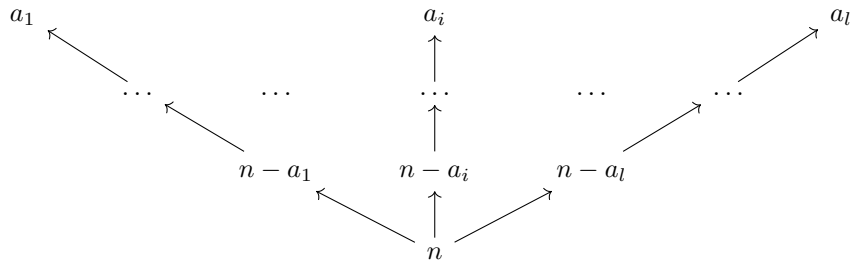
\[
\begin{tikzcd}
a_1&&&a_i&&&a_{l}\\
&\dots\arrow[lu]&\dots&\dots\arrow[u]&\dots&\dots\arrow[ru]&\\
&&n-a_1\arrow[lu]&n-a_i\arrow[u]&n-a_{l}\arrow[ru]&&\\
&&&\arrow[lu]n\arrow[u]\arrow[ru]&&&\\
\end{tikzcd}
\]

Moreover, if the quiver satisfies condition a), then $l=k-1$ and there is an additional branch having one of the forms in (\S); on the other hand, if the quiver is as in case b), then $l=k-2$ and there are two additional legs of the form described by (\S\S). 

We shall now use some numerical arguments to prove that, among all such possibilities, only the ones listed in the statement of the theorem can actually occur. First, let us spell out how the equality $0=(\alpha, e_v)$ can be rephrased: one has that
\[
0=(\alpha, e_v) \iff 2n=\sum_{i=1}^k(n-a_i) \iff \sum_{i=1}^ka_i=(k-2)n.
\]
Therefore, one has the following possibilities: 
\begin{itemize}
	\item[a)] in the cases of a branch ending with $5\longrightarrow 2$ or $4\rightarrow 1$ the equality $0=(\alpha, e_v)$ reads as 
	\[
	\frac{3}{n}+\sum_{i=1}^{k-1}\frac{1}{n/a_i}=k-2, 
	\]
	where $n\equiv 2\Mod{3}$ and $n\equiv 1\Mod{3}$, respectively, and $n>a_i\geq 2$, $a_i|n$ for every $i$; these shall be mentioned in the following as cases a.1 and a.2. On the other hand for a branch ending with $4\longrightarrow 2\longrightarrow 1$ we have
	\begin{equation}\label{eq:inner}
	\frac{2}{n}+\sum_{i=1}^{k-1}\frac{1}{n/a_i}=k-2,
		\end{equation}
	where $n$ has to be even and $a_i|n$; this is renamed as case a.3.
	\item[b)] there are two branches $3\longrightarrow 1$ and $0=(\alpha, e_v)$ is equivalent to 
	\[
	\frac{4}{n}+\sum_{i=1}^{k-2}\frac{1}{n/a_i}=k-2,
	\]
	and $a_i|n$ for every $i$ and $n$ has to be odd, and $a_i<n$, for every $i$. 
\end{itemize}
In cases a.1 and a.2 one has that $n\geq 4$ which forces $k\leq 4$: indeed, one has that for $n\geq 4$ $n/a_i\geq 2$ and, therefore, 
\[
\frac{3}{n}+\sum_{i=1}^{k-1}\frac{1}{n/a_i} \leq \frac{3}{4}+\frac{k-1}{2}< \frac{k+1}{2},
\]
which implies that 
\[
k-2<\frac{k+1}{2},
\]
and thus $k\leq 4$.
For $k=4$ and $n=4$ it is easily checked that quiver (\ref{quiver11}) in the statement of the result is the only possibility. If $k=3$, then the following inequality holds: 
\[
\left(\frac{1}{2}+\frac{1}{4}\right)n\geq \left(1-\frac{3}{n}\right)n=n-3,
\]
which forces $n\leq 12$; one can check that case a) cannot be realised for $n=4, 7, 11$ and that the cases $n=5, 8, 10$ give quivers (\ref{quiver8}), (\ref{quiver4}) and (\ref{quiver2}) respectively. Next, for case a.3 one has: n even and $k\leq 4$: indeed, since $n\geq 4$, from equation (\ref{eq:inner}), one has that 
\[
\frac{2}{4}+\frac{k-1}{2} \geq k-2,
\]
which implies that 
\[
\frac{k}{2}\leq 2.
\]
 If $k=4$ and $n=4$ then one gets quiver (\ref{quiver9}). If $k=3$, then $n\leq 12$: indeed, from equation (\ref{eq:inner}), we have 
 \[
 \frac{2}{n}+\frac{1}{2}+\frac{1}{3}\geq 1
 \] 
 
This leads to: quiver (\ref{quiver10}) for $n=4$, quivers (\ref{quiver5}) and (\ref{quiver6}) for $n=6$, quiver (\ref{quiver3}) for $n=8$, and quiver (\ref{quiver1}) for $n=12$. 

We turn now to case b): $n$ is odd and  $\geq 3$; this implies that $k\leq 4$: indeed, in the same way as the previous case, equation (\ref{eq:inner}) gives
\[
\frac{4}{3}+\frac{k-2}{3}\geq k-2 \implies \frac{8}{3}\geq\frac{2k}{3}.
\]
 Therefore, setting $k=4$ forces $n=3$, which leads to quiver (\ref{quiver12}). When $k=3$ one has that $n\leq 6$, which implies that $n=3$ or $n=5$. One checks that $n=3$ is impossible, whereas $n=5$ gives quiver (\ref{quiver7}). Since we have dealt with all the possible cases, the proof is complete.
\end{proof}
\begin{thm}\label{comb2}
Assume that $g\geq 1$. Then, the only pairs $(Q, \alpha)$, where $Q$ is a crab-shaped quiver and $\alpha \in \mc{F}(Q)$ is such that $\langle\alpha, \alpha\rangle=-1$ are the following:  
\begin{equation}
\begin{tikzcd}
\arrow[out=150, in=210,loop,distance=1cm]{}{} 
\arrow[out=130, in=230,loop,distance=2cm]{}[swap]{}  1
\end{tikzcd}
\end{equation}
\begin{equation}
\begin{tikzcd}
\arrow[out=150, in=210,loop,distance=1cm]{}{} 2 \arrow[r]&1
\end{tikzcd}
\end{equation}
\end{thm}
\begin{rem}\label{r:2delta-1-a0}
  Parallel to Remark \ref{r:2delta-1-star}, in the second case above,
  the quiver and dimension vector are also of the form $(2\delta,1)$
  where $\delta=(1)$ is the primitive imaginary root of affine type
  $\widetilde A_0$ (the Jordan quiver with one vertex and one arrow).
\end{rem}
\begin{proof}[Proof of Theorem \ref{comb2}]
As in the arguments of the previous theorem, we see that, if $v$ is the central vertex and $\alpha\in\mc{F}$, then there are 3 possibilities for the value of $(\alpha, e_v)$, i.e.,  $(\alpha, e_v)$ can be either $0$, $-1$ or $-2$. In general one has that 
\[
(\alpha, e_v)=2(1-g)n-\sum_{i=1}^k\alpha_{i, 1}.
\]
If $k=0$, then, by the argument leading to cases (a) and (b) in the proof of Theorem \ref{comb1}, one must have $n=1$ and $g=2$, which gives the first quiver of the statement of the result. Thus, one is left to show that for $k\geq 1$, there are no crab-shaped quivers satisfying the mentioned conditions other than the second quiver in the statement of the theorem. If $k\geq 1$, $n> 1$, which implies that either $(\alpha, e_v)=0$ or $-1$. In the first case, we get that 
\[
2(1-g)n=\sum_{i=1}^k\alpha_{i, 1};
\] 
but $g\geq 1$, which gives $\sum_{i=1}^k\alpha_{i, 1}\leq 0$, a contradiction. If $(\alpha, e_v)=-1$, then one must have $n=2$ and $\alpha_{i, 1}=1$ for $i=1, \dots, k$ and $g=1$. This implies that
\[
k=5-4g,
\]
which forces $k=1$. Therefore, we get the quiver 
\[
\begin{tikzcd}
\arrow[out=150, in=210,loop,distance=1cm]{}{} 2 \arrow[r]&1. 
\end{tikzcd}
\]
Since we dealt with all the possible cases the proof is complete.
\end{proof}

\begin{rem}\label{r:comb}Note that, to get a list of all the $(2,2)$-cases one has to take each of the pairs $(Q, \alpha)$ drawn above and consider the pair $(Q, 2\alpha)$. Moreover, given $q\in (\C^\times)^{Q_0}$ and $\theta \in \Z^{Q_0}$, it follows from Theorem \ref{t:flat-sigma} below that the given $\alpha$ are in
$\Sigma_{q, \theta}$ (since they are already in $\mc{F}(Q)$)
if and only if the following are satisfied: (a) they are in $N_{q,\theta}$, i.e.,  $q^\alpha=1$ and $\theta\cdot\alpha=0$, and (b) in the $(2\delta,1)$ cases (mentioned in Remarks \ref{r:2delta-1-star} and \ref{r:2delta-1-a0}), $\delta \notin N_{q,\theta}$, i.e.,  $q^\delta \neq 1$ or $\theta \cdot \delta \neq 0$.
\end{rem}
\section{General dimension vectors and decomposition}\label{s:decomp}
One fundamental tool in the classification theorem \cite[Theorem
1.4]{bellamy-schedler} is the canonical decomposition of a dimension
vector of a quiver variety into summands which lie in
$\Sigma_{\lambda, \theta}$, which is the additive version of the set
$\Sigma_{q, \theta}$ defined in this paper
(one just needs to replace the condition
$q^\alpha=1$ with $\lambda\cdot \alpha=0$). This appears in
Crawley-Boevey's canonical decomposition in the additive case (extended to the case $\theta \neq 0$ in \cite{bellamy-schedler}). Combinatorially, it says:
\begin{lem}\cite[Theorem 1.1]{cb-deco}, \cite[Proposition 2.1]{bellamy-schedler}
\label{lem-prod}Let $\alpha\in \N R^+_{\lambda, \theta}$. Then $\alpha$ admits a unique decomposition $\alpha=n_1\sigma^{(1)}+\cdots+n_k\sigma^{(k)}$ as a sum of elements $\sigma^{(i)}\in\Sigma_{\lambda, \theta}$ such that any other decomposition of $\alpha$ as a sum of elements from $\Sigma_{\lambda, \theta}$ is a refinement of this decomposition. 
\end{lem}
Geometrically, the statement (together with the consequence for symplectic resolutions) is:
\begin{thm}\cite[Theorem 1.1]{cb-deco}, \cite[Theorem 1.4]{bellamy-schedler}\label{thm-prod} The symplectic variety $\mathcal{M}_{\lambda,\theta} (Q, \alpha)=\mu^{-1}(\lambda)^{\theta\mr{-}ss}\!/\!/\GL(\alpha)$ is isomorphic to the product 
\[
\mathcal{M}_{\lambda,\theta} (Q, \alpha)\cong \prod_{i=1}^kS^{n_i}\mathcal{M}_{\lambda,\theta}(Q, \sigma^{(i)}).
\]
Moreover, it admits a symplectic resolution if and only if each $\mathcal{M}_{\lambda,\theta}(Q, \sigma^{(i)})$ admits a symplectic resolution.
\end{thm}

For multiplicative quiver varieties, the combinatorial statement still
holds, but it is not clear that such a geometric decomposition
holds. We instead prove a weaker statement, which gives a
decomposition into factors which might not be minimal, but still has
all of the needed properties. Moreover, the resulting classification
of symplectic resolutions is the same statement as if the canonical
decomposition as above held.  As a result we are able to generalise
Theorem \ref{main-result} to the case of general dimension vectors
(Theorem \ref{t:main-result-nonsigma}), and give its specialisation to
the crab-shaped case (Corollary \ref{c:crab-sr}).  To complete the
proof, we need to establish that
$\mm_{q,\theta'}(Q,\alpha) \to \mm_{q,\theta}(Q,\alpha)$ is a
symplectic resolution for many $\alpha$ not in $\Sigma_{q,\theta}$
(Theorem \ref{t:res-flat}), in order to handle such factors appearing
in the decomposition. In the additive case, such resolutions include
the Hilbert schemes of points in $\C^2$ and in hyperk\"ahler ALE
spaces (i.e., minimal resolutions of du Val singularities).

\subsection{Flat roots} \label{ss:flat-roots}
In order to write a product decomposition in the multiplicative setting,
the dimension vectors for the factors need to be more general than those in $\Sigma_{q,\theta}$. The dimension vectors turn out to include ``flat roots'', which are those for which the moment map is flat (this is true for roots in $\Sigma_{q,\theta}$). This condition is also very important in order to have a geometric understanding of the varieties.
\begin{defn}\label{d:flat-roots}
A vector $\alpha \in N_{q,\theta}$ is called \emph{flat} if, for every decomposition $\alpha=\alpha^{(1)}+ \cdots + \alpha^{(m)}$ with $\alpha^{(i)} \in R^+_{q,\theta}$, we have $p(\alpha) \geq p(\alpha^{(1)})+ \cdots + p(\alpha^{(m)})$. Let $\widetilde{\Sigma}_{q,\theta}$ be the set of flat roots.
\end{defn}
\begin{rem} As in \cite[Theorem 1.1]{cb-geom}, we could alternatively have
made the definition only requiring $\alpha^{(i)} \in N_{q,\theta}$. Indeed, it follows
from the proof of the
decomposition theorem (Theorem \ref{t:decompo}) below that if $\alpha \in N_{q,\theta}$, then there is a decomposition $\alpha = \beta^{(1)} + \cdots + \beta^{(k)}$ with each $\beta^{(i)}$ either in $\widetilde{\Sigma}_{q,\theta}$
or of the form $\beta^{(i)} = m \gamma, \gamma \in \Sigma_{q,\theta}^{\text{iso}}$, satisfying
$p(\alpha) \leq p(\beta^{(1)}) + \cdots + p(\beta^{(k)})$. Hence, if we know the inequality when the $\alpha^{(i)} \in R^+_{q,\theta}$, we also know it when the $\alpha^{(i)} \in N_{q,\theta}$.
\end{rem}
The definition has the following interpretation. Let
$\SL(\alpha):=\{g \in \GL(\alpha) \mid \prod_{i \in Q_0} \det(g_i) = 1\} \subset \GL(\alpha)$. Note that $\Phi_\alpha$ factors through the inclusion $\SL(\alpha) \to \GL(\alpha)$; let $\overline{\Phi_\alpha}: \mr{Rep}^{\circ}(\overline{Q}, \alpha)^{\theta\mr{-}ss}\to \SL(\alpha)$ be the factored map.
\begin{prop}\label{p:flat-char}
If $\alpha$ is a flat root, then 
$\overline{\Phi}_\alpha$
is flat over a neighbourhood of $q$. In particular,
 $\Phi_\alpha^{-1}(q)^{\theta\mr{-}ss}$ is a complete intersection and is equidimensional of dimension $g_\alpha+2p(\alpha)$.
\end{prop}
\begin{proof}
  The second statement follows from the argument of Proposition
  \ref{rep-equidim2} (following \cite[Theorem 1.11]{cb-geom}): all
  arguments go through with the strict inequality replaced by the
  non-strict one, except that no statement can be deduced about the
  stable (or simple) representations forming a dense subset.  For the
  first statement, concerning flatness, note that
  $\dim \overline{\Phi}_\alpha^{-1}(q)^{\theta\mr{-}ss} = \dim
  \mr{Rep}^{\circ}(\overline{Q}, \alpha)^{\theta\mr{-}ss} - \dim
  \SL(\alpha)$. Then the statement follows from the following general
  considerations. Suppose we are given a morphism of varieties
  $f:X \to Y$, with $X$ equidimensional. By upper semicontinuity of
  the fibre dimension, the minimum fibre dimension is
  $\dim X - \dim Y$ and the locus in $Y$ where the fibres have this
  minimal dimension is open.  Next, it is a standard fact that a
  morphism from a Cohen-Macaulay variety $X$ to a smooth variety $Y$
  is flat if and only if for every $x \in X$, with $y = f(x) \in Y$,
  one has the equality $\dim_x X = \dim_{y} Y + \dim f^{-1}(y)$.  It
  follows that $f$ is flat over the open locus where the fibres have
  minimum dimension. Now, back to the situation at hand, by the second
  statement of the proposition, the minimum dimension is attained over
  $q \in \SL(\alpha)$. As the domain and codomain are equidimensional
  and smooth, the aforementioned open locus is a neighbourhood of $q$
  over which $\overline{\Phi}_\alpha$ is flat.
\end{proof}
Putting this together with Proposition \ref{p:inverseimage}, we conclude the following analogue of the last statement of Proposition \ref{rep-equidim2}:
\begin{coro} \label{c:dense-flat}
For $\alpha$ a flat root, a dense subset of $\Phi_\alpha^{-1}(q)^{\theta\mr{-}ss}$ is
  given by the union of preimages of strata of types
  $(1,\beta^{(1)}; \ldots ;1, \beta^{(r)})$ with
  $p(\alpha)=\sum_{i=1}^r p(\beta^{(i)})$.
\end{coro}

Observe that
$\Sigma_{q,\theta} \subseteq \widetilde{\Sigma}_{q,\theta}$. The
opposite inclusion does not hold: for instance, with
$(q,\theta)=(1,0)$, one can take the quiver with two vertices and two
arrows, one a loop at the first vertex, the other an arrow to the
second vertex. Then the dimension vector $(m,1)$ is flat for all $m$,
but only in $\Sigma_{q,\theta}$ for $m=1$.
\begin{rem}\label{sigma-qt}
  Note that, for every pair $q,\theta$, there always exists $q'$ such
  that $N_{q',0}=N_{q,\theta}$, and hence
$\Sigma_{q,\theta}=\Sigma_{q',0}$ and $\widetilde{\Sigma}_{q,\theta} = \widetilde{\Sigma}_{q',0}$. Indeed, let $z \in \C^\times$ be
  a multiplicatively independent element from the $q_i$ (i.e.,
  $\langle z, q_i \rangle / \langle q_i \rangle$ is infinite cyclic,
  where $\langle-\rangle$ denotes the multiplicative group generated
  by the given elements).  Set $q'_i := q_i z^{\theta_i}$. Then $q'$
  has the desired properties.   
\end{rem}
\begin{rem}\label{mult-add}
  Similarly, given any parameters in the additive case (for ordinary quiver varieties),
  $\lambda \in \C^{Q_0}, \theta \in \Z^{Q_0}$, we also can construct
  $q' \in (\C^\times)^{Q_0}$ such that
  the sets $N$ and $\Sigma$ correspond. More precisely, letting $N^a_{\lambda,\theta}, \Sigma^a_{\lambda,\theta}, \widetilde{\Sigma^a}_{\lambda,\theta}$ denote the sets defined for the additive case, this means that
  $N^a_{\lambda,\theta} = N_{q',0}$,
  $\Sigma^a_{\lambda,\theta}=\Sigma_{q',0}$, and
  $\widetilde{\Sigma^a}_{\lambda,\theta} = \widetilde{\Sigma}_{q',0}$.
\end{rem}
  We recall, following \cite{cb-geom} and \cite{Su-fmmrq}, how to classify flat roots in terms of the fundamental region.
\begin{defn}
We say that the transformation $\alpha \mapsto s_v(\alpha)$ is a \emph{$(-1)$-reflection} if $s_v(\alpha)=\alpha-e_v$.
\end{defn}
We point out a useful geometric consequence of this definition:
\begin{prop} \label{p:-1-iso}
Suppose that $\alpha \mapsto s_v(\alpha)$ is a $(-1)$-reflection
and that $q_v = 1$ and $\theta_v = 0$. Then
there is a reflection isomorphism $\mm_{q,\theta}(s_v(\alpha)) \iso \mm_{q,\theta}(\alpha)$.
\end{prop}
\begin{proof}
There is an obvious map $\mm_{q,\theta}(\alpha-e_v) \to \mm_{q,\theta}(\alpha)$, given by $\rho \mapsto \rho\oplus \C_v$, where $\C_v$ is the trivial representation (all arrows act as zero).  We claim that it is an isomorphism.  In the decomposition
of any $\theta$-polystable representation of dimension $\alpha$
into stable representations, at least
one factor must have dimension vector which has positive pairing with $e_v$. By
 \cite[Lemma 5.1]{cb-shaw} (in the case $\theta = 0$, which extends to 
the general case by replacing simple representations by $\theta$-stable ones), this summand must be $\C_v$ itself.
Therefore, the obvious map is an isomorphism.
\end{proof}
\begin{defn} Given $\alpha \in R^+_{q,\theta}$, call a sequence
$v_1, \ldots, v_m \in Q_0$ a reflecting sequence if, setting
 $$(q^{(i)},\theta^{(i)},\alpha^{(i)}) := (u_{v_i} \cdots u_{v_1}(q),
  r_{v_i} \cdots r_{v_1}(\theta), s_{v_i} \cdots s_{v_1}(q)),$$
  we have (a) $\alpha^{(m)} \in \mc{F}(Q) \cup \{e_v \mid v \in Q_0\}$, and
  (b) $\alpha^{(i)}_{v_i} < \alpha^{(i-1)}_{v_i}$ for all $i$.
\end{defn}
\begin{lem} A reflecting sequence always exists.
\end{lem}
\begin{proof}
By definition, $\alpha$ is a root if and only if
 there exists a sequence of reflections at loopfree vertices
taking $\alpha$ to either the fundamental region or to an elementary root $e_v$ (it is imaginary in the former case and real in the latter case). Now, given $\alpha \in \N^{Q_0}$, let $N_\alpha := |\{\beta \in R^+ \text{ real} \mid (\alpha,\beta) > 0\}|$. Then each reflection satisfying (b) decreases $N_{\alpha}$ by one, and a non-trivial reflection
not satisfying (b) increases $N_{\alpha}$ by one. 
Now assume that $\alpha \in R^+_{q,\theta}$. Then $N_{\alpha} < \infty$.
Since $s_v(R^+ \setminus \{e_v\}) \subseteq R^+$, an arbitrary sequence
of reflections satisfying (b) will remain in $\N^{Q_0}$.
Thus, if $\alpha$ is real,
an arbitrary $N_{\alpha}-1$ reflections satisfying (b) will send $\alpha$ to $e_v$ for some $v \in Q_0$, and if $\alpha$ is imaginary, then an arbitrary
$N_\alpha$ reflections satisfying (b) will take $\alpha$ to $\mc{F}(Q)$.
\end{proof}
As in \cite[Theorem 1.2]{Su-fmmrq}, we have the following.
\begin{thm} \label{t:flat}
Let $\alpha \in R^+_{q,\theta}$. Pick any sequence of
  vertices $v_1, \ldots, v_m \in Q_0$ such that, for
  $(q^{(i)},\theta^{(i)},\alpha^{(i)}) := (u_{v_i} \cdots u_{v_1}(q),
  r_{v_i} \cdots r_{v_1}(\theta), s_{v_i} \cdots s_{v_1}(q))$,
  we have $\alpha^{(m)} \in \mc{F}(Q)$ and
  $\alpha^{(i)}_{v_i} < \alpha^{(i-1)}_{v_i}$.  Then $\alpha$ is flat
  if and only if (a) $\alpha^{(m)} \in \mc{F}(Q)$ is flat, and
 (b) for every $i$, either (b1)
  $(q^{(i)},\theta^{(i)},\alpha^{(i)})$ is an admissible reflection
  (Definition \ref{d:admiss-refl}) of
  $(q^{(i-1)},\theta^{(i-1)},\alpha^{(i-1)})$ (i.e.,
  $q^{(i-1)}_{v_i}\neq 1$ or $\theta^{(i-1)}_{v_i} \neq 0$), or (b2)
  $\alpha^{(i)}$ is a $(-1)$-reflection of $\alpha^{(i-1)}$.
\end{thm}
Analogously to \cite{Su-fmmrq}, we will show below that $\alpha^{(m)} \in \mc{F}(Q)$ is flat
if and only if it is not of the form $m \ell \delta$ for $\delta$ the minimal imaginary root of an affine Dynkin subquiver, $m \geq 2$, and $\ell \geq 1$ is such that $q^\delta$ is a primitive $\ell$-th root of unity.

The theorem actually gives an algorithm to determine if a root is
flat, by playing a variant of the numbers game \cite{Moz-rfgWg} (with
a cutoff in the inadmissible case as in \cite{GS-dmWge}).
\begin{proof}[Proof of Theorem \ref{t:flat}]
  Under an admissible reflection the condition of being flat does not
  change since
  $s_{v_i}: R^+_{q^{(i-1)},\theta^{(i-1)}} \to
R^+_{q^{(i)},\theta^{(i)}}$
is a bijection (as $e_{v_i} \notin R^+_{q^{(i-1)},\theta^{(i-1)}}$),
cf.~\cite[Lemma 5.2]{cb-geom}.  Here we let $q^{(0)} := q$ and $\theta^{(0)} := \theta$.  If we apply a $(-1)$-reflection, we
claim that the condition of being flat does not change. We only have
to show that, if
$\alpha^{(i)} \in \widetilde{\Sigma}_{q^{(i)},\theta^{(i)}}$, then
also
$\alpha^{(i-1)} \in \widetilde{\Sigma}_{q^{(i-1)},\theta^{(i-1)}}$,
since the converse follows immediately from the definition of flat.
Suppose on the contrary that
$\alpha^{(i-1)} = \beta^{(1)} + \cdots + \beta^{(k)}$ is a decomposition
with $\beta^{(j)} \in R^+_{q^{(i-1)},\theta^{(i)}}$ and
$p(\alpha^{(i-1)}) < p(\beta^{(1)}) + \cdots + p(\beta^{(k)})$. Since
$(\alpha^{(i-1)}, e_{v_i}) = 1$, there must exist $j$ with
$(\beta^{(j)}, e_{v_i}) \geq 1$, and hence also
$\beta^{(j)}_{v_i} \geq 1$. Set
$\gamma^{(\ell)} := \beta^{(\ell)} - \delta_{\ell j} e_{v_i}$ for all
$\ell$. Note that $p(\gamma^{(j)}) \geq p(\beta^{(j)})$, with equality
if and only if $(\beta^{(j)},e_{v_i}) = 1$. Then
\[
p(\alpha^{(i)}) = p(\alpha^{(i-1)}) < p(\beta^{(1)}) + \cdots + p(\beta^{(k)})
 \leq  p(\gamma^{(1)}) + \cdots + p(\gamma^{(k)}),
\]
so that $\alpha^{(i)} \notin \widetilde{\Sigma}_{q^{(i)},\theta^{(i)}}$. We have proved the contrapositive.

It remains to show that, if $(\alpha^{(i-1)}, e_{v_i}) > 1$ and
$(q_{v_i},\theta_{v_i})=(1,0)$, then
$\alpha^{(i-1)} \notin \widetilde{\Sigma}_{q^{(i-1)},\theta^{(i-1)}}$
In this case there can be no loops at $v_i$, so that $e_{v_i}$ is a
real root.  Moreover, $\alpha^{(i-1)}_{v_i} \geq 1$. Then,
$p(\alpha^{(i-1)}-e_{v_i}) = p(\alpha^{(i-1)}-e_{v_i}) + p(e_{v_i}) >
p(\alpha^{(i-1)})$.  So $\alpha^{(i-1)}$ is not flat.  
\end{proof}
\begin{rem}
  The same theorem as above applies in the additive case, to
  characterise the analogous set $\widetilde{\Sigma}_{\lambda,\theta}$
  of flat roots. Also, note that when $\theta=0$, the above proof
  simplifies the proof of \cite[Theorem 1.2]{Su-fmmrq}, since it does not require
  the classification \cite[Theorem 8.1]{cb-geom} of roots in
  $\mc{F}(Q) \setminus \Sigma_{\lambda,0}$.
\end{rem}
\begin{rem}Thanks to \cite[Theorem 3.1]{GS-dmWge}, the condition that any (or every) reflecting sequence consists only of admissible and $(-1)$-reflections is equivalent to the condition that, for every real root $\beta \in R^+_{q, \theta}$, we have $(\alpha, \beta)\leq 1$.
\end{rem}

\subsection{Fundamental and flat roots not in $\Sigma_{q,\theta}$}
To complete the characterisation, we need to determine the set
$\mc{F}(Q)\setminus \widetilde{\Sigma}_{q,\theta}$. This follows from
\cite[Theorem 8.1]{cb-geom}, which computes
$\mc{F}(Q) \setminus \Sigma_{\lambda}$ (in the additive case, but which
extends to the present setting). We state a sharper and more general
version:
\begin{thm}\label{t:flat-sigma}
  A root in $R^+_{q,\theta}$ is not in $\Sigma_{q,\theta}$ if and only if,  applying a
  reflecting sequence as in Theorem \ref{t:flat}, either one of the
  reflections is inadmissible, or the resulting element of $\mc{F}(Q)$ is
  one of the following:
\begin{enumerate}
\item[(a)] $m\ell\delta$ with $\delta$ the indivisible imaginary root for an affine Dynkin subquiver, $m \geq 2$, and $\ell$ is such that  $q^\delta$ is a primitive $\ell$-th root
  of unity; or 
\item[(b)] The support of $\alpha$ is $J \sqcup K$
for $J,K \subseteq Q_0$ disjoint subsets 
with exactly one arrow  in $\overline{Q}_1$ from a vertex $j \in J$ to a vertex $k \in K$, 
$\alpha_k = 1$,  and either:
\begin{enumerate}
\item[(b1)] $\alpha_j=1$, and $\alpha|_J \in R^+_{q,\theta}$, or
\item[(b2)] $Q|_J$ is affine Dynkin and $\alpha|_J = m \delta$ for some $m \geq 2$, with $\delta \in R^+_{q,\theta}$ indivisible 
  and $j$ an extending vertex of $Q|_J$.
\end{enumerate}
\end{enumerate}
Moreover, the root is not in $\widetilde{\Sigma}_{q,\theta}$ if and
only if one of the reflections is neither admissible nor a
$(-1)$-reflection, or the resulting element of $\mc{F}(Q)$ is in case (a).
\end{thm}
\begin{proof}
  First, it is clear that if an inadmissible reflection is applied in
  the sequence, $p(\alpha) \leq p(s_i \alpha) + p(\alpha-s_i\alpha)$ shows
  that $\alpha$ is not in $\Sigma_{q,\theta}$. So we can assume all
  reflections are admissible.  By \cite[Theorem 8.1]{cb-geom} in the
  additive case (with $\theta=0$), there is a sequence of admissible
  reflections resulting in one of the given cases (which is not assumed to be in
   $\mc{F}(Q)$). The proof in \emph{loc.~cit.}~extends verbatim to our case,
  replacing $N_{\lambda}$ by $N_{q,\theta}$. Although the additive
  statement has $\ell=1$ (or in characteristic $p$ has $\ell=p$), for
  us we note that if $q^\delta$ is a primitive $\ell$-th root of unity, i.e., $\ell \delta$ is $q$-indivisible, then 
  $\ell\delta \in \Sigma_{q,\theta}$, since every element $\beta \in
  N_{q,\theta}$ with $\beta < \ell
  \delta$ is real.  Note that, in \cite[Theorem 8.1]{cb-geom} the
  condition in (b2) that $\delta \in
  R^+_\lambda$ is not stated (since its goal is to produce
  non-exhaustive necessary conditions for $\alpha \in
  \Sigma$), but it follows from the proof that it is also a necessary
  condition for $\alpha \notin \Sigma$.

  We claim that in fact we can take the result to be in $\mc{F}(Q)$.
  Note that, in \cite{cb-geom}, it is not required in conditions
  (a),(b) that $\alpha$ be in $\mc{F}(Q)$. However, we already
  know that, in order to have $\alpha \in \Sigma_{q,\theta}$, there
  must be an admissible reflection sequence taking $\alpha$ to
  $\mc{F}(Q)$. Thus we may assume that $\alpha \in \mc{F}(Q)$ and
  moreover that it is sincere.  Then, applying \cite[Theorem
  8.1]{cb-geom}, there is a further sequence of admissible reflections
  taking $\alpha$ to one of the forms above.  After this we can apply
  an admissible reflection sequence to get back to an element of
  $\mc{F}(Q)$, necessarily $\alpha$ again. It is clear that doing so
  will not change the form as above, since $\alpha$ was assumed to be
  sincere, so the reflections cannot shrink the support of $\alpha$;
  note that, in cases (a) and (b2), this means that no reflections will
  be applied to the coefficients of the multiple of $\delta$.

  For the converse, it remains to verify that cases (a) and
  (b1),(b2) are not in $\Sigma_{q,\theta}$. In case (a)
  $p(m\delta) = 1 < mp(\delta)=m$. In case (b1),
  $p(\alpha) = p(\alpha|_J) + p(\alpha|_K)$. Finally, in case
  (b1), $p (\alpha) = p(\delta)+p(\alpha-\delta)$.

  The statement about flat roots follows from Theorem \ref{t:flat}
  together with the observation that case (a) is not flat (as
  $p(m\ell \delta) < mp(\ell \delta)$), whereas cases (b1) and
  (b2) are flat: this follows from \cite[Theorem 1.1]{Su-fmmrq},
  where it is shown that (b1) and (b2) are already flat as
  elements of $\Sigma_{1,0}$ (which is stronger).
\end{proof}

\subsection{Canonical decompositions}\label{ss:cd}
Let $\Sigma_{q,\theta}^{\text{iso}} \subseteq \Sigma_{q,\theta}$ be the subset
of isotropic imaginary roots. We use the notation $\N_{\geq 2}\cdot \Sigma_{q,\theta}^{\text{iso}} := \{m \alpha \mid m \geq 2, \alpha \in \Sigma_{q,\theta}^{\text{iso}}\}$.
\begin{thm}\label{t:decompo}
(i)  Given $\alpha \in N_{q,\theta}$, 
there exists
a unique decomposition $\alpha = \alpha^{(1)} + \cdots + \alpha^{(m)}$ with
$\alpha^{(i)} \in \Sigma_{q,\theta}$ such that any other such decomposition is
a refinement of this one.

(ii) There is also a unique decomposition $\alpha = \beta^{(1)} + \cdots + \beta^{(k)}$ with $\beta^{(i)} \in \widetilde{\Sigma}_{q,\theta} \cup \N_{\geq 2}\cdot \Sigma_{q,\theta}^{\text{iso}}$, satisfying the properties:
\begin{enumerate}
\item[(a)] Every element $\beta^{(i)}$ is of one of the following three types:
  \begin{itemize}
  \item[(1)] $\beta^{(i)} \in \Sigma_{q,\theta}$;
  \item[(2)] $\beta^{(i)} \in \N_{\geq 2} \cdot \Sigma_{q,\theta}^{\text{iso}}$;
  \item[(3)]
    $\beta^{(i)} \in \widetilde{\Sigma}_{q,\theta} \setminus
    \Sigma_{q,\theta}$; moreover, there is an admissible reflection
    sequence taking $\beta^{(i)}$ to an element of the fundamental
    region having only decompositions of the form (b2) in Theorem
    \ref{t:flat-sigma}.
  \end{itemize}
\item[(b)] Any other decomposition into
  $\widetilde{\Sigma}_{q,\theta}\cup \N_{\geq 2}\cdot \Sigma^{\text{iso}}_{q, \theta}$ satisfying (a) and (b) is a refinement
  of this one.
\end{enumerate}


(iii) The decomposition in (i) is a refinement of the one in (ii), obtained uniquely, after applying admissible reflections,
by performing decompositions $\alpha=\alpha|_J + \alpha|_K$ of type (b2) in Theorem \ref{t:flat-sigma}.



(iv) The direct sum map produces a Poisson isomorphism (of reduced varieties), using the decomposition in (ii),
\[
\prod_{i=1}^k \mathcal{M}_{q,\theta}(Q,\beta^{(i)}) \iso
\mathcal{M}_{q,\theta}(Q,\alpha).
\]

\end{thm}
Notice the following immediate consequence (of the decomposition in (ii)), giving a weakened version of \eqref{sigma-converse}:
\begin{coro}\label{c:wk-sigma-converse}
  If $\alpha$ is the dimension of a $\theta$-stable representation of
  $\Lambda^q$, then  one of the following three cases must hold: (1) $\alpha \in \Sigma_{q,\theta}$; (2) $\alpha \in \N_{\geq 2} \Sigma_{q,\theta}^{\text{iso}}$; (3) $\alpha$ is obtained by admissible reflections from an element in the fundamental region having only type (b2) decompositions in Theorem \ref{t:flat-sigma}.  
\end{coro}
\begin{rem}
  In the additive situation, the fact that cases (2) and (3) in the preceding corollary cannot occur was very recently given a simpler, unified proof in
  \cite{Crawley-Boevey2019}.
  \end{rem}
Note that, in the real case $\alpha \in \Sigma_{q,\theta}$, it is
obvious from the properties of admissible reflections that $\alpha$ is
the dimension of a $\theta$-stable representation (see \cite[Theorem
1.9]{cb-shaw} for a stronger statement).  In the imaginary case,
following \eqref{sigma-converse}, we only expect a stable
representation if $\alpha \in \Sigma_{q,\theta}$, but it is not at all
clear how to prove its existence.

In the proof of the theorem, we will produce also an algorithm for
constructing the $\beta^{(i)}$, by a sequence of reflections and
subtracting roots $e_i$, with the end result an element in $\mc{F}(Q)$
whose connected components give the imaginary $\beta^{(i)}$. For the
real roots, we obtain the unique decomposition into real roots in
$\Sigma_{q,\theta}$ (as a real root which is the sum of multiple real
roots cannot be in $\Sigma_{q,\theta}$).
\begin{rem}
  Observe that essentially the same proof as that provided below of
  Theorem \ref{t:decompo} was given in \cite{cb-deco} in the context
  of Nakajima quiver varieties, and indeed the result above holds in
  that setting. However, due to the simplifying properties of that
  case (such as expectation (*) holding, and $q$-divisibility
  coinciding with ordinary divisibility), Crawley-Boevey is able to
  show that the product decomposition in (iv) always refines to one
  using the decomposition of (i). Hence the statement given in
  \cite{cb-deco} is substantially simpler, eliminating parts (ii) and
  (iii). 
\end{rem}
\begin{proof}[Proof of Theorem \ref{t:decompo}]
  We first obtain the existence of the desired decompositions in (i) and (ii) satisfying (iii) and (iv).
  We prove this by induction on the sum of the entries of
  $\alpha$. If there is a vertex $v \in Q_0$ at which
  $(\alpha, e_v) > 0$ and either $q_v \neq 1$ or $\theta_v \neq 0$,
  then we can apply an admissible reflection. Since admissible
  reflections preserve the set of flat roots, the statements follows
  from Theorem \ref{t:yama} and the induction hypothesis.  So suppose
  that there is no such vertex.  Instead, suppose that $v \in Q_0$ is
  such that $(\alpha, e_v) > 0$ but $q_v=1, \theta_v=0$.  Then every
  decomposition of $\alpha$ into elements of
  $\widetilde{\Sigma}_{q,\theta}$ must have an element having positive
  Cartan pairing with $e_v$. This cannot happen by definition for the
  imaginary roots. Since $e_v$ is the only real root in
  $\Sigma_{q,\theta}$ with positive pairing with $e_v$, this implies
  that $e_v$ must appear as a summand of every decomposition of types
  (i) and (ii).  Similarly, the argument of the proof of Proposition \ref{p:-1-iso} shows that, in this case, the direct sum map yields an isomorphism
  $\mathcal{M}_{q,\theta}(Q,\alpha-e_v) \times
  \mathcal{M}_{q,\theta}(Q,e_v) \iso
  \mathcal{M}_{q,\theta}(Q,\alpha)$.
  We can apply the induction hypothesis to $\alpha-e_v$.

This reduces the theorem to the case that $( \alpha, e_i ) \leq 0$ for all $i$.
In this case we can decompose $\alpha$ into its connected components.
By Theorem \ref{t:flat-sigma}, all of these are flat roots, except for
elements of the form $m\ell\delta$ with
$\ell\delta \in \Sigma_{q,\theta}^{\text{iso}}$ and $m \geq 2$.

It remains to show that, given the situation (b1)
of Theorem \ref{t:flat-sigma}, we get a decomposition of our moduli space.
This follows from the 
arguments of \cite[\S 10, II]{cb-geom}. We briefly repeat them for the reader's convenience, adapting them to our situation (see \emph{loc.~cit.~} for details). Restrict the quiver to the vertices $\{j\} \cup K$ and the arrows incident only to these vertices. Let $a, a^*$ be the pair of reverse arrows between $j$ and $k$; without loss of generality suppose $a: j \to k$. Adding the relations at all the vertices of $K$, and using that $\alpha|_K \in N_{q,\theta}$ (since $\alpha$ and $\alpha|_J$ are), we obtain that $a^* a$ is equal to a sum of commutators. It thus has trace one, and since $\alpha_j=1$, it is zero.  Therefore either $a^*$ or $a$ acts by zero. In the former case we obtain a quotient representation with dimension vector $\alpha|_K$; in the latter we obtain a subrepresentation with this dimension.  These representations are $\theta$-semistable since $\theta \cdot \alpha|_K = 0$, and the original representation was $\theta$-semistable. So the polystabilisation of our representation decomposes into a direct sum of representations with dimension vectors $\alpha|_J$ and $\alpha|_K$.  Therefore the direct sum map $\mathcal{M}_{q,\theta}(Q,\alpha|_J) \times \mathcal{M}_{q,\theta}(Q, \alpha|_K) \to \mathcal{M}_{q,\theta}(Q,\alpha)$ is an isomorphism.

This
yields the desired decomposition in (ii), as well as the isomorphism
in (iv), and the decomposition in (iii).

Let us prove that the decompositions are unique. For (i), this is done in \cite[Theorem 1.1]{cb-deco}; the proof carries over verbatim, replacing $\Sigma_\lambda$ by $\Sigma_{q,\theta}$ and adapting all notions.  Let us consider (ii), whose proof is similar.  
We can obviously assume that $\alpha$ is sincere. We first claim 
that the uniqueness statement is unaffected by applying the aforementioned reduction to the fundamental region.
It is obvious that applying admissible reflections does not change the statement.  So we only have to show that, if  $(\alpha, e_v) > 0$, then every decomposition of type (ii) includes $e_v$.
In this case, any decomposition of the form (ii) must have $(\beta^{(i)}, e_v) > 0$ for some $v$.
If $\beta^{(i)} \neq e_v$, then we must have $\beta^{(i)} \notin \N \cdot \Sigma_{q,\theta}$. Thus $\beta^{(i)} \in \widetilde{\Sigma}_{q,\theta}$, and by assumption (3), it must
be related by admissible reflections to an element
of the fundamental region. Therefore, it
 does not have a positive pairing with any real roots.  This is a contradiction. We  therefore obtain that $\beta^{(i)} = e_v$ for some $i$.  Thus the uniqueness statement for $\alpha$ is equivalent to that for $\alpha-e_v$, as desired. 

This reduces us to the case that $\alpha$ is in the fundamental region. We clearly can get a decomposition in a unique way by iteratively replacing $\alpha$ by the sum $\alpha|_J  + \alpha|_K$ as in (b1) of Theorem \ref{t:flat-sigma}. We only have to show that every decomposition of the form (ii) refines such a decomposition of type (b1).
For a contradiction, suppose $\alpha$ is in the fundamental region and we have a decomposition of type (b1) with sets $J$ and $K$, but also a decomposition of type
(ii) with some $\beta^{(i)}$ not supported entirely on $J$ or $K$. After applying admissible reflections to $\beta^{(i)}$, this property continues to hold, since we cannot perform admissible reflections at the vertices $j$ and $k$.  So we can assume $\beta^{(i)}$ is itself in the fundamental region.  This contradicts our assumptions on $\beta^{(i)}$.   \qedhere

\end{proof}
As a consequence, we obtain the following description of
divisibility criteria for elements of $\widetilde{\Sigma}_{q,\theta}$ and $\Sigma_{q,\theta}$, analogous to \cite[Theorem
2.2]{bellamy-schedler}:
\begin{coro}\label{c:indiv} 
 Let $\alpha = m\beta$ for $\beta \in R^+_{q,\theta}$ $q$-indivisible and imaginary, and $m \geq 2$.
 Then,  $\alpha \in  \widetilde{\Sigma}_{q,\theta}$ if and only if $\beta \in \widetilde{\Sigma}_{q,\theta}$, $\beta$ is anisotropic, and  a reflecting sequence taking $\beta$ to the fundamental region involves only admissible reflections.
 In this case, also $\alpha \in \Sigma_{q,\theta}$.

   In particular, for $\gamma \in \widetilde{\Sigma}_{q,\theta}$, every rational multiple $r \gamma \in N_{q,\theta}$ for $r \in Q_{\leq 1}$ is also in $\widetilde{\Sigma}_{q,\theta}$. 


\end{coro}
\begin{proof}
  This follows from the classification of flat roots in Theorems
  \ref{t:flat}, \ref{t:flat-sigma}, and \ref{t:decompo}. Observe simply
  that if a reflection sequence for $\beta$ involves an inadmissible
  $(-1)$-reflection, then the same sequence for $\alpha$ involves an
  inadmissible $(-2)$-reflection (which is not allowed).
  On the other hand, if only admissible reflections are allowed, then $m\beta$ will also be flat unless $\beta$ is isotropic. In the anisotropic case, $m\beta \in \Sigma_{q,\theta}$, since the decompositions of type (b1) and (b2) cannot occur for a divisible vector. The last statement then follows by considering $\gamma$ and $r\gamma$ as multiples of a common vector. \qedhere



\end{proof}
\begin{rem} 
  Although we are working with reduced varieties throughout the paper,
  we emphasised this in Theorem \ref{t:decompo}.(iv) because it is not
  completely clear that the reflection isomorphisms in \cite[Theorem
  5.1]{yamakawa} are defined scheme-theoretically. On the other hand,
  this is the only obstacle here. That is, if these isomorphisms are
  defined scheme-theoretically, then the proof would appear to extend
  to this case, i.e., to not-necessarily-reduced multiplicative quiver
  schemes.
\end{rem}


\subsection{Symplectic resolutions for $q$-indivisible flat roots}
\begin{thm}\label{t:res-flat} Suppose that $\alpha \in \widetilde{\Sigma}_{q,\theta}$ is $q$-indivisible and $\mm_{q,\theta}(Q,\alpha)$ is non-empty.
  Then for suitable
$\theta' \geq \theta$, $\mm_{q,\theta'}(Q,\alpha) \to \mm_{q,\theta}(Q,\alpha)$ is a symplectic resolution.
\end{thm}
\begin{rem} Observe that the theorem also holds in the additive
  setting, where the result is also interesting.  Indeed, it explains
  and generalises the technique of framing used to construct
  resolutions such as Hilbert schemes of $\C^2$ or of hyperk\"ahler
  ALE spaces.  In the former case, the quiver is again the framed Jordan quiver
  (with two vertices and two arrows, a loop at the first vertex, and an arrow from the second to the first vertex).
  The dimension vector is $\alpha=(m,1)$.
  The theorem recovers the well-known statement that taking
  $\theta \neq 0$ gives a symplectic resolution of the singularity
  $\operatorname{Sym}^m \C^2$ in the additive case; this identifies
  with $\operatorname{Hilb}^m \C^2$. In the multiplicative case for
  the same quiver, by Theorem \ref{t:iso-mult-char} and Remark \ref{r:imc-open}, after
  localisation, we obtain a resolution of the character variety of the
  once-punctured torus in the multiplicative case for rank $m$ local
  systems with unipotent monodromy $A$ satisfying
  $\operatorname{rk}(A-I) \leq 1$.
\end{rem}
\begin{proof}[Proof of Theorem \ref{t:res-flat}]
  In view of Lemma \ref{l:varytheta},
   we only have to show that we can find
  $\theta' \geq \theta$ such that $\mm_{q,\theta'}(Q,\alpha)$ is smooth
  and $\mm_{q,\theta'}(Q,\alpha) \to \mm_{q,\theta}(Q,\alpha)$ is
  birational. We will make use of the combinatorial analysis of
  \cite[Section 8]{cb-geom}. Note that, if $\alpha \in \Sigma_{q,\theta}$, then the result follows from the discussion in Section \ref{ss:sing-alpha}, so we can
  assume that this is not the case.

  Let us take $\theta'$ generic such that the conditions of Remark \ref{r:rational stability} are satisfied (i.e., $\theta'$ is an integral multiple of a generic rational stability condition in a small neighbourhood of $\theta$). In particular this means that $\theta' \geq \theta$, every $\theta$-stable representation is $\theta'$-stable, and $\theta' \cdot \beta \neq 0$ for any $\beta < \alpha$ with $\beta \in N_{q,\theta}$.




By Corollary \ref{c:dense-flat}, for each connected component of
$\mm_{q,\theta}(Q,\alpha)$, there is a dense stratum of the form
$(1,\beta^{(1)};\ldots;1,\beta^{(r)})$ with
$p(\alpha)=\sum_{i=1}^r p(\beta^{(i)})$.  We need to show that each
representation $\rho$ in such a stratum is in the boundary of a unique
$\GL(\alpha)$-orbit in
$\mr{Rep}^{\theta'-s}(\Lambda^q(Q),\alpha)$. Equivalently, we
must show that there is a unique $\theta'$-stable representation
$\rho'$ up to isomorphism such that $\rho$ is the
$\theta$-polystabilisation of $\rho'$.

We will prove the statement by induction on $\alpha$, with respect to
the partial ordering $\leq$.  First, applying admissible and
$(-1)$-reflections, we reduce to the case that $\alpha$ is in the
fundamental region. Indeed, it is clear from Theorem \ref{t:yama} and
Proposition \ref{p:-1-iso} that applying these reflections causes no
harm. Note that each $(-1)$-reflection will modify stratum types by
removing a real root from the type; once we are in the fundamental region no real roots will appear.

We first show uniqueness. If $\rho'$ is as
above, then suppose that there is an exact sequence of $\theta$-semistable
representations of the form
$0 \to \psi \to \rho' \to \phi \to 0$.
By our assumptions, the dimension vectors of $\psi$ and $\phi$ are sums
of complementary subsets of the $\beta^{(i)}$.
It follows from the proof of
Theorem \ref{t:flat-sigma} (using \cite[Section 8]{cb-deco}) that
there is a corresponding decomposition of $\alpha$ as in Theorem
\ref{t:flat-sigma}, of type (b1) or (b2), with the following property:
in type (b1), $\alpha^{(1)} := \alpha|_J$ and
$\alpha^{(2)} := \alpha|_K$ are the dimension vectors of $\psi$ and
$\phi$, in either order; or, in type (b2), $\alpha^{(1)} := \delta$
and $\alpha^{(2)} := \alpha-\delta$ are these dimension vectors,
again in some order.
Note that the ordering of the $\alpha^{(i)}$ is fixed
by the conditions that $\dim \psi \cdot \theta < 0$ and
$\dim \phi \cdot \theta > 0$.  Next, since
$(\alpha^{(1)}, \alpha^{(2)}) = -1$, it follows from Proposition
\ref{dim-ext} that
$\dim \mr{Ext}^1(\psi, \phi)=\dim \mr{Ext}^1(\phi,\psi) = 1$. So the
extension $\rho'$ is uniquely determined, up to isomorphism, from
$\psi$ and $\phi$.

We claim that $\psi$ and $\phi$ are uniquely determined from their
dimension vectors up to isomorphism.  We give the argument for $\psi$;
the one for $\phi$ is symmetric. Let $C \subseteq \{1,\ldots,r\}$ be a
subset of indices such that $\dim \psi = \sum_{i \in C} \beta^{(i)}$.
(This set is unique except in case (b2) with
$\dim \psi= \alpha^{(2)}=\alpha-\delta$.)  There exists a unique
$i \in C$ such that $(\beta^{(i)}, \dim \phi) = -1$ (since
$(\alpha^{(1)}, \alpha^{(2)}) = -1$).  Now, define $\theta'':= \theta'|_{\text{supp} \dim \psi} - (\theta' \cdot \dim \psi) e_v$,
where $v$ is 
the unique vertex in $\text{supp} \dim \psi$ which has non-zero Cartan pairing with $\text{supp} \dim \phi$ (so $(\dim \psi)_v = 1$). By construction, $\theta'' \cdot \dim \psi = 0 = \theta'' \cdot \dim \phi$. Moreover,
$\theta'' \cdot \beta^{(j)} = \theta' \cdot \beta^{(j)}$ for $j \in C \setminus\{i\}$.  We claim that $\psi$ is $\theta''$-stable.
 By definition of $\theta'$-stability, every nonzero submodule $\eta$ of $\psi$ 
satisfies
$\theta' \cdot \dim \eta < 0$.
Now, if $\beta^{(i)} \not \leq \dim \eta$, then
$\theta'' \cdot \dim \eta = \theta' \cdot \dim \eta < 0$.
On the other hand, if $\eta$ is a proper submodule of $\psi$
with $\beta^{(i)} \leq \dim \eta$, then $\psi/\eta$ is a nonzero quotient
module with $\beta^{(i)} \not \leq \dim (\psi/\eta)$.  Then
$(\dim (\psi/\eta), \dim \phi) = 0$. By Proposition \ref{dim-ext}, $\mr{Ext}^1(\phi,\psi/\eta)=0$. Therefore, we have an exact sequence $0 \to \eta \to \rho' \to (\psi/\eta) \oplus \phi \to 0$. As a consequence, $\psi/\eta$ itself is a quotient module of $\rho'$.  It follows that $\theta' \cdot \dim (\psi/\eta) > 0$. Therefore, $\theta'' \cdot \dim (\psi/\eta) = \theta' \cdot (\psi/\eta) > 0$.
Therefore, $\theta'' \cdot \eta < 0$. We conclude that $\psi$ is $\theta''$-stable, as desired. By induction on $\alpha$, $\psi$ is then uniquely determined up to isomorphism.

This completes the proof of uniqueness. We move on to existence, which is similar.  Begin with a decomposition given by Theorem \ref{t:flat-sigma} of type (b1) or (b2). Let us keep the notation $\alpha^{(1)}, \alpha^{(2)}$ defined above. The same construction as above yields modifications $\theta^{(1)}, \theta^{(2)}$ of $\theta'$ such that $\theta^{(i)} \cdot \alpha^{(i)} = 0$.  By induction we can take $\theta^{(i)}$-stable representations $\phi, \psi$ of dimension vectors $\alpha^{(i)}$.  Then since $(\alpha^{(1)}, \alpha^{(2)}) = -1$, $\dim \mr{Ext}^1(\phi, \psi) = \dim \mr{Ext}^1(\psi, \phi)=1$. Assume that $\theta' \cdot \phi < 0$, otherwise swap $\phi$ and $\psi$. Then form a non-trivial extension $0 \to \phi \to \rho' \to \psi \to 0$.  The same computation as above guarantees that $\rho'$ is $\theta'$-stable.
\end{proof}
\begin{coro}\label{c:res-flat}
In the situation of the proposition, the normalisation of $\mm_{q,\theta}(Q,\alpha)$ is a symplectic singularity.
\end{coro}
\begin{proof} This follows
  since we have constructed a symplectic resolution (see Proposition \ref{prop-symp-sing} or Remark \ref{r:pss}).  
\end{proof}
\begin{rem} Note that the main step of the proof is to show that
  $\mm_{q,\theta'}(Q,\alpha) \to \mm_{q,\theta}(Q,\alpha)$ is
  birational for suitable $\theta' \geq \theta$. For this, we did not
  need the hypothesis that $\alpha$ is $q$-indivisible.  On
  the other hand, by Theorems \ref{t:flat} and \ref{t:flat-sigma},
  when
  $\alpha \in \widetilde{\Sigma_{q,\theta}} \setminus
  \Sigma_{q,\theta}$,
  $\alpha$ is actually indivisible (not merely
  $q$-indivisible).  For $\alpha \in \Sigma_{q,\theta}$,
  the birationality statement is Corollary \ref{c:birational}, which
  is easy. (Moreover, the full statement of Theorem \ref{t:res-flat}
  was established for $\alpha \in \Sigma_{q,\theta}$ in Section
  \ref{ss:sing-alpha}.) So it does not really add anything to state the birationality property without the $q$-indivisibility hypothesis.
\end{rem}

\subsection{Symplectic resolutions for general $\alpha$} 


\begin{thm}\label{t:main-result-nonsigma}
Assume that $\mm_{q,\theta}(Q,\alpha)$ is non-empty and that
the decomposition of Theorem \ref{t:decompo}.(ii) has no
elements $\beta^{(i)}$ of the forms 
(a) $\beta^{(i)}=2\gamma$ for $\gamma \in N_{q,\theta}$
and $p(\gamma)=2$, or (b)
 $\beta^{(i)}=m\gamma$
for $m\geq 2$ and $\gamma \in \Sigma_{q,\theta}^{\text{iso}}$.  Then:
\begin{itemize}
\item The normalisation of $\mathcal{M}_{q,\theta}(Q,\alpha)$ is a symplectic singularity;
\item   Each factor $\mm_{q,\theta}(Q,\beta^{(i)})$ with $\beta^{(i)} \notin \Sigma_{q,\theta}$ admits a symplectic resolution;
\item If for any factor $\beta^{(i)}$ there exists a $\theta$-stable
  representation of dimension
  $\gamma^{(i)} = \frac{1}{m} \beta^{(i)}$ with $m \geq 2$, then $\mm_{q,\theta}(Q,\alpha)$ does not admit a symplectic resolution. In fact, it has an open, singular, terminal, factorial subset.
\end{itemize}
\end{thm}
\begin{proof} The first statement follows if we show that the
normalisation of each factor $\mathcal{M}_{q,\theta}(Q,\beta^{(i)})$ is a symplectic singularity. For the factors such that $\beta^{(i)}$ is in $\Sigma_{q,\theta}$, this is a consequence of Theorem \ref{main-result}.  For  $\beta^{(i)} \notin \Sigma_{q,\theta}$, after applying admissible reflections it
  follows from Theorem \ref{t:flat-sigma} that it is indivisible. Hence $\beta^{(i)}$ is itself indivisible. By our assumptions, $\beta^{(i)}$ is flat. The result then follows from Theorem \ref{t:res-flat}. This also proves the second statement.

We proceed to the third statement.
 Under the hypotheses, since we have excluded the isotropic and 
$(2,2)$-cases, an open subset of 
  $\mathcal{M}_{q,\theta}(Q,\beta^{(j)})$ is factorial terminal
  singular by Theorem \ref{main-result} (see Theorem \ref{t:fact-term}).
  Hence so is an open subset of
  $\mathcal{M}_{q,\theta}(Q,\alpha)$, which therefore does not admit a symplectic resolution.
\end{proof}
\subsection{Classifications of symplectic resolutions of punctured character varieties}
Here, we combine the results of this section and Theorem \ref{comb1} to get a classification of all the character varieties of punctured surfaces which admit a symplectic resolution, modulo the conjectural results of the $(2,2)$-cases. As explained in Section \ref{section-char}, in order to get such a result, it suffices to consider multiplicative quiver varieties of crab-shaped quivers, where the parameter $q$ and the dimension vector $\alpha$ are chosen in an appropriate way; see Theorem \ref{t:iso-mult-char}. 

Let $Q$ be a crab-shaped quiver, $q\in(\C^\times)^{Q_0}$ and $\alpha \in N_{q, \theta}$ and consider the corresponding quiver variety $\mm_{q, \theta}(Q, \alpha)$. Then, if $\alpha \notin \mc{F}(Q)$, we can apply the algorithm of Theorem \ref{t:decompo} and obtain a decomposition where the dimension vectors of the factors are in the fundamental region (such dimension vectors are the connected components of the reflection of $\alpha$). Moreover, note that, in the crab-shaped case, all the vector components not containing the central vertex are Dynkin quivers of type $A$: therefore, the associated multiplicative quiver variety is just a point. This implies that we can assume, without loss of generality, that $\alpha$ be sincere ($\alpha_i > 0$ for all $i$) and in the fundamental region. After having performed this reduction, we can prove the following. 
\begin{coro}\label{c:crab-sr}
  Let $Q$ be a crab-shaped quiver and $\alpha \in N_{q,\theta}$ a sincere vector in the fundamental region.
Further assume that $(Q,\alpha)$ is \textbf{not} one of the following cases: 
\begin{itemize}
\item[(a)] $\beta=\frac{1}{2}\alpha \in N_{q,\theta}$ and
  $(Q,\beta)$ is one of the quivers in Theorem \ref{comb1} and Theorem \ref{comb2};
\item[(b)] $Q$ is affine Dynkin (of type
  $\tilde{A}_0$ (i.e., the Jordan quiver with one vertex and one arrow), $\tilde{D}_4$ or $\tilde{E}_6, \tilde{E}_7, \tilde{E}_8$)
  and $\alpha$ is a $q$-divisible multiple of the indivisible imaginary root $\delta$ of $Q$.
\end{itemize}
Then:
\begin{itemize}
\item The normalisation of $\mm_{q,\theta}(Q,\alpha)$ is a symplectic singularity;
\item If $\alpha$ is $q$-indivisible, 
$\mm_{q,\theta}(Q,\alpha)$ admits a symplectic resolution;
\item 
  If $\alpha=m\beta$ for $m \geq 2$ and there exists a $\theta$-stable
  representation of $\Lambda^q(Q)$ of dimension $\beta$,
then $\mm_{q,\theta}(Q,\alpha)$ does not admit a symplectic resolution (it contains an open singular factorial terminal subset);
\item In the case that $\alpha$ is $q$-divisible, the condition of the preceding part is always satisfied, except possibly in the case: (c) $Q = Q^e \cup \{*\}$
is an affine Dynkin quiver $Q^e$ of type $\tilde A_0, \tilde D_4, \tilde E_6, \tilde E_7$, or $\tilde E_8$ 
together with an additional vertex $\{*\}$
and an additional arrow from this vertex to one with dimension vector $1$ in $\delta$, and $\alpha$ has
  the form
  $(p,p\ell\delta)$ for $p \geq 2$ a prime, $\ell\delta$ $q$-indivisible. Here
  $\delta$ denotes the indivisible imaginary root of $Q^e$.
\end{itemize}
\end{coro}
\begin{rem} In the final part of the corollary, expectation (*) from
  the introduction predicts that the exception indeed fails to satisfy
  the conditions of the preceding part. Nonetheless, we
  believe that, also in this case, there should not exist a projective
  symplectic resolution; this would be implied by Conjecture
  \ref{conj:cy2} in the appendix together with the consequence that
  follows it, by \cite[Theorem 1.5]{bellamy-schedler}.
  \end{rem}
\begin{proof}[Proof of Corollary \ref{c:crab-sr}]
  By Theorem \ref{t:decompo}, we know that $\alpha$ is flat unless it
  is a positive integral multiple of an isotropic root, excluded in
  case (b).  Then the first and third statements follow from Theorems
  \ref{main-result} and the second from \ref{t:res-flat}.  For the
  fourth statement, we apply Theorem \ref{t:flat-sigma}.
  to see that, in the crab-shaped case,
  the
  dimension vector can only be in the fundamental region but not in $\Sigma_{q,\theta}$ if the quiver is a
  framed affine Dynkin quivers of types $\tilde A_0, \tilde D_4, \tilde E_6, \tilde E_7$, and $\tilde E_8$, and the dimension vector is $(1,\ell\delta)$.
  Thus only prime multiples of this vector can be $q$-divisible but have no factor in $\Sigma_{q,\theta}$. 

  It then remains only to show 
  that
  $\mm_{q,\theta}(Q,\beta) \neq \emptyset$ for $\beta \in \Sigma_{q,\theta}$,
  not of the form $(1,\ell\delta)$ for a framed affine Dynkin quiver. Let us first assume that $\theta = 0$. In the star-shaped
  case, the result follows from
  \cite[Theorem
  1.1]{cb-shaw}.  In the crab-shaped case with $g > 0$ loops at the
  central vertex, with $\theta = 0$, it suffices by Theorem \ref{t:iso-mult-char} and Remark
  \ref{r:imc-open} to show
  that, for all conjugacy classes $\mathcal{C}_1, \ldots, \mathcal{C}_m \subset \GL(n, \C)$
  with product of determinants equal to one, there exists a solution
  to the equations $[A_1,B_1]\cdots[A_g,B_g] = C_1 \cdots C_m$ for
  $C_i \in \mathcal{C}_i$. This follows because there is a solution to the
  equation $[X_1,Y_1]=C$ for arbitrary $C \in \SL(n, \C)$ (by
  \cite[Theorems 1, 2]{Tho-csglg}).

  Now assume $\theta \neq 0$. In most cases (excepting the case of one loop and one branch), one can extend \cite[Theorem 1.1]{cb-shaw} to this case using \cite[\S 4.3, 4.4]{yamakawa}; however, we may give a more direct argument. Since
  we are not in the situation
 of a framed affine Dynkin quiver with dimension vector $(1,\ell\delta)$, note that $\alpha \in \Sigma_{q,0}$ as well, so from the $\theta \neq 0$ case and \cite[Theorem 1.11]{cb-shaw}, we know that there exists a simple representation of $\Lambda^q$ of dimension $\alpha$. This is automatically $\theta$-stable, so we obtain the desired non-emptiness statement. 
\end{proof}

\begin{rem}The assumptions made in the above theorem relate to the
  fact that in cases $(a)$ and $(b)$, it is still unknown whether a
  symplectic resolution exists, as this problem seems to be solvable
  only through a deep understanding of the local structure of the
  variety. Nonetheless in case (a), we expect such symplectic
  resolutions to exist and to be constructible by using analogous
  techniques to the ones used by Bellamy and the first author in
  \cite[Theorem 1.6]{bellamy-schedler} (see Remark \ref{r:(2,2)}).
\end{rem}

\subsection{Proof of Theorems \ref{t:char-g0} and \ref{t:char-gg0}} \label{ss:proof-char}
Theorems \ref{t:char-g0} and \ref{t:char-gg0} follow from Corollary
\ref{c:crab-sr}, together with Theorem \ref{t:iso-mult-char}, as
follows.

First, we claim that $q$-divisibility for the collection of conjugacy
classes $\mathcal{C}$ coincides with the same-named property for the
dimension vector $\alpha$ of the corresponding crab-shaped quiver. To
see this, first note that $m \cdot \mathcal{C}$ indeed corresponds to
$m \cdot \alpha$. So we only have to show that the condition that
$\prod_i \det \mathcal{C}_i = 1$ is equivalent to $q^\alpha = 1$.
This is true by construction.

Next, we claim that, for $g=0$, the condition $\ell \geq 2n$ of
Theorem \ref{t:char-g0} is equivalent to the condition that
$\alpha \in \mathcal{F}(Q)$, whereas for $g > 0$, we have
$\alpha \in \mathcal{F}(Q)$ unconditionally. By the chosen ordering of
the $\xi_{i,j}$, we have $(\alpha, e_i) \leq 0$ for all $i \in Q_0$
except possibly the node. There, the condition $\ell \geq 2n$ is
equivalent to $(\alpha, e_i) \leq 0$.  On the other hand, when
$g \geq 1$, then as there is a loop at the node, it is automatic that
$(\alpha, e_i) \leq 0$ for $i \in Q_0$ the node, and hence in this
case $\alpha \in \mathcal{F}(Q)$ automatically.

We now claim that the dimension of $\mathcal{X}(g, k, \mathcal{C})$
equals $2p(\alpha)$ when the quiver is not Dynkin or affine Dynkin and
moreover $\ell \geq 2n$ or $g \geq 1$. This follows
from Theorem \ref{t:flat-sigma} and Theorem \ref{t:iso-mult-char} (see also
Remark \ref{r:imc-open}),
provided that $\alpha$ is not both $q$-divisible and isotropic. However,
the latter conditions,
for $\alpha \in \mathcal{F}(Q)$, are equivalent to saying
that the graph $\Gamma_{\mathcal{C}}$ is affine Dynkin and $\alpha$ is
$q$-divisible.

With the preceding claims established, we proceed to the proof of the theorems.
Note that applying reflections as earlier in this section preserves the property that a dimension vector is one corresponding to a character variety (by Theorem \ref{t:iso-mult-char}).  So we can always reduce to case that $\alpha \in \mathcal{F}(Q)$, unless we end up with something with Dynkin support (hence $\mathcal{X}(g,k,\mathcal{C})$ is a point) or something where $\alpha$ becomes negative (hence $\mathcal{X}(g,k,\mathcal{C})$ is empty). This proves the first part of Theorem \ref{t:char-g0}.  The remaining assertions of the theorems follow from the above claims, which allow us to translate Corollary
\ref{c:crab-sr} into the given results via Theorem \ref{t:iso-mult-char}.

\section{Open questions and future directions}\label{sec-fut}
In this section we pose some questions concerning the cases which are left out from the analysis carried out in the previous sections. This includes the question of to what extent the decomposition in Theorem \ref{t:decompo} can be refined, as in the additive case in \cite{cb-deco}.

One interesting direction of research which naturally arises from the results proved in this paper and the work \cite{tirelli} of the second author is the study of analogous problems in the context of the Higgs bundle moduli spaces, which appear in the picture via the non-abelian Hodge correspondence.

We also say a few words on how one might hope to study the local structure of formal moduli spaces of polystable objects in a 2-Calabi--Yau category and prove that, under suitable conditions, formal neighbourhoods of such moduli spaces are quiver varieties associated to a quiver which arises from the deformation theory of the objects parametrised by the moduli spaces. This is relevant to the present context as it would make it possible to give an alternative and more insightful proof of Proposition \ref{sing-quiv}, as explained
 at the beginning of Section \ref{section-sing}.
 
 Before getting into these issues, we begin by discussing some cases where the multiplicative quiver varieties are known to be non-empty.

 \subsection{Non-emptiness of multiplicative quiver varieties}\label{ss:ne-mqv}
 As explained in the previous sections, one of the subtleties in
 the study of multiplicative quiver varieties is the
 fact that it is not known in general when they are non-empty (nor how
 many connected components they have). On the other hand, there are
 special cases in which non-emptiness can be shown. For example, when
 $q=1$, then, for any quiver $Q$ and any vector $\alpha\in \N^{Q_0}$,
 the zero representation is a suitable element of
 $\mr{Rep}(\Lambda^q, \alpha)$, since the invertibility condition is
 automatically satisfied as well as the multiplicative preprojective
 relation. Thus $\mm_{1,0}(Q,\alpha) \neq \emptyset$. More generally,
 for every real root $\beta \in R^+_{q,\theta}$, then by applying
 reflection sequences as in Section \ref{ss:flat-roots} (see also the
 discussion after Corollary \ref{c:wk-sigma-converse}), we conclude
 that $\mm_{q,\theta}(Q,\beta) \neq \emptyset$. As a result, if
 $\alpha$ can be expressed as a sum of real roots in $R^+_{q,\theta}$
 (not necessarily coordinate vectors) then
 $\mm_{q,\theta}(Q,\alpha) \neq \emptyset$.

 Another important and less trivial case in which we are guaranteed
 that $\mm_{q, \theta}(Q, \alpha)$ is non-empty is when $Q$ is
 crab-shaped (for arbitrary $\alpha \in N_{q,\theta}$): this follows
 from the arguments of the proof of Corollary \ref{c:crab-sr} (relying on
 \cite{cb-monod} and \cite{Tho-csglg}). We remark that, by the
 arguments of
 \cite[\S 4.3, 4.4]{yamakawa}, relying on the
 correspondence between
  character varieties of
  punctured surfaces and moduli of parabolic bundles and 
  \cite[\S 5]{inaba}, it follows that these varieties are in fact irreducible except possibly
  in certain cases of a crab-shaped quiver with a single loop (when, after
  reducing to the fundamental region, the support includes exactly one branch). 


 We can also ask when the stable locus
 $\mm_{q,\theta}^s(Q,\alpha) \neq \emptyset$.  Note that an answer to
 this question for all $\alpha$ also answers the question of
 non-emptiness of the entire locus, since every point in
 $\mm_{q,\theta}(Q,\alpha)$ is represented by a polystable
 representation. More explicitly,
 $\mm_{q,\theta}(Q,\alpha) \neq \emptyset$ if and only if $\alpha$ can
 be represented as a sum of roots $\alpha^{(i)}$ for which
 $\mm_{q,\theta}^s(Q,\alpha^{(i)}) \neq \emptyset$.  Note that, when
 $\alpha \in \Sigma_{q,\theta}$, then non-emptiness of
 $\mm_{q,\theta}(Q,\alpha)$ is equivalent to that of
 $\mm_{q,\theta}^s(Q,\alpha)$, by Proposition \ref{rep-equidim2}.  Our
 expectation (*) says that $\mm_{q,\theta}^s(Q,\alpha) \neq \emptyset$
 implies $\alpha \in \Sigma_{q,\theta}$.
 
 \subsection{Refined decompositions for multiplicative quiver varieties}
 \label{ss:ref-dec}

Recall Crawley-Boevey's canonical decomposition in the additive case (Theorem \ref{thm-prod}, Lemma \ref{lem-prod}).  It is useful to ask to what extent such a decomposition holds in the multiplicative setting, refining the one of Theorem \ref{t:decompo}. Let $\beta^{(i)}, \alpha^{(i,j)}$ be as in Theorem \ref{t:decompo},
and group together the $\alpha^{(i,j)}$ that are equal, yielding distinct
$\gamma^{(i,j)}$ each occurring $r_{i,j} \geq 1$ times.  Note that, when $\gamma^{(i,j)}$ is anisotropic, then $r_{i,j}=1$, since $r_{i,j} \gamma^{(i,j)} \in \Sigma_{q,\theta}$, by the uniqueness of the decomposition in Theorem \ref{t:decompo}.(i).
\begin{conj}\label{conj:cb-decomp}
We have a decomposition as follows:
\begin{equation}
\mm_{q,\theta}(Q,\alpha) \cong \prod_{i,j} S^{r_{i,j}}\mm_{q,\theta}(Q,\gamma^{(i,j)}).
\end{equation}
\end{conj}
 The following proposition partly resolves the conjecture modulo expectation (*). 
\begin{prop}\label{p:decomp-refine}
If (*) holds, then the decomposition of Theorem \ref{t:decompo}.(iv) refines to one of the form
\begin{equation}
\mm_{q,\theta}(Q,\alpha) \cong \prod_{i,j} \mm_{q,\theta}(Q,r_{i,j}\gamma^{(i,j)}).
\end{equation}
Moreover, in this case, the direct sum map
\begin{equation} \label{e:ds-iso}
S^{r_{i,j}} \mm_{q,\theta}(Q,\gamma^{(i,j)}) \to \mm_{q,\theta}(Q,r_{i,j} \gamma^{(i,j)})
\end{equation}
is surjective.
\end{prop}
\begin{proof}
It suffices to decompose each of the $\mm_{q,\theta}(Q,\beta^{(i)})$.  By
Proposition \ref{t:flat-sigma}, the first statement follows by the arguments of \cite[Section 5]{cb-deco} verbatim, replacing simple representations by $\theta$-stable ones. For the second statement, if $r_{i,j} > 1$,  then $\gamma^{(i,j)}$ is isotropic.  Then,  the canonical decomposition of $r_{i,j} \gamma^{(i,j)}$ appearing in Theorem \ref{t:decompo}.(i) is just as a sum of $r_{i,j}$ copies of $\gamma^{(i,j)}$. Thus the statement follows from Theorem \ref{t:decompo}.(i) and expectation (*), since every representation in $\mm_{q,\theta}(Q,r_{i,j} \gamma^{(i,j)})$ is represented by a polystable one.
\end{proof}
Therefore, modulo (*), Conjecture \ref{conj:cb-decomp} reduces to the following statement:
\emph{The natural map \eqref{e:ds-iso} an isomorphism.}
This should have a positive answer if Conjecture \ref{conj:cy2} holds, since as explained in Section \ref{ss:cy}, the multiplicative quiver varieties would formally locally be additive quiver varieties, compatibly with the direct sum map; then the statement reduces to Theorem \ref{thm-prod}.
\begin{eg}
Suppose that $\beta^{(i)}$ is the following dimension vector supported on a framed type $\widetilde E_6$ quiver:
\begin{equation}\begin{tikzcd}
&&n&&&\\
&&2n\arrow[u]&&&\\
n&\arrow[l]2n&\arrow[l]3n\arrow[u]\arrow[r]&2n\arrow[r]&n\arrow[r]&1\\ 
\end{tikzcd}\end{equation}
Then, by the star-shaped case of Theorem \ref{t:iso-mult-char} 
(proved in \cite[Section 8]{cb-shaw}), the variety 
$\mm_{1,0}(Q,\beta^{(i)})$ is isomorphic to
the character variety of rank $3n$ local systems on
the three-punctured sphere $\Sigma_{0,3} = \mathbb{P}^1\setminus \{0,1,\infty\}$ with unipotent monodromies: about the first two punctures, there should be $n$ Jordan blocks of size three (or some refinement), and about the third puncture, there should be $n-2$ Jordan blocks of size three, one Jordan block of size four, and one of size two (or some refinement).  On the other hand, $\mm_{1,0}(Q,\delta)$ is the character variety of rank $3$ local systems
on $\Sigma_{0,3}$ with arbitrary unipotent monodromies.  Conjecture \ref{conj:cb-decomp} then asks whether the first variety is isomorphic to the $n$-th symmetric power of the second; it does not seem so obvious that this should be the case.
\end{eg}
One of the difficulties in trying to adapt the proof of the analogous statement
to Conjecture \ref{conj:cb-decomp} in the additive case 
(\cite[Section 3]{cb-deco}) is that, in the multiplicative case, it is
no longer guaranteed that one of the components of $\gamma^{(i,j)}$
equals one (since $\gamma^{(i,j)}$ need only be $q$-indivisible, not
indivisible). It seems it may be a better approach to prove Conjecture \ref{conj:cy2}, as stated above.

Note finally that the proof of Theorem \ref{thm-prod} (\cite[Theorem
1.4]{bellamy-schedler}), in the additive case, relied on hyperk\"ahler
twistings, for which one needs to assume that the parameter $\lambda$
is real. In fact, some of the issues we face (such as expectation (*))
are not yet resolved, to our knowledge, in the general additive case where both
$\lambda \notin \R$ and $\theta \neq 0$.


\subsection{Symplectic resolutions and singularities}
In view of our results and the flexibility of symplectic
singularities, as well as the relationships between multiplicative and
additive quiver varieties, we propose the following:
\begin{conj} Every multiplicative quiver variety is a symplectic
  singularity.
\end{conj}
Note that a product of Poisson varieties is a symplectic singularity
if and only if each of the factors is (because normality and being
symplectic on the smooth locus have this property, and in the
definition of symplectic singularity it is equivalent to check the
extension property for one or all resolutions of singularities).
Therefore the conjecture reduces to the case of factors appearing in
Theorem \ref{t:decompo}, and if (*) holds, to the case
$\alpha \in \Sigma_{q,\theta} \cup \N_{\geq 2}
\Sigma^{\text{iso}}_{q,\theta}$ by Proposition
\ref{p:decomp-refine}. If Conjecture \ref{conj:cb-decomp} 
holds,
then we can furthermore reduce to the case
$\alpha \in \Sigma_{q,\theta}$. In that case, by Theorem \ref{t:char-gg0}, the only issue is normality, which would be implied if the variety is formally locally an additive one, as predicted by Conjecture \ref{conj:cy2} and the discussion thereafter.

Next, we ask to what extent the property of having a symplectic
resolution is equivalent to the same property for the factors.
\begin{ques} (i) Is it true that $\mm_{q,\theta}(Q,\alpha)$ admits a
  symplectic resolution if and only if each of the factors
  $\mm_{q,\theta}(Q,\beta^{(i)})$ does?

  (ii) Suppose that (*) holds. Is it true that
  $\mm_{q,\theta}(Q,\alpha)$ admits a symplectic resolution if and
  only if each of the factors
  $\mm_{q,\theta}(Q,r_{i,j} \gamma^{(i,j)})$, appearing in Proposition \ref{p:decomp-refine}, does?
\end{ques}
Note that, when $r_{i,j} > 1$, then
$\gamma^{(i,j)} \in \Sigma_{q,\theta}^{\text{iso}}$, and hence is
$q$-indivisible.  Therefore, in this case,
$\mm_{q,\theta}(Q,\gamma^{(i,j)})$ has a symplectic resolution by
varying $\theta$, by Lemma \ref{l:varytheta} and Corollary
\ref{c:birational}. Since it is a surface, so does its $r_{i,j}$-th
symmetric power, by the corresponding Hilbert scheme.  Therefore, if
Conjecture \ref{conj:cb-decomp} holds,
we can ignore these factors in (ii)
above, and only consider the ones with $r_{i,j}=1$. Also notice that all $q$-indivisible factors, including those with $\gamma^{(i,j)} \notin \Sigma_{q,\theta}$, admit symplectic resolutions. Also, if any factor $\gamma^{(i,j)}$ appears which is a $\geq 2$ multiple of the dimension vector of a $\theta$-stable representation, there can be no symplectic resolution of $\mm_{q,\theta}(Q,\alpha)$, nor of $\mm_{q,\theta}(Q,\gamma^{(i,j)})$. So again, for the question (ii), it is enough to consider only the factors $\gamma^{(i,j)}$ which are anisotropic, $q$-divisible, and not a $\geq 2$ multiple of the dimension vector of a $\theta$-stable representation.




\subsection{Moduli of parabolic Higgs bundles and the Isosingularity Theorem}
We restrict the attention to the case of crab-shaped quivers, the
corresponding multiplicative quiver varieties of which, as explained
in Section \ref{section-char}, lead to the study of character
varieties of (possibly non-compact) Riemann surfaces.

Character varieties of compact Riemann surfaces are important, among many reasons, as they appear as the \tit{Betti side} of the non-abelian Hodge correspondence, which is a series of results that establish isomorphisms between apparently unrelated moduli spaces, see \cite{simpson1994-1}. Such a correspondence holds also in the case of non-compact curves, thanks to the work of Simpson, see \cite{simpson-harm}.
 
For the compact case, in \cite{tirelli}, the second author exploited a fundamental result of Simpson, called the Isosingularity Theorem, \cite{simpson1994-2}, to show how the statements proved in \cite[\S 8]{bellamy-schedler} could be translated to the \tit{Dolbeault side} of the non-abelian Hodge correspondence, i.e.,  to the moduli spaces of semistable Higgs bundles of degree $0$. 

In the light of the results of this paper and the non-abelian Hodge correspondence in the non-compact setting, it is a natural question to ask whether an analogue of the main theorems of \cite{tirelli} holds for the Dolbeault moduli spaces defined on complex curves with punctures, which turn out to be the  moduli spaces of \tit{parabolic} Higgs bundles. Before stating some conjectural results, we recall the relevant definitions. 
Motivated by the work of Simpson, \cite{simpson-harm}, we recall filtered local systems, following \cite[\S 4]{yamakawa}, which gives a slightly different but nonetheless equivalent definition from the one given in Simpson, \cite{simpson-harm}. 
\begin{defn}\cite[Definition 4.5]{yamakawa}\label{4.2.3}
	Let $X$ be a compact Riemann surface and $D \subset X$ be a finite subset. 
	Let $L$ be a local system on $X \setminus D$. 
	For a collection of non-negative integers $l=(l_p)_{p \in D}$, 
	a {\em filtered structure} on $L$ of {\em filtration type} $l$ is a collection 
	$(U_p, F_p)_{p\in D}$, where for each $p\in D$:  
	\begin{enumerate}
		\item[(i)] $U_p$ is a neighbourhood of $p$ in $X$ (we set $U_p^* :=U_p \setminus \{ p \}$); and   
		\item[(ii)] $F_p$ is a filtration  
		\[
		L |_{U_p^*} = F_p^0(L) \supset F_p^1(L) \supset \cdots
		\supset F_p^{l_p}(L) \supset F_p^{l_p+1}(L)=0
		\]
		by local subsystems of $L |_{U_p^*}$.
	\end{enumerate}
	
	Two filtered structures $(U_p, F_p)_{p\in D}, (U'_p, F'_p)_{p\in D}$ of the same filtration type 
	are {\em equivalent} if for each $p\in D$, there exists a neighbourhood $V_p \subset U_p \cap U'_p$ 
	of $p$ such that $F_p$ and $F'_p$ coincide on $V^*_p$.
	A local system $L$ together with an equivalence class of filtered structures 
	$F =[ (U_p, F_p)_{p\in D} ]$
	is called a {\em filtered local system} on $(X,D)$ of filtration type $l$. 
\end{defn}

From \cite{yamakawa} one has also the following definition of (semi)stability. 
\begin{defn} \cite[Definition 4.6]{yamakawa}\label{4.2.4} 
	Let $(L,F)$ be a filtered local system on $(X,D)$ of filtration type $l$. 
	Let $\beta = (\beta_p^j \mid p\in D,~j=0,\dots ,l_p)$ be a collection of rational numbers 
	satisfying $\beta_p^i < \beta_p^j$ for any $p$ and $i <j$ -- such a collection is called a {\em weight}. The pair $(L,F)$  is said to be {\em $\beta$-semistable} if 
	for any non-zero proper local subsystem $M \subset L$ the following inequality holds:
	\[
	\sum_{p\in D} \sum_j \beta_p^j \frac{ \mr{rank} \left(  M \cap F_p^j(L) \right)  /\left( 
		M \cap F_p^{j+1}(L) \right)  }{\mr{rank}\ M} 
	\leq
	\sum_{p\in D} \sum_j \beta_p^j \frac{\mr{rank}  \left(  F_p^j(L)/
		F_p^{j+1}(L) \right)  }{\mr{rank}\ L}.
	\]
	$(L,F)$ is {\em $\beta$-stable} if the strict inequality always holds.
\end{defn}
Yamakawa established a correspondence between semistable filtered local systems and multiplicative quiver varieties of star-shaped quivers. This is, as mentioned, a particular case of the correspondence outlined in Section \ref{section-char}. To shed more light on this correspondence, we spell out the correspondence between the parameters $q, \alpha, \theta$ defining a multiplicative quiver variety $\mm_{q, \theta}(Q, \alpha)$ of a star shaped quiver and the weights $\beta$ of a filtered local system $(L, F)$: start with a crab-shaped quiver $Q$ with vertex set $Q_0=\{0, (i, j)_{i\in \{1, \dots n\}, j\in\{1, \dots, l_i\}} \}$ -- i.e. $Q$ has $n$ legs, each of which has length $l_i$, for $i=1, \dots, n$. Moreover, let $\alpha$ be a dimension vector, $\alpha\in \N^{Q_0}$, $\xi\in (\C^\times)^{Q_0}$ a collection of non-zero complex numbers and $\beta\in \Q^{Q_0}$ a collection of rational numbers. Then, on one side, one can consider filtered local systems $(L, F)$ on $(\mb{P}^1, \{p_1, \dots, p_n\})$, where $p_i, i=1, \dots, n$ are pairwise distinct points in $\mb{P}^1$, with stability parameter $\beta$ and such that: 
\begin{enumerate}
	\item $\mr{rank}(L)=\alpha_0$,
	\item $\dim F^j_{p_i}(L)=\alpha_{i, j}$,
	\item the local monodromy of $F^j_{p_i}(L)/F^{j+1}_{p_i}(L)$ around $p_i$ is given by the scalar multiplication by $\xi^j_{p_i}$ for all $i, j$.
\end{enumerate}
On the other side, one can consider the multiplicative quiver variety $\mm_{q, \theta}(Q, \alpha)$, where $Q$ and $\alpha$ are as above and $q$ and $\theta$ are given as 
\[
q_0:=\prod_i(\xi^0_{p_i})^{-1},\ \ \ q_{i, j}=\xi^{j-1}_{p_i}/\xi^j_{p_i}
\]
\[
\theta_0:=\frac{\sum_{i, j}\theta_{i, j}\alpha_{i, j}}{\alpha_0},\ \ \ \theta_{i, j}=\beta^{j}_{p_i}-\beta^{j-1}_{p_i}.
\]

The other main concept in the non-abelian Hodge correspondence on non-compact curves is that of parabolic Higgs bundle, which we recall below (note that, in \cite{simpson-harm} the term \tit{filtered Higgs bundle} is used instead).
\begin{defn} Let $X$ and $D$ be a compact Riemann surface and a reduced divisor on $X$ respectively. Let $E\rightarrow X$ be a holomorphic vector bundle on $X$. A \tit{parabolic structure} on $E$ is the datum of weighted flags $(E_{i,p}, \alpha_{i, p})_{p\in D}$
\[
E_p=E_{1, p}\supseteq E_{2, p}\supseteq\dots\supseteq E_{l+1, p}=0,\]
\[
0\leq \alpha_{1, p}<\dots<\alpha_{l, p}<1,
\] 
for each $p\in D$. A \tit{morphism of parabolic vector bundles} is a morphism of holomorphic vector bundles which preserves the parabolic structure at every point $p\in D$. 
\end{defn}
\begin{defn}
Given a parabolic bundle $(E, (E_{i, p}, \alpha_{i, p})_{p\in D})$, its \tit{parabolic degree} is defined to be 
\[
\mr{pardeg}(E)=\deg(E)+\sum_{p\in D}\sum_i m_i(p)\alpha_{i,p},
\]
where $m_i(p)=\dim E_{i, p}-\dim E_{i, p+1}$ is called the \tit{multiplicity} of $\alpha_{i, p}$.
\begin{rem}Given the notion of parabolic degree, stability and semistability of a parabolic bundle are defined as in the case of vector bundles, using the parabolic degree in place of the ordinary degree, see \cite[\S 2]{logares-martens} for more details and an outline on some geometric properties of the corresponding moduli spaces.
\end{rem} 
\end{defn}

\begin{defn}
A \tit{parabolic Higgs bundle} on $(X, D)$ is a parabolic bundle $(E, (E_{i, p}, \alpha_{i, p})_{p\in D})$ together with a meromorphic map $\Phi:E\rightarrow E\otimes K_X$ with poles of order at most $1$ at the points $p\in D$. The residue of $\Phi$ at marked points is assumed to preserve the corresponding filtration. 
\end{defn}
From the Riemann-Hilbert correspondence, we know that representations of the fundamental group of a punctured surface with fixed monodromies correspond bijectively to filtered local systems. Moreover, in \cite{simpson-harm}, the following theorem is proved. 
\begin{thm}\cite[Theorem, p. 718]{simpson-harm} There is a one-to-one correspondence between (stable) filtered local systems and (stable) parabolic Higgs bundles of degree zero.
\end{thm}
Even though we will not go into the details of this correspondence, we shall at least explain how it works at the level of parameters $\alpha$ and $\beta$. To this purpose, let $X$ be a compact Riemann surface and $D\subset X$ a finite subset of distinct points of $X$. Fixing $p\in D$, consider the sets
\[
\{(\lambda,\alpha_p^j) \in \C \times [0,1)\ |\ 
\text{the action of $\mr{Res}_p \Phi$ on $E_{j, p}/E_{j-1, p}$ has an eigenvalue $\lambda$}\},
\] 
\[
\left\{(\xi,\beta_p^k) \in \C^\times \times \R\ |\ 
\begin{array}{l}
\text{the monodromy of $F_p^k(L)/F_p^{k+1}(L)$} \\ 
\text{along a simple loop around $p$ has an eigenvalue $\xi$}
\end{array}
\right\}.
\]
Then the correspondence between these two sets is explicitly given by $(\lambda, \alpha) \mapsto (\xi, \beta)$, where   
\[
\beta := \alpha - \Re \lambda, \qquad   
\xi := \exp (-2\pi \sqrt{-1} \lambda). 
\]

Another fundamental result of Simpson that is crucially used in \cite{tirelli} is the Isosingularity theorem, which, roughly speaking, states that the moduli spaces of the non-abelian Hodge Theorem in the compact case, i.e.,  with no punctures, are \'etale isomorphic at corresponding points. It is still not known whether the same result holds in the non-compact case. 
\begin{conj} The Isosingularity theorem holds between the moduli space of semistable filtered local systems for fixed parameters and the moduli space of semistable parabolic Higgs bundles of degree zero with corresponding parameters.
\end{conj}
From the above conjectural result, in combination with the results proved in Section \ref{section-sing} and Conjecture \ref{conj:cy2} below, one should be able to deduce the following:
\begin{conj} The moduli space of semistable parabolic Higgs bundles of
degree zero with fixed parameters is a symplectic singularity, admitting a symplectic resolution if and only if the corresponding character variety admits a symplectic resolution. 
\end{conj} 
A possible strategy to prove the results listed above would be to study the local structure of moduli spaces of (semistable) objects in Calabi--Yau categories. More details are provided in the next subsection. 

\subsection{Moduli spaces in 2-Calabi--Yau categories} \label{ss:cy}
As mentioned above, the problem of studying the singularities of a variety can be carried out by analysing the local structure of the variety itself around a point.

This method is powerful in certain cases, e.g., when one is able to prove some locally \'etale isomorphism between the variety of interest and another variety whose singularities are well-known: for example, this is carried out by Kaledin and Lehn in \cite{kaledin-lehn} and later by Arbarello and Sacc\`a in \cite{arb-sacca}, where they prove that, given a strictly semistable bundle in the moduli space of semistable sheaves on a K3 surface with a fixed non-generic polarisation, there exists an \'etale neighbourhood around that point that is isomorphic to an affine quiver variety, which depends on the point itself. Similar computations, which find their inspiration from \cite{kaledin-lehn}, were also performed by the Bellamy and the first author in \cite{bellamy-schedler} in the context of quiver and character varieties.

The fact that such a technique can be used and gives the same results in so many apparently different situations suggests that these are indeed particular cases of a series of theorems which should apply in much greater generality, namely in the context of 2-Calabi--Yau categories. In the work \cite{bocklandt-et-al} the authors carry out a detailed study of the deformation theory of representation spaces of 2-Calabi--Yau algebras and they show that among all semisimple representations, the ones that correspond to smooth points are precisely the simple ones. It is known that, when $Q$ is a Dynkin quiver and $q=1$, there is an isomorphism 
\[
\Lambda^1(Q)\cong \Pi^0(Q),
\] 
between the multiplicative preprojective algebra and the additive preprojective algebra, as shown in \cite[Corollary 1]{cb-monod}. Moreover, given that the additive preprojective algebras of Dynkin quivers with at least one arrow have infinite homological dimension, the above isomorphism suggests that, in general, multiplicative preprojective algebras are not 2-Calabi--Yau. On the other hand, a conjectural statement can be made for the case of non-Dynkin quivers. 
\begin{conj}\label{conj:cy2}
Let $Q$ be a connected non-Dynkin quiver and $q\in (\C^\times)^{Q_0}$. Then $\Lambda^q(Q)$ is a 2-Calabi--Yau algebra. 
\end{conj}
Assuming the above conjecture, then \cite[Theorem 6.3, 6.6]{bocklandt-et-al} implies the following for $\theta=0$: given a dimension vector $\alpha\in \Sigma_{q, \theta}$ and a point $x\in\mm_{q, \theta}(Q, \alpha)$, then formally locally around $x$, we have an isomorphism 
\[
\widehat{\mm_{q, \theta}(Q, \alpha)}_x\cong\widehat{\mathcal{M}_{0,0}(Q', \alpha')}_0
\] 
for some appropriate quiver $Q'$ and dimension vector $\alpha'$. If we can generalise this to arbitrary $\theta$ and prove the conjecture, there
would be many interesting consequences. First, it would make it possible to handle the $(2,2)$-case, where one may construct a symplectic resolution by first performing GIT (replacing $\theta$ by suitably generic $\theta'$) and then performing a blow-up of the singular locus. Moreover, this result would imply normality for $\mm_{q, \theta}(Q, \alpha)$, without any assumption on $\alpha$ and $\theta$.

Secondly, it would be interesting to extend such arguments and results to coarse moduli spaces of objects in 2-Calabi--Yau categories. This would indeed make it possible to reduce the study of singularities of a numerous class of moduli spaces to answering the following question:\tit{ does the moduli space parametrise objects of a 2-Calabi--Yau category?}

Knowing more about the singularities of coarse moduli spaces of objects in Calabi--Yau categories might allow one to prove a generalisation of the Isosingularity theorem in the non-compact case, provided that a suitable version of the 2-Calabi--Yau condition holds for the category of parabolic Higgs bundles.

\subsection{Character varieties and Higgs bundles for arbitrary groups}Another interesting problem would be to analyse whether the results of the present paper can extend to character varieties of (punctured) Riemann surfaces with representations in arbitrary groups and, via the non-abelian Hodge correspondence, to the moduli spaces of (parabolic) Higgs principal bundles. Since the results on punctured character varieties contained in the present paper are deduced based on the correspondence described in Section \ref{section-char}, which heavily relies on the fact that one considers representations and conjugacy classes inside $\GL(n, \C)$, it seems unlikely that the techniques used here could be applied in the more general setting of $G$-representations. On the other hand, a possible approach could be the one outlined in the previous subsection. To this end, the question to be answered would be the following: \tit{given an algebraic group $G$, are there cases, other than $G=GL_n$, where the category of morphisms $\rho:\pi_1(X) \rightarrow G$ is 2-Calabi--Yau?}

\bibliographystyle{amsalpha}
\bibliography{biblio}

\providecommand{\bysame}{\leavevmode\hbox to3em{\hrulefill}\thinspace}
\providecommand{\MR}{\relax\ifhmode\unskip\space\fi MR }
\providecommand{\MRhref}[2]{%
  \href{http://www.ams.org/mathscinet-getitem?mr=#1}{#2}
}
\providecommand{\href}[2]{#2}
\begin{thebibliography}{HLRV13b}

\bibitem[AMM98]{alekseev}
Anton Alekseev, Anton Malkin, and Eckhard Meinrenken, \emph{Lie group valued
  moment maps}, J. Differential Geom. \textbf{48} (1998), no.~3, 445--495.

\bibitem[Art69]{artin}
Michael~F Artin, \emph{Algebraic approximation of structures over complete
  local rings}, Publications Math{\'e}matiques de l'IH{\'E}S \textbf{36}
  (1969), 23--58.

\bibitem[AS15]{arb-sacca}
E.~{Arbarello} and G.~{Sacc{\`a}}, \emph{{Singularities of moduli spaces of
  sheaves on K3 surfaces and Nakajima quiver varieties}}, arXiv:1505.00759, May
  2015.

\bibitem[Bea00]{beauville}
Arnaud Beauville, \emph{Symplectic singularities}, Inventiones mathematicae
  \textbf{139} (2000), no.~3, 541--549.

\bibitem[BFN]{BFN-mdCb2}
A.~Braverman, M.~Finkelberg, and H.~Nakajima, \emph{Towards a mathematical
  definition of {C}oulomb branches of 3-dimensional {$\mathcal{N}=4$} gauge
  theories, {II}}, arXiv:1601.03586.

\bibitem[BGV16]{bocklandt-et-al}
Raf Bocklandt, Federica Galluzzi, and Francesco Vaccarino, \emph{The
  {N}ori-{H}ilbert scheme is not smooth for 2-{C}alabi-{Y}au algebras}, J.
  Noncommut. Geom. \textbf{10} (2016), no.~2, 745--774.

\bibitem[BK16]{kapranov}
Roman Bezrukavnikov and Mikhail Kapranov, \emph{Microlocal sheaves and quiver
  varieties}, Ann. Fac. Sci. Toulouse Math. (6) \textbf{25} (2016), no.~2-3,
  473--516.

\bibitem[BLPW16]{blpw14}
T.~Braden, A.~Licata, N.~Proudfoot, and B.~Webster, \emph{{Quantizations of
  conical symplectic resolutions II: category $\mathcal O$ and symplectic
  duality}}, Ast\`erique \textbf{384} (2016), 75, 179.

\bibitem[BMO11]{BMO-qcsr}
A.~Braverman, D.~Maulik, and A.~Okounkov, \emph{Quantum cohomology of the
  {S}pringer resolution}, Adv. Math. \textbf{227} (2011), no.~1, 421--458.
  \MR{2782198 (2012h:14133)}

\bibitem[Boa07]{boalch}
Philip Boalch, \emph{Quasi-{H}amiltonian geometry of meromorphic connections},
  Duke Math. J. \textbf{139} (2007), no.~2, 369--405.

\bibitem[BS16]{bellamy-schedler}
G.~{Bellamy} and T.~{Schedler}, \emph{{Symplectic resolutions of Quiver
  varieties and character varieties}}, arXiv:1602.00164, 2016.

\bibitem[CB01]{cb-geom}
William Crawley-Boevey, \emph{Geometry of the moment map for representations of
  quivers}, Compositio Math. \textbf{126} (2001), no.~3, 257--293.

\bibitem[CB02]{cb-deco}
\bysame, \emph{Decomposition of {M}arsden-{W}einstein reductions for
  representations of quivers}, Compositio Math. \textbf{130} (2002), no.~2,
  225--239.

\bibitem[CB03]{cb-norm}
\bysame, \emph{Normality of {M}arsden-{W}einstein reductions for
  representations of quivers}, Math. Ann. \textbf{325} (2003), no.~1, 55--79.

\bibitem[CB04]{cb-indec}
\bysame, \emph{Indecomposable parabolic bundles and the existence of matrices
  in prescribed conjugacy class closures with product equal to the identity},
  Publ. Math. Inst. Hautes \'Etudes Sci. (2004), no.~100, 171--207.

\bibitem[CB13]{cb-monod}
\bysame, \emph{Monodromy for systems of vector bundles and multiplicative
  preprojective algebras}, Bull. Lond. Math. Soc. \textbf{45} (2013), no.~2,
  309--317.

\bibitem[CBH19]{Crawley-Boevey2019}
William Crawley-Boevey and Andrew Hubery, \emph{A new approach to simple
  modules for preprojective algebras}, Algebras and Representation Theory
  (2019).

\bibitem[CBS06]{cb-shaw}
William Crawley-Boevey and Peter Shaw, \emph{Multiplicative preprojective
  algebras, middle convolution and the {D}eligne-{S}impson problem}, Adv. Math.
  \textbf{201} (2006), no.~1, 180--208.

\bibitem[CF17]{fairon}
Oleg Chalykh and Maxime Fairon, \emph{Multiplicative quiver varieties and
  generalised ruijsenaars–schneider models}, Journal of Geometry and Physics
  \textbf{121} (2017), 413 -- 437.

\bibitem[CHZ14]{CHZ-moHsCb}
S.~Cremonesi, A.~Hanany, and A.~Zaffaroni, \emph{Monopole operators and hilbert
  series of coulomb branches of 3d {$\mathcal{N} = 4$} gauge theories}, J. High
  Energy Phys. \textbf{005} (2014), arXiv:1309.2657.

\bibitem[Dre91]{Drezet}
J.-M. Drezet, \emph{Points non factoriels des vari\'et\'es de modules de
  faisceaux semi-stables sur une surface rationnelle}, J. Reine Angew. Math.
  \textbf{413} (1991), 99--126. \MR{1089799 (92d:14009)}

\bibitem[EL17]{etgu-lekili}
T.~{Etg{\"u}} and Y.~{Lekili}, \emph{{Fukaya categories of plumbings and
  multiplicative preprojective algebras}}, arXiv:1703.04515, March 2017.

\bibitem[Fle88]{flenner}
Hubert Flenner, \emph{Extendability of differential forms on nonisolated
  singularities}, Invent. Math. \textbf{94} (1988), no.~2, 317--326.

\bibitem[GS11]{GS-dmWge}
Q\"endrim~R. Gashi and Travis Schedler, \emph{On dominance and minuscule {W}eyl
  group elements}, J. Algebraic Combin. \textbf{33} (2011), no.~3, 383--399.
  \MR{2772538}

\bibitem[HLRV11]{h-l-rv-1}
Tam\'as Hausel, Emmanuel Letellier, and Fernando Rodriguez-Villegas,
  \emph{Arithmetic harmonic analysis on character and quiver varieties}, Duke
  Math. J. \textbf{160} (2011), no.~2, 323--400.

\bibitem[HLRV13a]{h-l-rv-2}
\bysame, \emph{Arithmetic harmonic analysis on character and quiver varieties
  {II}}, Adv. Math. \textbf{234} (2013), 85--128.

\bibitem[HLRV13b]{h-l-rv-3}
Tam\'{a}s Hausel, Emmanuel Letellier, and Fernando Rodriguez-Villegas,
  \emph{Positivity for {K}ac polynomials and {DT}-invariants of quivers}, Ann.
  of Math. (2) \textbf{177} (2013), no.~3, 1147--1168. \MR{3034296}

\bibitem[Ina13]{inaba}
Michi-Aki Inaba, \emph{Moduli of parabolic connections on curves and the
  {R}iemann-{H}ilbert correspondence}, J. Algebraic Geom. \textbf{22} (2013),
  no.~3, 407--480. \MR{3048542}

\bibitem[Jor14]{jordan}
David Jordan, \emph{Quantized multiplicative quiver varieties}, Adv. Math.
  \textbf{250} (2014), 420--466.

\bibitem[Kal06]{Kalss}
D.~Kaledin, \emph{Symplectic singularities from the {P}oisson point of view},
  J. Reine Angew. Math. \textbf{600} (2006), 135--156. \MR{MR2283801
  (2007j:32030)}

\bibitem[Kal09]{Kalnpa}
D.~B. Kaledin, \emph{Normalization of a {P}oisson algebra is {P}oisson}, Tr.
  Mat. Inst. Steklova \textbf{264} (2009), no.~Mnogomernaya Algebraicheskaya
  Geometriya, 77--80, arXiv:math/0310173. \MR{2590837 (2011b:17045)}

\bibitem[Kat96]{katz}
Nicholas~M. Katz, \emph{Rigid local systems}, Annals of Mathematics Studies,
  vol. 139, Princeton University Press, Princeton, NJ, 1996.

\bibitem[Kin94]{king}
A.~D. King, \emph{Moduli of representations of finite-dimensional algebras},
  Quart. J. Math. Oxford Ser. (2) \textbf{45} (1994), no.~180, 515--530.

\bibitem[KL07]{kaledin-lehn}
D.~Kaledin and M.~Lehn, \emph{Local structure of hyperk\"ahler singularities in
  {O}'{G}rady's examples}, Mosc. Math. J. \textbf{7} (2007), no.~4, 653--672,
  766--767.

\bibitem[KLS06]{kaledin-lehn-sorger}
D.~Kaledin, M.~Lehn, and Ch. Sorger, \emph{Singular symplectic moduli spaces},
  Invent. Math. \textbf{164} (2006), no.~3, 591--614.

\bibitem[LM10]{logares-martens}
Marina Logares and Johan Martens, \emph{Moduli of parabolic {H}iggs bundles and
  {A}tiyah algebroids}, J. Reine Angew. Math. \textbf{649} (2010), 89--116.

\bibitem[Maf02]{Maffei}
A.~Maffei, \emph{A remark on quiver varieties and {W}eyl groups}, Ann. Sc.
  Norm. Super. Pisa Cl. Sci. (5) \textbf{1} (2002), no.~3, 649--686.
  \MR{1990675}

\bibitem[MFK02]{mumford}
D.~Mumford, J.~Fogarty, and F.~Kirwan, \emph{Geometric invariant theory},
  Ergebnisse der Mathematik und ihrer Grenzgebiete. 2. Folge, Springer Berlin
  Heidelberg, 2002.

\bibitem[MGW19]{mcbreen-gammage-webster}
M.~{McBreen}, B.~{Gammage}, and B.~{Webster}, \emph{{Homological Mirror
  Symmetry for Hypertoric Varieties II}}, arXiv:1903.07928, March 2019.

\bibitem[MO19]{maulik-okounkov}
Davesh Maulik and Andrei Okounkov, \emph{Quantum groups and quantum
  cohomology}, Ast\'{e}risque (2019), no.~408, ix+209. \MR{3951025}

\bibitem[Moz90]{Moz-rfgWg}
S.~Mozes, \emph{Reflection processes on graphs and {W}eyl groups}, J. Combin.
  Theory Ser. A \textbf{53} (1990), no.~1, 128--142.

\bibitem[MW18]{mcbreen-webster}
M.~{McBreen} and B.~{Webster}, \emph{{Homological Mirror Symmetry for
  Hypertoric Varieties I}}, arXiv:1804.10646, April 2018.

\bibitem[Nak94]{nakajima}
Hiraku Nakajima, \emph{Instantons on {ALE} spaces, quiver varieties, and
  {K}ac-{M}oody algebras}, Duke Math. J. \textbf{76} (1994), no.~2, 365--416.

\bibitem[Nak16]{Nak-mdCb1}
H.~Nakajima, \emph{Towards a mathematical definition of {C}oulomb branches of
  3-dimensional {$\mathcal{N}=4$} gauge theories, {I}}, Adv. Theor. Math. Phys.
  \textbf{20} (2016), no.~3, 595--669. \MR{3565863}

\bibitem[Nam01]{naminote}
Yoshinori Namikawa, \emph{A note on symplectic singularities}, ArXiv
  Mathematics e-prints (2001).

\bibitem[O'G99]{OGr-K3}
Kieran~G. O'Grady, \emph{Desingularized moduli spaces of sheaves on a {$K3$}},
  J. Reine Angew. Math. \textbf{512} (1999), 49--117. \MR{1703077}

\bibitem[{Sim}90]{simpson-harm}
Carlos~T. {Simpson}, \emph{Harmonic bundles on noncompact curves}, J. Amer.
  Math. Soc. \textbf{3} (1990), no.~3, 713--770.

\bibitem[Sim94a]{simpson1994-1}
Carlos~T. Simpson, \emph{Moduli of representations of the fundamental group of
  a smooth projective variety i}, Publications Mathématiques de l'IHÉS
  \textbf{79} (1994), 47--129 (eng).

\bibitem[Sim94b]{simpson1994-2}
\bysame, \emph{Moduli of representations of the fundamental group of a smooth
  projective variety ii}, Publications Mathématiques de l'IHÉS \textbf{80}
  (1994), 5--79 (eng).

\bibitem[Su06]{Su-fmmrq}
X.~Su, \emph{Flatness for the moment map for representations of quivers}, J.
  Algebra \textbf{298} (2006), no.~1, 105--119. \MR{2214399 (2007c:16029)}

\bibitem[Tho61]{Tho-csglg}
R.~C. Thompson, \emph{Commutators in the special and general linear groups},
  Trans. Amer. Math. Soc. \textbf{101} (1961), 16--33. \MR{0130917}

\bibitem[{Tir}17]{tirelli}
A.~{Tirelli}, \emph{{Symplectic resolutions for Higgs moduli spaces}},
  arXiv:1701.07468, January 2017.

\bibitem[{Van}08a]{van-den-bergh-1}
Michel {Van den Bergh}, \emph{Double {P}oisson algebras}, Trans. Amer. Math.
  Soc. \textbf{360} (2008), no.~11, 5711--5769.

\bibitem[{Van}08b]{van-den-bergh-2}
\bysame, \emph{Non-commutative quasi-{H}amiltonian spaces}, Poisson geometry in
  mathematics and physics, Contemp. Math., vol. 450, Amer. Math. Soc.,
  Providence, RI, 2008, pp.~273--299.

\bibitem[{Yam}08]{yamakawa}
Daisuke {Yamakawa}, \emph{Geometry of multiplicative preprojective algebra},
  Int. Math. Res. Pap. IMRP (2008), 1--77.

\bibitem[Yeu]{Yeu-wCYsmr}
Wai-Kit Yeung, \emph{Weak {C}alabi--{Y}au structures and moduli of
  representations}, arXiv:1802.05398.

\end{thebibliography}
\end{document}